\documentclass[11pt]{amsart}   	
\usepackage{geometry}             
\usepackage[dvipsnames]{xcolor}   		             		
\usepackage{graphicx}				
\usepackage{amssymb,mathrsfs}
\usepackage{amssymb,amsthm,amsmath,stmaryrd}
\usepackage{tikz}
\usepackage{tikz-cd}
\usepackage{accents,upgreek,enumerate}
\usepackage[headings]{fullpage}
\usepackage{bm}
\usepackage[all]{xy}
\usepackage{caption}

\usetikzlibrary{calc}
\usetikzlibrary{fadings}
\usetikzlibrary{decorations.pathmorphing}
\usetikzlibrary{decorations.pathreplacing}

\newcounter{marginnote}
\DeclareMathAlphabet{\mathpzc}{OT1}{pzc}{m}{it}

\usepackage[backref]{hyperref}
\hypersetup{
colorlinks   = true,          %Colors links instead of ugly boxes
urlcolor     = violet,          %Color for external hyperlinks
linkcolor    = Purple,          %Color of internal links
citecolor   = RubineRed             %Color of citations
}

\newtheorem{theorem}{Theorem}[subsection]
\newtheorem{corollary}[theorem]{Corollary}
\newtheorem{lemma}[theorem]{Lemma}
\newtheorem{proposition}[theorem]{Proposition}

\newtheorem{quasi-theorem}[theorem]{Quasi-Theorem}

\theoremstyle{definition}
\newtheorem{definition}[theorem]{Definition}

\newtheorem{construction}[theorem]{Construction}
\newtheorem{example}[theorem]{Example}
\newtheorem{blank remark}[theorem]{}

\theoremstyle{remark}
\newtheorem{rem1}[theorem]{Remark}

\newtheorem{not1}[theorem]{Notation}

\newcommand{\PP}{\mathbb{P}}         
		
\newcommand{\RR} {{\mathbb R}}		
\newcommand{\ZZ} {{\mathbf Z}}		

\def\setminus{\smallsetminus}

\DeclareMathOperator{\conv}{conv}

\DeclareMathOperator{\Aut}{Aut}
\DeclareMathOperator{\Area}{Area}

\DeclareMathOperator{\val}{val}
\DeclareMathOperator{\floor}{floor}

\newcommand{\cal}{\mathcal}

\def\cM{{\cal M}}

\newcommand{\Mbar}{\overline{\cM}\vphantom{\cM}}

\def\trop{\mathrm{trop}}
\def\floor{\mathrm{floor}}
\def\mult{\mathrm{mult}}

\def\di{\mathrm{div}}

\newcommand{\rect}{\mathsf{rect}}

\newcommand{\sU}{\mathsf{U}}
\newcommand{\sD}{\mathsf{D}}

\newcommand{\sR}{\mathsf{R}}
\newcommand{\brak}[2]{\langle #1, #2 \rangle}
\newcommand{\dist}{\operatorname{dist}}
\newcommand{\argmin}{\operatorname{argmin}}

\makeatletter
\def\blfootnote{\xdef\@thefnmark{}\@footnotetext}
\makeatother

\title{Counting tropical curves in $\PP^1\times\PP^1$: computation {\it \&} polynomiality properties}

\date{}

\author[Corey]{Daniel Corey}
\address{Daniel Corey \\ Technische Universit\"at Berlin\\ Institut f\"ur Mathematik}
\email{\href{mailto:corey@math.tu-berlin.de}{corey@math.tu-berlin.de}}

\author[Markwig]{Hannah Markwig}
\address{Hannah Markwig \\ Fachbereich Mathematik\\
Universit\"at T\"ubingen}
\email{\href{mailto:hannah@math.uni-tuebingen.de}{hannah@math.uni-tuebingen.de}}

\author[Ranganathan]{Dhruv Ranganathan}
\address{Dhruv Ranganathan \\ Department of Mathematics\\
University of Cambridge}
\email{\href{mailto:dr508@cam.ac.uk}{dr508@cam.ac.uk}}

\begin{document}

\begin{abstract}
Counts of curves in $\PP^1\times \PP^1$ with fixed contact order with the toric boundary and satisfying point conditions can be determined with tropical methods \cite{Mi03}. If we require that our curves intersect the zero- and infinity-section only in points of contact order $1$, but allow arbitrary contact order for the zero- and infinity-fiber, the corresponding numbers reveal beautiful structural properties such as piecewise polynomiality, similar to the case of double Hurwitz numbers counting covers of $\PP^1$ with special ramification profiles over zero and infinity \cite{AB17}. This result was obtained using the floor diagram method to count tropical curves.

Here, we expand the tropical tools to determine counts of curves in $\PP^1\times\PP^1$. We provide a computational tool (building on \textsc{Polymake} \cite{polymake}) that determines such numbers of tropical curves for any genus and any contact orders via a straightforward generalization of Mikhalkin's lattice path algorithm. The tool can also be used for other toric surfaces. To enable efficient computations also by hand, we introduce a new counting tool (for the case of rational curves with transverse contacts with the infinity section) which can be seen as a combination of the floor diagram and the lattice path approach: subfloor diagrams. We use both our computational tool and the subfloor diagrams for experiments revealing structural properties of these counts.

We obtain first results on the (piecewise) polynomial structure of counts of rational curves in $\PP^1\times\PP^1$ with arbitrary contact orders on the zero- and infinity-fiber and restricted choices for the contact orders on the zero- and infinity-section.

\noindent \textbf{MSC 2020 Classification}: 14T90 (primary), 14N10, 14M25, 05E14 (secondary).

\noindent \textbf{Keywords}: enumerative geometry, floor diagrams, lattice paths, tropical curves. 
\end{abstract}

\maketitle

\vspace{-0.2in}

\setcounter{tocdepth}{1}
%\tableofcontents

\section{Introduction}

The enumerative geometry of the pair $(\mathbb P^1|0,\infty)$ concerns counts of covers of $\mathbb P^1$ by algebraic curves of a fixed genus with prescribed ramification conditions at $0$ and $\infty$. It is a highly structured theory, and is closely related to intersection theory on the moduli space of curves and the double ramification cycle~\cite{BSSZ,CMR22,GJV05,GV05,JPPZ}. The simplest invariants in this theory are {\it double Hurwitz numbers}. Let $\mathbf x$ be a vector of length $n$ with vanishing sum. The invariant $\mathsf{H}_g(\mathbf x)$ concerns the number of maps from $n$ pointed genus $g$ curves to $\mathbb P^1$ such that the zero/pole order at the marked point $p_i$ is the $i^{\mathrm{th}}$ entry in $\mathbf x$. The Riemann--Hurwitz formula determines the total amount of additional ramification. If all additional ramification is forced to be simple, and if the locations of the images of the ramification points are fixed and generic in $\mathbb P^1$, the automorphism--weighted count of covers is $\mathsf{H}_g(\mathbf x)$. 

The first place where the rich structure of double Hurwitz numbers  can be seen is in the {\it polynomiality properties} of the invariants, via a beautiful result of Goulden, Jackson, and Vakil. They prove that as the vector $\mathbf x$ varies in the space of vectors with vanishing sum, the rational number $\mathsf{H}_g(\mathbf x)$ is naturally a {\it polynomial function} on the chambers of a hyperplane arrangement~\cite{GJV05}. The polynomiality properties have been studied in detail through tropical geometry~\cite{CJM1,CJM2} and intersection theory~\cite{CM14,SSV}. 

Among the numerous interesting results on double Hurwitz numbers, two are especially relevant to us here.  First, the double Hurwitz numbers can be practically calculated and studied through tropical geometry~\cite{CJM1}. Second, the piecewise polynomiality of the numbers reflects the polynomiality of the double ramification cycle~\cite{BSSZ,CM14}. 

\subsection{Context: polynomiality in higher dimensions} The purpose of this article is explore the geometry of the pair $(\mathbb P^1\times\mathbb P^1|D)$ where $D$ is the toric boundary divisor. The tropical correspondence theorem of Mikhalkin again gives a practical way to calculate and study these numbers~\cite{Mi03}. The higher dimensional replacement for the double ramification cycle, sometimes referred to as the {\it double double}, is an object of significant recent and ongoing interest~\cite{Herr,HMPPS,HPS19,HS21,MR21}. Although it has received attention, its basic properties have not been fully understood. 

In the case of $\mathbb P^1\times\mathbb P^1$, the role of the invariants is played by logarithmic Gromov--Witten invariants, which are modeled on counting curves of a fixed genus $g$ with fixed tangency order with the $0$ and $\infty$ fibers and sections of this surface\footnote{Of course, the terminology of section and fiber is interchangeable in this geometry.}. By additionally imposing point conditions, we obtain invariants. Our guiding question is the behaviour of these invariants as the tangency orders vary. Ardila and Brugall\'e proved that if the contact orders along the fiber directions are generic, then the stationary logarithmic Gromov--Witten theory is piecewise polynomial in the entries of the remaining contact data~\cite{AB17}. 

The question we raise and explore in this paper is what happens for non-generic contact orders. The higher double ramification cycle fails to be a polynomial class, and we view the polynomiality properties of the counts as a way in which to explore this failure. We explain the link to logarithmic intersection theory in \S\ref{sec: log-int}.

\subsection{Results and techniques} We formulate the counts precisely. The class of a curve in $\PP^1\times\PP^1$ can be written as $aB+bF$, where $B$ is a section and $F$ is a fiber. We say it has bidegree $(a,b)$. 

Our goal is to study how the invariants behave as some of the discrete parameters are allowed to move; we work in genus $0$. Precisely, fix $a$, the number $n$ of contact points with the $0$-fiber, and the number $m$ of contact points with the $\infty$-fiber. View the entries of the partition $\nu_1=(\nu_{11},\ldots,\nu_{1n})$ of contact orders with the $0$-fiber and the entries of the partition $\nu_2=(\nu_{21},\ldots,\nu_{2m})$ of contact orders with the $\infty$-fiber as variables. We combine the two partitions $\nu_1$ and $\nu_2$ into a pair $\nu=(\nu_1,\nu_2)$ giving the contact orders with the boundary fibers. The contact orders of the boundary sections are similarly given by a pair $\mu$. Logarithmic Gromov--Witten theory of $\mathbb P^1\times\mathbb P^1$ associates to these data an invariant $N(\mu,\nu)$, see~Definition \ref{def-logGWI}. If $\mu=((1^a),(1^a))$, the work of Ardila--Brugall\'e mentioned above shows that the invariants exhibit polynomiality. The results of this paper take first steps to generalize to non-trivial contact orders $\mu$ with the boundary sections.

\noindent
\textbf{New computational and combinatorial tools.} We provide two new tools to study these curve counts. Fundamentally, we build on Mikhalkin's correspondence theorem~\cite{Mi03} and the computational methods introduced therein. We immediately transfer the count into a tropical one, and focus on the problem of enumerating tropical curves. First, we provide a computational tool (building on \textsc{Polymake} \cite{polymake}) that computes such numbers of tropical curves for any genus and any contact orders via a straight-forward generalization of Mikhalkin's lattice path algorithm~\cite{Mi03}. The tool can also be used for other toric surfaces. 

Second, we introduce a new combinatorial tool (for the case of rational curves with transverse contacts with the infinity section) which generalizes the two best known tools for enumerating tropical curves: (i) the floor diagram approach of Brugall\'e--Mikhalkin, using which the Ardila--Brugall\'e result was proved~\cite{AB17,BM,BM08} and the lattice path algorithm in Mikhalkin's original paper.

We use both our computational tool and the subfloor diagrams for experiments revealing structural properties of these counts.

\noindent
\textbf{New results on polynomiality.} By using the above tools, we obtain first results on the piecewise polynomial structure of counts of rational curves in $\PP^1\times\PP^1$ with arbitrary contact orders on the zero- and infinity-fiber and restricted choices for the contact orders on the zero- and infinity-section. We discover two new results on polynomiality:
\begin{enumerate}[(i)]
\item In Theorem~\ref{maintheorem1}, we find a new polynomiality property for invariants with tangency profiles $(c), (1^{c}), (x_1,\ldots,x_n), (d)$, i.e., two opposing sides have transverse and maximal tangency respectively, and the next two opposing sides have maximal and variable tangency respectively. The invariants are polynomial for generic choices of the $x_i$'s.
\item In Theorem \ref{maintheorem2}, we find a new polynomiality property when invariants have transverse contact orders with the $\infty$-section, tangency at most $2$ with the $0$-section, and variable tangency along the fibers, The invariants are polynomial in the entries of the variable tangencies. 
\end{enumerate}

The results together suggest that aspects of the polynomiality survive in the $2$-dimensional case, and we hope this motivates and paves the way for further study. Throughout the paper, the \textsc{Polymake} computations and the new subfloor diagram techniques play a key role in proving these results. We are able to provide numerous examples throughout to showcase the utility of the computational results. 

\subsection{Broader context and nearby mathematics} The correspondence theorems of Mikhalkin and Nishinou--Siebert allow one to transfer enumerative problems of curves in toric varieties to combinatorics, i.e., the enumeration of tropical curves~\cite{Mi03,NS06}. The enumeration of tropical curves is itself a very difficult problem, even when all tangency orders are generic. Mikhalkin gave a beautiful algorithm to solve this problem in dimension $2$. The resulting algorithm is based on a {\it completely discrete} data structure, i.e., the {\it lattice path}. Brugall\'e and Mikhalkin then noticed that in specialized geometries, including all Hirzebruch surfaces, one can specialize the tropical constraints for the enumerative problem to distill the tropical curve to another {completely discrete}, rather than piecewise linear, data structure. A feature of our approach here is that we are, in a sense, able to combine the strengths of the floor diagram techniques and the lattice path algorithms. 

The floor diagram techniques have revealed a number of important properties about enumerative invariants and generalize well to $\mathbb P^1\times\mathbb P^1$ where tangency with the sections is transverse, as well as Hirzebruch surfaces~\cite{AB17,BGM,BG16,Blo21,B15,CJMR17,FM09}. For another interesting approach to enumerating tropical curves, see~\cite{MR19}.

\subsection{Outline} This paper is organized as follows. In \S\ref{sec-countcurves}, we introduce our main tropical counting problem, concerning curves in $\PP^1\times\PP^1$ satisfying boundary conditions and the analogous logarithmic Gromov--Witten invariants. In \S\ref{sec: log-int} we link these invariants to logarithmic intersection theory, for context. The correspondence theorem establishing the equality of these logarithmic Gromov--Witten invariants and their tropical counterparts is proved in \cite{MR16,R15b}. It opens the possibility to obtain structural results like piecewise polynomiality for logarithmic Gromov--Witten invariants by means of tropical geometry, which is our line of arguments.

In \S\ref{sec-path} we focus on lattice paths. We recall in \S\ref{subsec-latticepath} Mikhalkin's definition of lattice paths and their multiplicity and relation to tropical curve counts. In \S\ref{subsec-polymake}, we present our computational tool, based on \textsc{polymake} \cite{polymake}, that counts lattice paths and determines their multiplicity for any toric surface, any genus, and any boundary conditions. In \S\ref{subsec-structurepath}, we present structural properties of certain lattice path counts. In particular, we present our first main result, Theorem \ref{maintheorem1}, a polynomiality statement for a count where we require full contact order for one fiber and one section. This is similar to the polynomiality of one-part Hurwitz numbers mentioned above. The rest of the section is devoted to the proof of Theorem \ref{maintheorem1} and some examples. 
The proof has three steps: first, we characterize the paths that contribute to our count with a nonzero multiplicity. This is done in \S\ref{subsec-charpaths}.
The final step is to prove that for each such path, the multiplicity is  polynomial. This is done in \S\ref{subsec-multpp}. But before we can proceed to this final step, we have to develop a technical trick that allows us to compute multiplicities of paths more easily: we show that one can in fact pick any right turn in a lattice path and perform the recursion for the negative multiplicity, it does not necessarily have to be the first right turn as in Mikhalkin's original description \cite{Mi03}. This is done in \S\ref{subsec-generalizedmikpos}. We end our section on lattice paths by providing examples in \S\ref{subsec-pathex} on the polynomial structure.

The final section, \S\ref{sec-subfloor}, is devoted to subfloor diagrams. This is our new tool, a combination of a lattice path count with a floor decomposition technique, which enables us to compute small examples by hand and also admits structural results on piecewise polynomiality. Our main result in this section is Theorem \ref{maintheorem2} on piecewise polynomiality of counts in $\PP^1\times\PP^1$ with certain boundary behaviour. In \S\ref{subsec-countsubfloor}, we define subfloor diagrams and their counts with multiplicity. In \S\ref{subsec-tropcurvessubfloor} we prove that the count of subfloor diagrams is equal to the corresponding count of tropical curves, and, accordingly, to a logarithmic Gromov--Witten invariant. In \S\ref{subsec-ppsubfloor} finally, we prove our main result of this section, Theorem \ref{maintheorem2} on piecewise polynomiality of counts in $\PP^1\times\PP^1$ with certain boundary behaviour. We end the section on subfloor diagrams by providing examples in \S\ref{subsec-subfloorex} on the piecewise polynomial structure.

\subsection{Acknowledgements}
The authors thank Michael Joswig, Lars Kastner, and Ajith Urundolil Kumaran for useful discussions. We gratefully acknowledge support by the Deutsche Forschungsgemeinschaft (DFG, German Research Foundation), TRR 195. D.C. is supported by the SFB-TRR project   \textit{Symbolic Tools in Mathematics and their Application} project-ID 286237555. D.R. is supported by EPSRC New Investigator Award EP\slash V051830\slash 1.

\section{Counts of tropical and logarithmic curves in $\PP^1\times\PP^1$}\label{sec-countcurves}

\subsection{Logarithmic Gromov--Witten invariants of $\PP^1\times\PP^1$}\label{subsec-logGWI}
Logarithmic stable maps and logarithmic Gromov--Witten invariants were developed in ~\cite{AC11,Che10,GS13} and have been studied via tropical geometry in a number of different papers~\cite{CJMR17,Gro15,MR16,MR19,R15b}. We recall the setup in our context.

Let $\mu_1,\mu_2$ be partitions of $d_1$, and $\nu_1,\nu_2$ partitions of $d_2$. We record this data in a tuple $(\mu,\nu)=(\mu_1,\mu_2,\nu_1,\nu_2)$. There is a moduli stack with logarithmic structure
\begin{equation*}
\overline{M}^{\mathsf{log}}_{g,n}(\PP^1\times\PP^1,(\mu,\nu)),
\end{equation*}
such that a map from a logarithmic scheme $S$ to the moduli space is equivalent to a diagram 
\begin{equation*}
\begin{tikzcd}
\mathscr C \arrow{d}\arrow{r}{f} & \mathbb \PP^1\times \PP^1 \\
S, 
\end{tikzcd}
\end{equation*}
where $\mathscr C$ is a family of connected marked rational nodal  curves and $f$ is a map of logarithmic schemes, whose underlying map is stable in the usual sense, and where the contact orders with the toric boundary are given by $\mu$ and $\nu$. The $n$ markings refer to the ``interior markings'', i.e., the markings that have contact order equal to $0$. 

This moduli space is a virtually smooth Deligne-Mumford stack and it carries a virtual fundamental class denoted by $[\overline{M}^{\mathsf{log}}_{g,n}(\PP^1\times\PP^1,(\mu,\nu))]^\mathsf{vir}$ in degree 
\begin{equation*}
\mathsf{vdim} = g+\ell(\mu_1)+\ell(\mu_2)+\ell(\nu_1)+\ell(\nu_2)-1 +n.
\end{equation*}

Note that in the case when the genus $g$ is equal to $0$, the virtual class coincides with the usual fundamental class, as the moduli space of logarithmic stable maps is logarithmic smooth and irreducible of the expected dimension~\cite[\S3]{R15b}. 

For each of the  $n$ marked points, there are evaluation morphisms
\begin{equation*}
\mathsf{ev}_i: \overline{M}^{\mathsf{log}}_{g,n}(\PP^1\times\PP^1,(\mu,\nu))\to \PP^1\times\PP^1.
\end{equation*}

%For each of the $n$ marks there is a cotangent line bundle, whose first Chern class is denoted $\psi_i$.

\begin{definition}\label{def-logGWI}
Let $\ell(\mu_1)+\ell(\mu_2)+\ell(\nu_1)+\ell(\nu_2)-1 =n$. The \emph{logarithmic Gromov--Witten} invariant is defined as the following intersection number on $\overline{M}^{\mathsf{log}}_{0,n}(\PP^1\times\PP^1, (\mu,\nu))$:
\begin{equation}\label{eq-lgw}
N_g(\mu,\nu):=
\int_{[\overline{M}^{\mathsf{log}}_{g,n}(\PP^1\times\PP^1,(\mu,\nu))]^\mathsf{vir}} \prod_{j=1}^n \mathsf{ev}_j^\ast([pt]).
\end{equation}
In the case that $g = 0$, we use the notation $N(\mu,\nu)$. 
\end{definition}

We note that in this paper, we will almost always set $g = 0$. 

\subsection{Logarithmic intersection perspective}\label{sec: log-int} We briefly explain the connection to logarithmic intersection theory. Let $(X,D)$ be a pair consisting of a smooth projective variety (or Deligne--Mumford stack) $X$ and a simple normal crossings divisor $D$. The {\it logarithmic intersection theory} of $(X,D)$ is the study of the intersection theory on all blowups of $X$ along boundary strata, considered simultaneously. For a detailed introduction to the subject, see~\cite{HS21,MPS21,MR21}. 

More precisely, a {\it simple blowup} of $(X,D)$ is a blowup of $X$ at a stratum of $D$. The simple blowup has a natural normal crossings divisor, given by the reduced preimage of $D$. A {\it simple blowup sequence} is a sequence of such simple blowups with respect to these divisors. If $X'\to X$ is a simple blowup sequence, there is a natural pullback map in Chow
\[
\mathsf{CH}^\star(X)\hookrightarrow\mathsf{CH}^\star(X').
\]
The {\it logarithmic Chow ring} of $(X,D)$ is then defined to be
\[
\mathsf{logCH}^\star(X,D) = \varinjlim_{X'\to X} \mathsf{CH}^\star(X'),
\]
with transition maps as described above. A basic source of classes in this ring are {\it piecewise polynomial functions}\footnote{The terminology of piecewise polynomial here should not be confused with the piecewise polynomial nature of the invariants: the two are not related!} on (all subdivisions) of the generalized cone complex associated to $(X,D)$, in the sense of~\cite{ACP,CCUW}. 

Fixing a genus $g$ and partition data $\mu$ and $\nu$, there is a distinguished element in the logarithmic Chow ring of $(\Mbar_{g,m},\partial \Mbar_{g,m})$, where $m$ is $n$ plus the total length of all the partitions: the (rank $2$) higher double ramification cycle associated to $g$ and $(\mu,\nu)$, see~\cite{HMPPS,HS21,MR21}. It generalizes the well-known double ramification cycle from~\cite{JPPZ} in two ways: first, it is a version for $\mathbb P^1\times\mathbb P^1$ rather than just $\mathbb P^1$, and second, it has been lifted to the logarithmic Chow theory. In the case when $g = 0$, the class is just the fundamental class of $\Mbar_{0,n}$. 

With this preliminary work, we can now explain the motivation for this work from logarithmic intersection theory. 

\noindent
{\bf Rank 1 story.} The rank 1, or ``usual'' double ramification cycle can be related to double Hurwitz numbers in two ways.  The double Hurwitz numbers are equal to the intersection of the logarithmic double ramification cycle with piecewise polynomial classes by~\cite{CMR22}. The polynomiality of the double ramification cycle in ordinary Chow of $\Mbar_{g,n}$ leads to the piecewise polynomiality of the double Hurwitz numbers~\cite{BSSZ,CM14}. 

\noindent
{\bf Rank 2 story.} The rank 2 higher double ramification cycle is known not to be polynomial. It is shown in~\cite{RUK22} that the numbers $N_g(\mu,\nu)$ are intersections of the logarithmic double ramification cycles against piecewise polynomial functions. 

The results in this paper therefore suggest that the polynomiality of the higher double ramification cycle may not be lost forever. We now focus on genus $0$ throughout the rest of the paper, since the story is already mysterious in this case.

\subsection{Tropical analogues of logarithmic Gromov--Witten invariants}\label{subsec-tropGWI}

To investigate the structure of the numbers $N(\mu,\nu)$, we use tropical geometry. We introduce basic notions of tropical Gromov--Witten theory and use the correspondence theorem which yields an equality of the numbers $N(\mu,\nu)$ introduced above with their tropical counterparts. For more details and other cases of tropical Gromov--Witten invariants and correspondence theorems, see, e.g.,  \cite{CJMR17, Gro10, MR16, MR19, MR08, Mi03, R15b}.

A rational \emph{(abstract) tropical curve} is a connected metric tree $\Gamma$ with unbounded rays called ends. Locally around a point $p$, $\Gamma$ is homeomorphic to a star with $r$ half-rays. 
The number $r$ is called the \emph{valence} of the point $p$ and denoted by $\val(p)$.
We require that there are only finitely many points with $\val(p)\neq 2$. 
A finite set of points containing  (but not necessarily equal to the set of) all points of valence larger than $2$ may be chosen; its elements  are called  \emph{vertices}. 
By abuse of notation, the underlying graph with this vertex set is also denoted by $\Gamma$. 
Correspondingly, we can speak about \emph{edges} and \emph{flags} of $\Gamma$. A flag is a tuple $(V,e)$ of a vertex $V$ and an edge $e$ with $V\in \partial e$. It can be thought of as an element in the tangent space of $\Gamma$ at $V$, i.e., as a germ of an edge leaving $V$, or as a half-edge (the half of $e$ that is attached to $V$). Edges which are not ends have a finite length and are called \emph{bounded edges}.

A rational \emph{marked tropical curve} is a tropical curve such that some of its ends are labeled. An isomorphism of a tropical curve is a homeomorphism respecting the metric and the markings of ends.  The \emph{combinatorial type} of a tropical curve is obtained by dropping the information on the metric.  A \emph{Newton fan} is a multiset $D=\{v_1,\ldots,v_k\}$ of vectors $v_i\in \ZZ^2$ satisfying $\sum_{i=1}^k v_i=0$. 

\begin{definition}
A \emph{tropical stable map to $\RR^2$} is a tuple $(\Gamma,f)$ where $\Gamma$ is a marked abstract tropical curve and $f:\Gamma\to \Sigma$ is a piecewise integer-affine map of polyhedral complexes satisfying:
\begin{itemize}
\item On each edge $e$ of $\Gamma$, $f$ is of the form 
\begin{equation*}
t\mapsto a+t\cdot v \mbox{ with } v\in \ZZ^2,
\end{equation*}
where we parametrize $e$ as an interval of size the length $l(e)$ of $e$. The vector $v$, called the \emph{direction}, arising in this equation is defined up to sign, depending on the starting vertex of the parametrization of the edge. We will sometimes speak of the direction of a flag $v(V,e)$. If $e$ is an end we use the notation $v(e)$ for the direction of its unique flag.
\item The \emph{balancing condition} holds at every vertex, i.e.,
\begin{equation*}
\sum_{e \in \partial V} v(V,e)=0.
\end{equation*}
\item The \emph{stability condition} holds, i.e., for every $2$-valent vertex $v$ of $\Gamma$, the star of $v$ is not contained in the relative interior of any single cone of $\Sigma$.
\end{itemize}
\end{definition}
For an edge with direction $v=(v_1,v_2) \in \ZZ^2$, we call $w=\gcd(v_1,v_2)$ the \emph{expansion factor} and $\frac{1}{w}\cdot v$ the \emph{primitive direction} of $e$.

An isomorphism of tropical stable maps is an isomorphism of the underlying tropical curves respecting the map. The \emph{degree} of a tropical stable map is the Newton fan given as the multiset of directions of its ends.

We say that a tropical stable map is \emph{to $\PP^1\times\PP^1$ of degree $(\mu,\nu)$} if its Newton fan equals
\begin{align*}
\Big\{ &\mu_{11}\cdot \binom{-1}{0},\ldots, \mu_{1\ell(\mu_1)}\cdot \binom{-1}{0} ,\mu_{21}\cdot \binom{1}{0},\ldots, \mu_{2\ell(\mu_2)}\cdot \binom{1}{0} , \\ &\nu_{11}\cdot \binom{0}{-1},\ldots, \nu_{1\ell(\nu_1)}\cdot \binom{0}{-1} ,\nu_{21}\cdot \binom{0}{1},\ldots, \nu_{2\ell(\nu_2)}\cdot \binom{0}{1} \Big\}.
\end{align*}

The \emph{combinatorial type} of a tropical stable map is the data obtained when dropping the metric of the underlying graph. More explicitly, it consists of the data of a finite graph $\Gamma$, and  
for each edge $e$ of $\Gamma$, the expansion factor and primitive direction of $e$.

The image of a tropical stable map is a tropical plane curve. Often, the image already determines the map. For that reason, by abuse of notation, we also call tropical stable maps tropical curves.

\begin{definition}\label{def-tropGWI}

Fix $n$ point conditions $p_1,\ldots,p_n\in \RR^2$ in tropical general position. We say that a marked rational tropical stable map $(\Gamma,f)$ \emph{meets the point conditions} if the marked end $j$ is contracted to $p_j\in \RR^2$ for $j=1, \ldots, n$.
The \emph{tropical Gromov--Witten invariant} 
\begin{equation*}
N^{\trop}(\mu,\nu)
\end{equation*} is the weighted number of marked rational tropical stable maps $(\Gamma,f)$ to $\PP^1\times\PP^1$ of degree $(\mu,\nu)$  meeting $n$ point conditions, where $n=\ell(\mu_1)+\ell(\mu_2)+\ell(\nu_1)+\ell(\nu_2)-1$.
Each such tropical stable map is counted with \emph{multiplicity} $\frac{1}{\mathrm{Aut}(f)}m_{(\Gamma,f)}$, where 
$m_{(\Gamma,f)}=\prod_V \mult(V)$ is the product of all vertex multiplicities and
\begin{equation*}
\mult(V)= \begin{cases}1 & \mbox{ if precisely one marked end $j$ is adjacent to $V$ }\\ |\det(v_{e_1},v_{e_2})| & \mbox{if no marked end is adjacent to $V$.} \end{cases}
\end{equation*}
Here, $v_{e_1}$ and $v_{e_2}$ are the direction vectors of two edges $e_1$, $e_2$ adjacent to $V$. 
\end{definition}
It follows from the general position hypothesis that each vertex is $3$-valent. Because of the balancing condition it plays then no role which $2$ adjacent edges we use above for the multiplicity.
It can be shown that a tropical Gromov--Witten invariant does not depend on the location of the conditions $p_i$ \cite{GM07c}.

The image $f(\Gamma)$ of a rational tropical stable map with Newton fan $D$ and $\#D=n+1$ through $n$ points in tropical general position is fixed, i.e., there is no other rational tropical stable map of the same combinatorial type passing through the points. That is, the stable map cannot vary in a higher-dimensional family within its combinatorial type.

\begin{example}
Figure \ref{fig-exstablemap} shows a tropical stable map counting for $N^{\trop}((2),(1,1),(1,1,1),(1,1,1))$ with multiplicity $2$. As usual, the picture encodes both the abstract graph $\Gamma$ and the image $f(\Gamma)$: we leave a small gap for a crossing of images of edges which are separate in $\Gamma$, we write expansion factors next to edges, we draw points for the contracted marked ends.
\end{example}

The following theorem follows from  \cite{R15b}.

\begin{theorem}[Correspondence Theorem]\label{thm-corres}
The logarithmic Gromov--Witten invariants of $\PP^1\times\PP^1$ from Definition \ref{def-logGWI} are equal to their tropical counterparts from Definition \ref{def-tropGWI}, i.e., we have
\begin{equation*}
N(\mu,\nu)=N^{\trop}(\mu,\nu).
\end{equation*}
\end{theorem}

\section{Lattice paths and their computation}\label{sec-path}

In this section, we provide a slight generalization of Mikhalkin's lattice path count of tropical curves satisfying point conditions, such that ends of higher weight corresponding to points of higher contact order with the toric boundary are included. We also provide a computational tool building on \textsc{Polymake} \cite{polymake}) that computes those numbers. We use our tool for first experiments concerning piecewise polynomial structure. 

\subsection{Mikhalkin's lattice path algorithm}\label{subsec-latticepath}
Let $\Delta$ be a lattice polygon in $\RR^{2}$.

\begin{definition}
  A path $\gamma: [0,n] \rightarrow \RR^2$ is called a \textit{lattice path} if $\gamma |_{[j-1,j]}, j=1,\ldots,n$ is an affine-linear map and $\gamma(j) \in \ZZ^2 $ for all $ j=0\ldots,n$.
\end{definition}

Let $ \lambda$ be a fixed linear map of the form $ \lambda: \RR^2 \to \RR, \lambda (x,y) = x-\varepsilon y $, where $ \varepsilon $ is a small irrational number. A lattice path is called $\lambda$-increasing if $\lambda \circ \gamma$ is strictly increasing. Let $\Delta$ be a lattice polygon. Let $p$ and $q$ be the points in $\Delta$ where $\lambda|_{\Delta}$ reaches its minimum and maximum, respectively. 

Let $\beta = (\beta(1),\beta(2),\beta(3),\ldots,)$ be a finite sequence of nonnegative integers. Define
\begin{equation*}
I\beta = \sum_{k\geq 1} k \cdot \beta(k), \hspace{20pt} \text{and} \hspace{20pt} |\beta| = \sum_{k\geq 1} \beta(k). 
\end{equation*}
For each edge $e$ of $\Delta$, let $\beta_{e} = (\beta_{e}(1), \beta_{e}(2),\beta_e(3),\ldots)$ be a sequence of nonnegative integers such that $I\beta_e = |e\cap \ZZ^2|-1$, and let $\vec{\beta} = \{\beta_e \, : \, e\in E(\Delta)\}$. Denote by $\vec{\beta}_{+}$, resp. $\vec{\beta}_{-}$,  the set of $\beta_{e}$ where $e$ is an edge traversed by the path from $p$ to $q$ going clockwise, resp. counterclockwise, around $\partial \Delta$. Define
\begin{equation*}
I\vec{\beta} = \sum_{e \in E(\Delta)} I\beta_{e}, \hspace{20pt} \text{and} \hspace{20pt} |\vec{\beta}| = \sum_{e \in E(\Delta)} \beta_e. 
\end{equation*}
 Denote by $n_{+}$, resp. $n_{-}$, the lattice length of the unique path from $p$ to $q$ along $\partial \Delta$ going clockwise, resp. counterclockwise.   A $\vec{\beta}_{\pm}$-\textit{initial path} is a path $\delta_{\pm}:[0,n_{\pm}] \to \partial \Delta$ such that
\begin{itemize}
	\item $\delta_{\pm}$ is a $\lambda$-increasing path from $p$ to $q$ along $\partial \Delta$ (clockwise for $\vec{\beta}_{+}$, and counterclockwise for $\vec{\beta}_{-}$);
	\item $\delta$ has exactly $\beta_{e}(k)$ steps along the edge $e$ for each edge $e$ encountered by $\delta_{\pm}$. 
\end{itemize}

\begin{definition} \label{def-multpath}
	Let $\gamma:[0,n]\rightarrow \Delta$ be a $\lambda$-increasing path from $p$ to $q$, that is, $\gamma(0)=p$ and $\gamma(n)=q$. The \emph{multiplicities} $\mult_+(\gamma)$ and $\mult_-(\gamma)$ are defined recursively as follows:
	\begin{enumerate}
		\item \label{def-mu-a}
		$\mult_{+}(\delta_{+}):=1$, if $\delta_{+}$ is a $\vec{\beta}_{+}$--initial, and  $\mult_{-}(\delta_{-}):=1$ if $\delta_{-}$ is a $\vec{\beta}_{-}$--initial path.
		\item \label{def-mu-b}
		If $\gamma \neq \delta_{\pm}$ let $k_{\pm} \in [0,n]$ be the smallest number such that $\gamma$ makes a left turn (respectively a right turn for $\mult_-$) at $\gamma(k_{\pm})$. (If no such $k_{\pm}$ exists we set $\mult_{\pm}(\gamma):=0$). Define two other $\lambda$-increasing lattice paths as follows:
		\begin {itemize}
		\item $\gamma_{\pm}':[0,n-1]\rightarrow \Delta$ is the path that cuts the corner of $\gamma(k_{\pm})$, i.e., $\gamma'_{\pm}(j):=\gamma(j)$ for $j<k_{\pm}$ and $\gamma'_{\pm}(j) := \gamma(j+1)$ for $j \geq k_{\pm}$.
		\item $\gamma''_{\pm}:[0,n]\rightarrow \Delta$ is the path that completes the corner of $\gamma(k_{\pm})$ to a parallelogram, i.e., $\gamma''_{\pm}(j):=\gamma(j)$ for all $j \neq k_{\pm}$ and $\gamma''_{\pm}(k_{\pm}) :=\gamma(k_{\pm}-1)+\gamma(k_{\pm}+1)-\gamma(k_{\pm})$ (see Figure \ref{fig-paths}).
		
		\begin{figure}
			\begin{center}
				\includegraphics{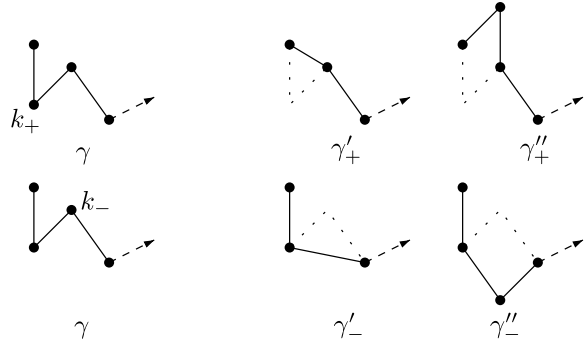}
				\caption{Cutting or completing a corner for the recursion in the lattice path algorithm.}\label{fig-paths}
			\end{center}
		\end{figure}
		
		\end {itemize}
		Let $T$ be the triangle with vertices $\gamma(k_{\pm}-1),\gamma(k_{\pm}), \gamma(k_{\pm}+1)$. Then we set
		\begin{equation*}
		\mult_{\pm}(\gamma):=2 \cdot \Area{T} \cdot \mult_{\pm}(\gamma'_{\pm}) + \mult_{\pm}(\gamma''_{\pm}).
		\end{equation*}
		As both paths $\gamma'_{\pm}$ and $\gamma''_{\pm}$ include a smaller area
		with $\delta_{\pm}$, we can assume that their multiplicity is known. If
		$\gamma''_{\pm}$ does not map to $\Delta$, $\mult_{\pm}(\gamma''_{\pm})$ is
		defined to be zero.
	\end{enumerate}
	Finally, the multiplicity $\mult(\gamma)$ is defined to be the product 	$\mult(\gamma) := \mult_+(\gamma) \mult_-(\gamma)$. Strictly speaking, $\mult_{\pm}(\gamma)$ depends on $\vec{\beta}$. If it is necessary to emphasize the dependence on $\vec{\beta}$, we write $\mult_{\pm}^{\vec{\beta}}(\gamma)$. 
\end{definition}

\begin{rem1}
	As we discuss in \S \ref{subsec-generalizedmikpos}, to compute $\mult_{-}(\gamma)$, we may choose \textit{any} left turn at step \ref{def-mu-b}, not just the first one. Similarly, to compute  $\mult_{+}(\gamma)$, we may choose \textit{any} right turn. 
\end{rem1}

Note that the multiplicity of a path $\gamma$ is positive only if the recursion above ends with the path $\delta_+:[0,n_+]\rightarrow \Delta$ (respectively $\delta_-$). In other words, if we end up with a ``faster'' or ``slower'' path $\delta':[0,n']\rightarrow \Delta$ such that $\delta_+([0,n_+])=\delta'([0,n'])$ but $n'\neq n_+$, or with a path that has different step sizes, then the multiplicity is zero.

\begin{definition}[Mikhalkin's position]
Let $p_1,\ldots,p_n$ be points in $\RR^2$ in tropical general position. We say that they are in \emph{Mikhalkin's position} if they lie on a line with small irrational negative slope through $(0,0)$ and their relative distances grow. The line is denoted by $L_\lambda$.
\end{definition}

Let $\Delta$ be a lattice polygon. Fix $\lambda$, and two $\lambda$-increasing lattice paths $\delta_{+}:[0,n_+]\rightarrow \partial \Delta$ (going clockwise around $\partial \Delta$) and $ \delta_-: [0,n_-] \rightarrow \partial \Delta$ (going counterclockwise around $\partial \Delta$) on the boundary of $\Delta$.  Let $D$ be the Newton fan dual to $\Delta$, where the expansion factors are given by the lattice lengths of the steps of $\delta_+$ and $\delta_-$.

\begin{theorem}[\cite{Mi03}] \label{thm-latticepaths}
The number of $\lambda$-increasing lattice paths in $\Delta$, counted with multiplicity as defined above, equals the number of tropical stable maps with Newton fan $D$ passing through $n$ points in Mikhalkin's position, counted with multiplicity.
\end{theorem}

\begin{example}\label{ex-latticepath}
Let $\Delta= \conv\{(0,0),(1,2),(2,3),(3,2),(4,0)\}$ as depicted in Figure \ref{fig-pathex}. For the path $\delta_+$ we take very lattice point on the upper boundary from $(0,0)$ to $(4,0)$, for the path $\delta_-$, we take only one step. Thus the Newton fan $D$ equals $\{\binom{-2}{1}, \binom{-1}{1}, \binom{1}{1}, \binom{2}{1},\binom{0}{-4}\}$.
Figure \ref{fig-pathex} shows the lattice path recursion for the path $\delta_+$ itself. We have $\mult_+(\delta_+)=1$. To compute the negative multiplicity, we follow the algorithm in Definition \ref{def-multpath}. The first right corner can be completed or cut. If it is completed, the next right corner has to be cut, as for the completion, the path would not be contained in $\Delta$ anymore. The next right turn can be completed or cut again, in each case after that it is possible only to cut the remaining corners. 

If the first corner is cut, all the next corners have to be cut, too. 

In the end, we obtain three possible Newton subdivisions of multiplicity $3\cdot 3\cdot 4=36$, $3\cdot 1\cdot 8=24$ and $1\cdot 5\cdot 8=40$. Altogether, $\mult(\delta_+)=1\cdot \mult_-(\delta_+)=1\cdot(36+24+40)=100$.
The Newton subdivisions which are constructed with this recursion for the lattice path are shown in Figure \ref{fig-pathex}, below, we depict the three dual tropical stable maps with Newton fan $D$ passing through $4$  points in Mikhalkin's position. 
\end{example}

\begin{figure}
\begin{center}
\includegraphics{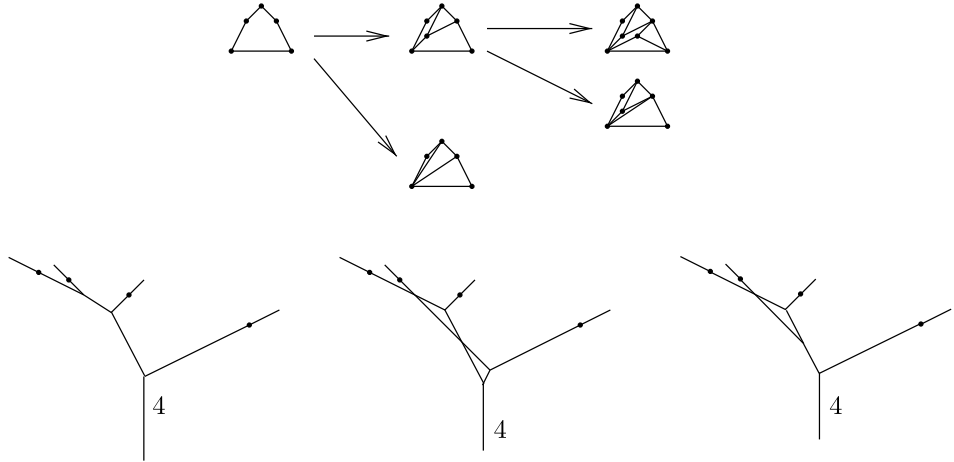}
\caption{Paths and dual tropical curves.}\label{fig-pathex}
\end{center}
\end{figure}

\subsection{The \textsc{Polymake} tool to compute lattice paths}\label{subsec-polymake}

We overview the \textsc{Polymake} \cite{polymake} program that computes multiplicity of a $\lambda$-increasing lattice path, which can be found at the following github repository. 
\begin{center}
	\url{https://github.com/dcorey2814/latticePathAlgorithm}
\end{center}
\noindent There are 3 main functions.

\noindent \textbf{Function}: \texttt{multiplicity}

\textit{Input}: an integer $g$, a lattice polytope $\Delta$, the multiplicities $\vec{\beta}$, and a $\lambda$-increasing lattice path $\gamma$. 

\textit{Output}: the multiplicity of the lattice path $\gamma$ in $\Delta$. 

\medskip 

\noindent \textbf{Function}: \texttt{count\_lattice\_paths\_with\_multiplicity}

\textit{Input}: an integer $g$, a lattice polytope $\Delta$, and the multiplicities $\vec{\beta}$. 

\textit{Output}: the number of length $(|\vec{\beta}|+g-1)$ $\lambda$--increasing lattice paths in $\Delta$, counted with multiplicity. 

\medskip

\noindent \textbf{Function}: \texttt{nonzero\_paths}

\textit{Input}: an integer $g$, a lattice polytope $\Delta$, and the multiplicities $\vec{\beta}$.

\textit{Output}: all length $(|\vec{\beta}| + g -1)$ $\lambda$--increasing lattice paths $\gamma$ in $\Delta$ with nonzero multiplicity, together with their multiplicities. 

\medskip 

We illustrate the functionality with some examples. Let $\Delta$ be the rectangle with vertices $(0,0)$, $(0,5)$, $(3,0)$, and $(3,5)$, and consider the multiplicities:
\begin{equation*}
\mu_{{\tiny \mbox{up}}} = (1,1,1), \hspace{10pt} \mu_{{\tiny \mbox{down}}} = (3), \hspace{10pt} \nu_{{\tiny \mbox{left}}} = (3,2), \hspace{10pt} \nu_{{\tiny \mbox{right}}} = (5).
\end{equation*}
We translate from partitions to sequences $\vec{\beta}$ that have an entry at the $i$-th place for every part $i$ appearing in the partition:
\begin{equation*}
\beta_{{\tiny \mbox{up}}} = (3), \hspace{10pt} \beta_{{\tiny \mbox{down}}} = (0,0,1), \hspace{10pt} \beta_{{\tiny \mbox{left}}} = (0,1,1), \hspace{10pt} \beta_{{\tiny \mbox{right}}} = (0,0,0,0,1).
\end{equation*}

\noindent First, we record $\Delta$, its vertices, and its edges.
\begin{verbatim}
  $Delta = new Polytope(POINTS=>[[1,0,0], [1,3,0], [1,0,5], [1,3,5]]);
  $vertices = $Delta->VERTICES->minor(All,~[0]);
  $edge_indices = $Delta->VERTICES_IN_FACETS;
  @edges = map {new Pair<Vector, Vector>(
                $vertices->row($_->[0]), $vertices->row($_->[1]))} 
           @$edge_indices;
\end{verbatim}
The multiplicity is recorded as a map, where the keys are the edges, and the values are the multiplicity along the edge. 
\begin{verbatim}
  $betas = new Map<Pair<Vector,Vector>,Vector>;
  $betas -> {$edges[0]} = new Vector([0,1,1]);
  $betas -> {$edges[1]} = new Vector([0,0,1]);
  $betas -> {$edges[2]} = new Vector([0,0,0,0,1]);
  $betas -> {$edges[3]} = new Vector([3]);
\end{verbatim}
A length $\ell$ lattice path is recorded as a $\ell \times 2$ matrix. For example, the lattice path 
\begin{equation*}
  \gamma = (0,5) \to (0,3) \to (0,0) \to (1,2) \to (2,5) \to (3,5) \to (3,0).
\end{equation*}
is recorded as
\begin{verbatim}
  $gamma = new Matrix([[0,5], [0,3], [0,0], [1,2], [2,5], [3,5], [3,0]]);
\end{verbatim}
We compute the multiplicity of $\gamma$ by running 
\begin{verbatim}
  multiplicity(0, $Delta, $betas, $gamma);
\end{verbatim}
which returns $1440$. This agrees with the calculation in Example \ref{ex:lowerMultiplicity}.  Similarly, running 
\begin{verbatim}
  count_lattice_paths_with_multiplicity(0, $Delta, $betas);
\end{verbatim} 
returns 19170, the number of length $6$ $\lambda$-increasing paths in $\Delta$ counted with $\vec{\beta}$--multiplicity.  Finally, running 
\begin{verbatim}
  nonzero_paths(0, $Delta, $betas);
\end{verbatim} 
returns a map whose keys are the length $6$ $\lambda$-increasing paths in $\Delta$ with nonzero multiplicity, and the values are their multiplicities. In this case, there are 16 such paths. 

\subsection{Structural results obtained via lattice paths} \label{subsec-structurepath}

Let $\Delta$ be a rectangle. Given two tuples of partitions of the lattice width resp.\ height $(\mu_{1}, \mu_{2})$ and $(\nu_{1}, \nu_{2})$, as before we let $N(\mu_{1}, \mu_{2},\nu_{1}, \nu_{2})$ be the number of tropical genus 0 curves subject to the boundary data, see Figure \ref{fig:rectanglecnu}.  Given $\lambda$-increasing paths $\gamma,\gamma'$, we write $\gamma \prec_{L} \gamma'$ if $\gamma'$ appears in the recursion to compute $\mult_{-}(\gamma)$. Similarly, we write $\gamma \prec_{R} \gamma'$ if $\gamma'$ appears in the recursion to compute $\mult_{+}(\gamma)$. Given two lattice points $p_1, p_2$, the step $p_1 \to p_2$ is recorded as the vector $s = \brak{a}{b} = p_2 - p_1$.  Denote by $\sD(\gamma)$ to be the multiset of down steps. 

\begin{figure}
	\includegraphics[width=\textwidth]{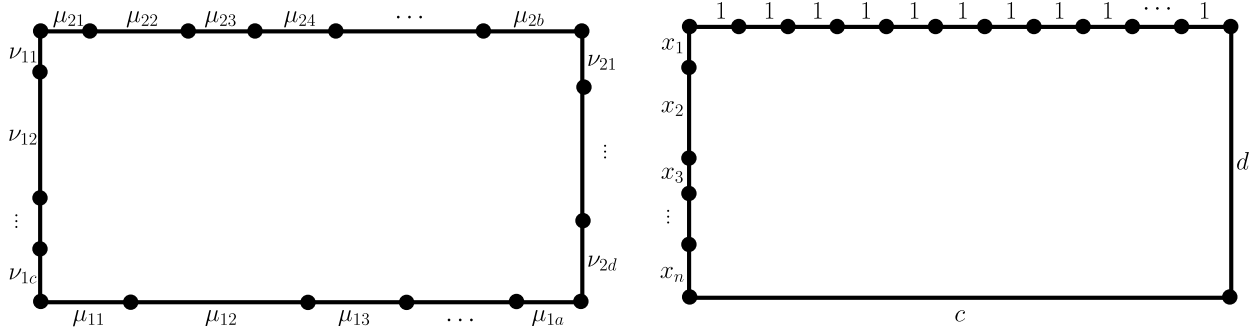}
	\caption{Rectangle with boundary conditions.}
	\label{fig:rectanglecnu}
\end{figure}

 We gather some general properties of $\lambda$-increasing lattice paths in $\Delta$.  

\begin{proposition}
	\label{prop:downsteps}
	If $\gamma \prec_{L} \gamma'$ or $\gamma \prec_{R} \gamma'$, then $\sD(\gamma') \subseteq \sD(\gamma)$. 
\end{proposition}

\begin{proof}
	The two cases are similar, so we only handle the  $\gamma \prec_{L} \gamma'$ case. 	For a fixed $\gamma$, it suffices to consider the two paths $\gamma'$ and $\gamma''$ obtained by cutting the corner and pushing out the parallelogram at the first right turn, respectively. Pushing-out a parallelogram preserves the multiset of down steps, i.e., $\sD(\gamma'') = \sD(\gamma)$. If the first right turn does not have a down step, then $\sD(\gamma'') = \sD(\gamma)$ (note that cutting the corner cannot introduce a new downward step since $\gamma$ is $\lambda$-increasing). 	Otherwise, the down step must be the second step in the turn, cutting the corner removes this down step, and so $\sD(\gamma') \subset \sD(\gamma)$. 
\end{proof}

\begin{proposition}
	\label{prop:noSEstepPositiveMultiplicity}
	Suppose $\mu_{2} = (1,\ldots,1)$ and $\mult(\gamma) \neq 0$.
	\begin{enumerate}
		\item Each step can go at most one unit to the right. 
		\item The path $\gamma$ cannot have a down-right step. 
	\end{enumerate}
\end{proposition}

\begin{proof}
	Part (1) is essentially  \cite[Lemma~3.6]{GM07b}, so we focus on (2). 
	We prove that if $\gamma$ has a down-right step, then $\mult_{-}(\gamma) = 0$ by induction on $A_{+}(\gamma)$, the area of the region inside $\Delta$ above $\gamma$. If $A_{-}(\gamma) = 0$, then $\gamma$ traces the lower and left edges of $\Delta$, and hence has no down-right step. Now consider the general case, and assume that the first down-right step is $z_{1} \to z_{2}$. If there is no right turn, then $\gamma$ is not $\beta$--initial, hence $m_{-}(\gamma) = 0$. So we may assume that a right turn exists. Furthermore, we may assume that $z_{1} \to z_{2}$ is part of this right turn, since otherwise this down-right step persists to the next step of the upper multiplicity algorithm, so $\mult_{-}(\gamma) = 0$ by the inductive hypothesis. Completing the parallelogram at the first right turn translates the down-right step, so it suffices to prove that $\mult_{-}(\gamma') = 0$ where $\gamma'$ is the path obtained by cutting the corner.  If the first right turn is $w \to z_{1} \to z_{2}$, then $w\to z_1$ cannot be vertical. If we cut the corner, $\gamma'$ goes right by two units. Given the $\mu_{2} = (1,\ldots,1)$ hypothesis, statement (1) tells us that $\mult_{-}(\gamma') = 0$. If we complete the corner to a parallelogram, the down-right step remains and so by induction, $m_{-}(\gamma'') = 0$. Accordingly, also $\mult_{-}(\gamma) = 0$.
  Suppose that the first right turn is $z_1 \to z_2 \to w$. Then $z_2 \to w$ must have an rightward component, and so the step $z_1\to w$ in $\gamma'$ goes two units right. Given the $\mu_{2} = (1,\ldots,1)$ hypothesis, statement (1) tells us that $\mult_{-}(\gamma') = 0$.  
\end{proof}

\begin{proposition}
	\label{prop:canthappen}
	Suppose $\mu_{1} = \mu$ and $\nu_{2} = \nu$ are both partitions in just one part. Suppose $\gamma$ is a lattice path such that $\mult(\gamma) \neq 0$, and $\gamma'$ is a path in the recursion of computing $\mult_{+}(\gamma)$ or $\mult_{-}(\gamma)$ (possibly $\gamma = \gamma'$) with nonzero multiplicity. Then $\gamma'$ cannot have a lattice point in the relative interior of the right edge or the bottom edge. 
\end{proposition}

\begin{proof}
	Suppose $\gamma \prec_{R} \gamma'$. If $\gamma'$ has a point in the relative interior of the right edge, then all subsequent paths in this recursion also have this point, so their positive multiplicities are all 0. Therefore, $\gamma'$ cannot have a point on the right edge. By symmetry, if $\gamma \prec_{L} \gamma''$ then $\gamma''$ cannot have a point in the relative interior of the bottom edge. 	
	
	Now suppose that $\gamma'$ has a point in the relative interior of the bottom edge. Then every path in the recursion before $\gamma'$ also have this point, in particular so does $\gamma$. By an argument analogous to the previous paragraph, this also means that every path in the recursion to compute $\mult_{-}(\gamma)$ also has this point, and so $\mult_{-}(\gamma)=0$, a contradiction. Therefore, $\gamma'$ does not have a point in the relative interior of the bottom edge. By a symmetric argument, if $\gamma \prec_{L} \gamma''$, then $\gamma''$ cannot have a point in the relative interior of the right edge. 
\end{proof}

For the rest of this section, we consider the following specific boundary conditions. Fix $c$, $d$, and $n$. Define the function

\begin{equation*}
N_{\rect}(x_1,\ldots,x_n) = N((c), (1^{c}), (x_1,\ldots,x_n), (d)). 
\end{equation*}

\begin{theorem}\label{maintheorem1}
	The function $N_{\rect}(x_{1},\ldots,x_{n})$ is (generically) polynomial. 
	
	More precisely, there is a polynomial in the $x_i$ which equals $N_{\rect}(x_{1},\ldots,x_{n})$ for all choices of $x_i$ which do not satisfy the equalities from Equation (\ref{eq:walls}).
\end{theorem}

Notice that we even obtain polynomiality for this count, not only piecewise polynomiality. This is similar to the case of one-part Hurwitz numbers, i.e., double Hurwitz numbers for which one branch point is fully ramified, which are also polynomial and not only piecewise polynomial.

The rest of this section is devoted to the proof of Theorem \ref{maintheorem1} and some examples. 
The proof has three steps: first, we characterize the paths that contribute to $N_{\rect}(x_{1},\ldots,x_{n})$ with a nonzero multiplicity. This is done is \S\ref{subsec-charpaths}. In particular, the number of such paths is finite and does not depend on the concrete values for our variable $x_1,\ldots,x_n$, as long as they satisfy some genericity inequalities.
The third step is to prove that for each such path, the multiplicity is piecewise polynomial in the $x_i$, which then completes the proof of the theorem. This is done in \S\ref{subsec-multpp}. But before we can proceed to this final step, we have to develop a technical trick that allows us to compute multiplicities of paths more easily: we show that one can in fact pick any right turn in a lattice path and perform the recursion for the negative multiplicity, it does not necessarily have to be the first right turn as in Mikhalkin's original description \cite{Mi03}. This is done in \S\ref{subsec-generalizedmikpos}.

\subsection{Characterizing the paths}\label{subsec-charpaths}

We proceed by counting lattice paths and their multiplicities. Each such lattice path has $n+c+1$ steps. Given a subset $I = \{i_1,\ldots, i_k\}$ of $[n]$, denote by $x_{I} = x_{i_1} + \cdots + x_{i_{k}}$. Assume that
% We assume that   Assume that $x_1,\ldots,x_n$ are distinct and that
\begin{equation}
\label{eq:walls}
    x_{I} \neq x_J%\sum_{i\in I} x_i \neq \sum_{j\in J} x_j
\end{equation}
for all subsets $I \neq J$ of $[n]$; in particular, the $x_i$'s are distinct. 

In view of Proposition \ref{prop:noSEstepPositiveMultiplicity}, since each step goes at most 1 to the right, the upper multiplicity is resolved ``column-wise'' in the sense of \cite[\S3]{GM07b}. That means, if we follow the paths in the recursion and view the triangles and parallelograms as part of a subdivision with which we fill the rectangle, then this subdivision contains all lines of the form $\{x=i\}$ for $i=0,\ldots,c$ if we assume the rectangle has coordinates $(0,0)$, $(c,0)$, $(0,d)$, $(c,d)$.
 So, let us consider the ways in which a single column is resolved, i.e., filled with triangles and parallelograms. Let $T$ be the trapezoid with vertices $(0,0)$, $(0,a)$, $(1,a)$, $(1,b)$, and choose integers $\eta_1,\ldots,\eta_k$ such that $\eta_{I} \neq \eta_{J}$ for any pair of disjoint subsets $I,J\subset [k]$. Let $\gamma$ be the lattice path going from $(0,a)$ to $(1,b)$ around the lower end of $T$ with step sizes $\eta_1,\ldots,\eta_k$ along the left edge, see Figure \ref{fig:1-step}. As is standard practice, denote by $\chi_{E}$ the indicator vector of $E\subset [n]$. Define $\vec{\beta}$ by
 \begin{equation*}
 \vec{\beta} = \{\beta_{\text{left}}, \beta_{\text{right}}, \beta_{\text{up}}, \beta_{\text{down}}\}
 \end{equation*}
where $\beta_{\text{up}} = \beta_{\text{down}} = (1)$, $\beta_{\text{left}} = \chi_{\{\eta_1,\ldots,\eta_{n}\}}$, and $\beta_{\text{right}}$ is to be determined. 
\begin{figure}
	\includegraphics[height=4cm]{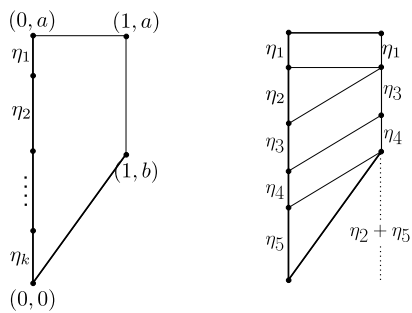}
	\caption{The path at the start of resolving a subdivision columnwise.}
	\label{fig:1-step}
\end{figure}

\begin{lemma}
	\label{lem:1-step}
	Given $\vec{\beta}$ as above, we have $\mult_{+}^{\vec{\beta}}(\gamma) \neq 0$ if and only if $\beta_{\text{right}} = \chi_{E}$ where $E \subset \{\eta_1,\ldots,\eta_k\}$. In this case, $\mult_{+}^{\vec{\beta}}(\gamma) = \prod_{\eta_{i} \notin E} \eta_{i}$.	
\end{lemma} 

\begin{proof}	
	Suppose that $\mult_{+}^{\vec{\beta}}(\gamma) \neq 0$.  If $\gamma \prec_{R} \gamma'$ and $\gamma'$ traces the top and right sides of $T$, the steps along the right edge have distinct sizes of the form $\eta_i$ by Proposition \ref{prop:downsteps}. Since the $\eta_J$'s are distinct, there is a unique $J\subset [k]$ such that $\eta_J = a-b$,  so $\beta_{\text{right}} = \chi_{E}$ where $E = \{\eta_j \, : \, j\in J\}$.  	
		
	Conversely, suppose $\beta_{\text{right}} = \chi_E$ where $E$ is as in the previous paragraph.  In the recursion to compute $\mult_{+}^{\vec{\beta}}(\gamma)$, each right turn occurs on the left edge. At each stage of the recursion, there is exactly one choice, cut the corner or complete the parallelogram, that leads to a positive multiplicity. If the step-size before the right turn is $\eta_{j}$ for $j\in J$, then complete the parallelogram; otherwise, cut the corner. See Figure \ref{fig:1-step}. That these are the only choices again follows from Proposition \ref{prop:downsteps} and the fact that the $\eta_J$'s are distinct. The triangles obtained by cutting the corners have areas $\eta_i$ for $i \notin J$, so $\mult_{+}^{\vec{\beta}}(\gamma) = \prod_{\eta_{i} \notin E} \eta_{i}$, as required. 
\end{proof}

As before, denote by $\sD(\gamma)$ the down steps of $\gamma$, and denote by $\sU\sR(\gamma)$ the up-right steps of $\gamma$. 
\begin{lemma}
	Suppose $\gamma$ is a $\lambda$-increasing path with $n+c+1$ steps and $\mult(\gamma) \neq 0$. 
	\begin{enumerate}
		\item The last step is the full right edge of $\Delta$.
		\item The first step off the left edge is up-right. 
		\item $\sD(\gamma) = \{\brak{0}{-x_1},\ldots,\brak{0}{-x_n},\brak{0}{-d}\}$.
		\item $\sU\sR(\gamma) = \{\brak{1}{x_{I_1}}, \ldots, \brak{1}{x_{I_{\ell}}} \}$ where $I_1,\ldots,I_{\ell}$ is a partition of $[n]$, and, for each $a=1,\ldots,\ell$ $\{\brak{0}{-x_i} \, : \, i \in I_{a}\}$ are among the down steps taken before $\brak{1}{x_{I_{a}}}$. 
		\item The remaining steps are of the form $\brak{1}{0}$.   
	\end{enumerate}	
\end{lemma}

\begin{proof}
	Statement (1) follows from Proposition \ref{prop:canthappen}. 
	
	Consider (2). Suppose the first step off the left edge is $\brak{1}{0}$.  In order for $\mult_{-}(\gamma) \neq 0$, the bottom-left corner $x$ of $\Delta$ must be a point in a lattice path $\gamma'$ with $\gamma \prec_{L} \gamma'$. The only way to reach $x$ is to complete a parallelogram to it. Such a parallelogram must be a rectangle, but the resulting path $\gamma'$ has a point in the relative interior of the bottom edge. This contradicts Proposition \ref{prop:canthappen}.  
	
	Consider (3). Any $\beta_{-}$--initial path has as down steps exactly $\{\brak{0}{-x_1},\ldots,\brak{0}{-x_n}\}$. So, by (1) and  Proposition \ref{prop:downsteps}, we have $\{\brak{0}{-x_1},\ldots,\brak{0}{-x_n},\brak{0}{-d}\} \subset \sD(\gamma)$. By Proposition \ref{prop:noSEstepPositiveMultiplicity}(1), exactly $c$ of the steps cannot be down, so there are exactly $n+1$ down-steps. This proves (3). 
	
	Finally, consider (4) and (5). Since $\mu_2 = (1,\ldots,1)$, the upper multiplicity is resolved column-wise, and there are no down-right steps by Proposition \ref{prop:noSEstepPositiveMultiplicity}. Set 
	\begin{equation*}
	J_{\leq a}(\gamma) = \{j\in [n] \, : \, \brak{0}{-x_j} \text{ is a down-step of } \gamma \text{ in the columns } x=b \text{ for } b\leq a\}.
	\end{equation*}
	Write $\gamma_a$ for the path obtained by resolving the first $a$ columns of $\gamma$. By induction and Lemma \ref{lem:1-step}, the down-steps of $\gamma_a$ along $x=a$ are $\brak{0}{-x_i}$ for $i$ in  $I = J_{\leq a}(\gamma) \setminus (I_{1} \cup \cdots \cup I_{a-1})$ and the step to column $x=a+1$ is either $\brak{1}{0}$ or $\brak{1}{x_{I_a}}$ for some $I_{a} \subset I$. The union $I_{1} \cup \cdots \cup I_{\ell}$ is $[n]$ since $\gamma$ must reach the top-right corner of $\Delta$ (by (1)), and the down-steps before the right-edge have sizes $x_{1},\ldots,x_{n}$. 
\end{proof}

Thus, we may enumerate the lattice paths $\gamma$ with $\mult(\gamma) \neq 0$. Choose pairwise disjoint sequences $A_0,\ldots,A_{c-1}$ in $[n]$, and pairwise disjoint subsets $B_{0},\ldots,B_{c-1}$ of $[n]$ such that
\begin{enumerate}
	\item $\cup A_{i} = [n] = \cup B_i$;
	\item $B_{0} \neq \emptyset$, and
	\item $B_k \subset \left( \bigcup_{i=0}^{k} A_i \right) \setminus  \bigcup_{i=0}^{k-1} B_i$.
\end{enumerate}
Set $\mathbf{A} = (A_0,\ldots,A_{c-1})$ and $\mathbf{B} = (B_0,\ldots,B_{c-1})$. Let $\gamma(\mathbf{A},\mathbf{B})$ be the unique $\lambda$-increasing lattice path such that, in the $k$-th column (for $k=0,1,\ldots,c-1$), there are downward steps of sizes $(x_i \, : \, i \in A_{k})$ in order, and the step to the next column goes $\sum_{i\in B_k} x_i$ units up (or 0 units up, if $B_{k} =\emptyset$). 

Note that in particular, the number of such paths is finite and does not depend on the concrete values for our variable $x_1,\ldots,x_n$, as long as they satisfy the genericity inequalities mentioned above.
Before we can prove that for each path $\gamma(\mathbf{A},\mathbf{B})$, the multiplicity is polynomial in the $x_i$ (away from the inequalities from above), we further study properties of lattice path multiplicities more generally.

\subsection{Generalized Mikhalkin position and cut tropical curves}\label{subsec-generalizedmikpos}

Denote by $\mathfrak{S}_{n-1}$ the symmetric group on $n-1$ numbers. Given a $n$-step lattice path $\gamma:[0,n] \to \Delta$ and $\sigma \in \mathfrak{S}_{n-1}$, define $\mult_{-}(\gamma,\sigma)$ recursively in the following way. The definition is similar to $\mult_{-}(\gamma)$, except that in the recursive step, instead of letting $k\in [1,n-1]$ be the smallest number such that there is a right turn at $k$, let $k\in [1,n-1]$ be
\begin{equation}
\label{eq:permutedRightTurn}
	k = \argmin(\sigma(i) \, : \,  \text{ there is a right turn at } i  ).
\end{equation}
Denote by $\gamma'$ the path obtained by cutting this corner, and $\gamma''$ the path obtained by completing the parallelogram. Then
\begin{equation*}
\mult_{-}(\gamma,\sigma) = 2 \cdot \Area T \cdot \mult_{-}(\gamma', \sigma') + \mult_{-}(\gamma'', \sigma). 
\end{equation*}
Here, $\sigma' = \partial_{k} \circ \sigma \circ \delta^{\sigma^{-1}(k)}$ where $\delta^{i}:[n-2] \to [n-1]$ be the unique order-preserving injection whose image misses $i$ and $\partial_{i}:[n-1] \to [n-2]$ the unique order-preserving surjection such that $\partial_{i}(i+1) = i$.  When $\sigma$ is the identity permutation, $\mult_{-}(\gamma,\sigma)$ is exactly the definition of $\mult_{-}(\gamma)$. 

The function $\mult_{-}(\gamma,\sigma)$ is related to tropical curve counting in the following way. As before, fix a line $L_\lambda$ with a small negative irrational slope.  

\begin{definition}
	\label{def:GeneralizedMikhalkinPosition}
	Let $\vec{p} = (p_1,\ldots,p_n)$ be a sequence of point on $L_\lambda$ with increasing $x$-values, and set $u_{i} = \dist(p_i,p_{i+1})$.  For a permutation $\sigma \in \mathfrak{S}_{n-1}$, the points $\vec{p}$ are in \textit{$\sigma$--Mikhalkin position } if 
	\begin{equation*}
	u_{\sigma(1)} \ll u_{\sigma(2)} \ll \cdots \ll u_{\sigma(n-1)}.
	\end{equation*} 
\end{definition}
\noindent Note that ordinary Mikhalkin position is the case where $\sigma$ is the identity. 

\begin{definition}
	Fix a $\lambda$-increasing lattice path $\gamma$ with $n$ steps, and points $\vec{p}$ on $L_\lambda$. A \textit{lower} $(\lambda, \gamma, \vec{p})$--  \textit{tropical curve} is a weighted rational 1-dimensional polyhedral complex $\Gamma_{-}$ in the half-space below $L_\lambda$ satisfying the following properties.
	\begin{enumerate}
		\item Each point on $\Gamma_{-} \setminus L_\lambda$ satisfies the balancing condition.
		\item Every connected component of $\Gamma_{-}$ intersects $L_\lambda$ and $\Gamma_{-} \cap L_\lambda = \{p_1,\ldots,p_{n}\}$.
		\item Each $p_i$ is a 1-valent vertex of $\Gamma_{-}$ and its adjacent edge is dual to the $i$th step in $\gamma$.   
	\end{enumerate}
\end{definition}
\noindent Such a $\Gamma_{-}$ should be thought of as the part of a full tropical curve $\Gamma$ that lies below $L_\lambda$. 

\begin{lemma}
	\label{lem:multiplicityLowerTropicalCurves}
	Let $\gamma$ be a $n$-step $\lambda$-increasing lattice path, $\vec{p} = (p_1,\ldots,p_n)$ a sequence of points on $L_\lambda$ in $\sigma$--Mikhalkin position. The number of lower $(\lambda, \gamma, \vec{p})$--tropical curves, counted with multiplicity, is $\mult_{-}(\gamma,\sigma)$. 
\end{lemma}

\begin{proof}
	This is essentially the same as in the proof that Mikhalkin's lattice path algorithm counts tropical curves, see \cite[\S7, Theorem~2]{Mi03}. We sketch a proof here. Suppose $(p_1,\ldots,p_n)$ are in $\sigma$--Mikhalkin position. Allow the edges emanating from the $p_{i}$'s to grow. Let $k$ be the integer in Formula \eqref{eq:permutedRightTurn}.  The first pair of edges to meet corresponds those emanating from $p_{k}$ and $p_{k+1}$. When they meet, they either form a 3-valent vertex or a crossing, corresponding to either a triangle or a parallelogram in the lattice path recursion.  This all takes place in a small strip below the line $L_\lambda$, so now replace $L_\lambda$ with its parallel counterpart $L_\lambda'$, and replace $p_{1},\ldots,p_{n}$ with the intersection of these germs with $L_\lambda'$. See Figure \ref{fig:generalizedMikhalkinPosition} for an example.  
\end{proof}

\begin{figure}
	\includegraphics[height=3cm]{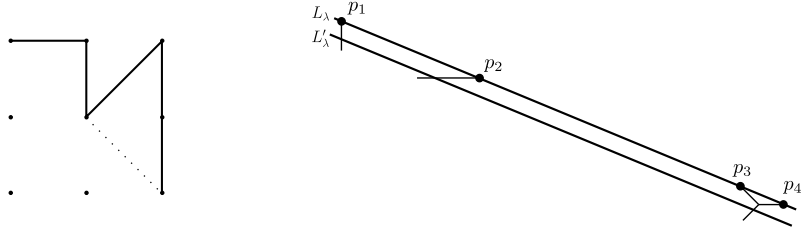}
	\caption{A step in the lattice path recursion, and its tropical curve counterpart. The points $p_1,p_2,p_3,p_4$ are in $\sigma$--Mikhalkin position for the permutation $\sigma = [3,1,2]$. }
	\label{fig:generalizedMikhalkinPosition}
\end{figure}

\begin{lemma}
	\label{lem:multIndependentOfSigma}	
	Given any two ordering $\sigma_1, \sigma_2$, we have that $\mult_{-}(\gamma,\sigma_1) = \mult_{-}(\gamma,\sigma_2)$. 
\end{lemma}

\begin{proof}
	This follows from Lemma \ref{lem:multiplicityLowerTropicalCurves} and the fact that the number of lower $(\lambda,\gamma,\vec{p})$--tropical curves, counted with multiplicity, is independent of $\vec{p}$ \hspace{0.03cm}  (see Theorem 4.8 in \cite{GM07c}). 
\end{proof}

The following statement is used to show the piecewise polynomiality of a path $\gamma(\mathbf{A},\mathbf{B})$:

\begin{proposition}\label{prop:rightturns}
	The negative multiplicity of a lattice path stays the same if at each step of the recursion, we allow to use any right turn to perform the recursion, not necessarily the first right turn. 
\end{proposition}

The analogous statement holds for the positive multiplicity and left turns.

\begin{proof}
	We proceed by induction on the area of the region in $\Delta$ bounded above by $\gamma$. Suppose $\gamma$ has $n$ steps, and suppose there is a right turn at $k$.  Choose a permutation $\sigma$ such that $k$ satisfies Equation \eqref{eq:permutedRightTurn}. By choosing this right turn in the recursion, the expression for the multiplicity is 
	\begin{equation*}
		2 \cdot \Area T \cdot \mult_{-}(\gamma',\sigma') + \mult_{-}(\gamma'',\sigma)
	\end{equation*}
	which equals 
	\begin{equation*}
	2 \cdot \Area T \cdot \mult_{-}(\gamma') + \mult_{-}(\gamma'')
	\end{equation*}
	by Lemma \ref{lem:multIndependentOfSigma}. The proposition now follows from the inductive hypothesis.  
\end{proof}

\begin{example}
We compute the negative multiplicity of the path from Example \ref{ex-latticepath}, see Figure \ref{fig-pathex}. There, we always picked the first right turn in the recursion. Now, we start by picking the second right turn. We then obtain the two paths $\gamma'$ and $\gamma''$ depicted in Figure \ref{fig-pathex3}.
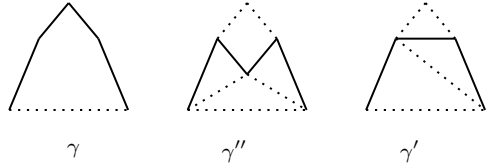
\begin{figure}

\tikzset{every picture/.style={line width=0.75pt}} %set default line width to 0.75pt        

\begin{tikzpicture}[x=0.75pt,y=0.75pt,yscale=-1,xscale=1]
%uncomment if require: \path (0,784); %set diagram left start at 0, and has height of 784

%Straight Lines [id:da9105251112154373] 
\draw    (210,144) -- (225,108) ;
%Straight Lines [id:da44833360925666454] 
\draw    (225,108) -- (240,126) ;
%Straight Lines [id:da6764575117544581] 
\draw    (240,126) -- (255,108) ;
%Straight Lines [id:da02548616742704668] 
\draw    (165,108) -- (180,144) ;
%Straight Lines [id:da14550174739049526] 
\draw  [dash pattern={on 0.84pt off 2.51pt}]  (120,144) -- (180,144) ;
%Straight Lines [id:da8229803673136863] 
\draw    (315,108) -- (345,108) ;
%Straight Lines [id:da2767748586327744] 
\draw  [dash pattern={on 0.84pt off 2.51pt}]  (210,144) -- (240,126) ;
%Straight Lines [id:da2220734734822961] 
\draw  [dash pattern={on 0.84pt off 2.51pt}]  (225,108) -- (240,90) ;
%Straight Lines [id:da3755623019714146] 
\draw  [dash pattern={on 0.84pt off 2.51pt}]  (210,144) -- (270,144) ;
%Straight Lines [id:da4795909614556474] 
\draw  [dash pattern={on 0.84pt off 2.51pt}]  (300,144) -- (360,144) ;
%Straight Lines [id:da5238430401096036] 
\draw    (120,144) -- (135,108) ;
%Straight Lines [id:da5415962949000316] 
\draw    (150,90) -- (165,108) ;
%Straight Lines [id:da4645396613036843] 
\draw    (135,108) -- (150,90) ;
%Straight Lines [id:da6657878631713857] 
\draw    (255,108) -- (270,144) ;
%Straight Lines [id:da07158337700782269] 
\draw    (345,108) -- (360,144) ;
%Straight Lines [id:da23929299503620038] 
\draw    (300,144) -- (315,108) ;
%Straight Lines [id:da9242302051286585] 
\draw  [dash pattern={on 0.84pt off 2.51pt}]  (315,108) -- (330,90) ;
%Straight Lines [id:da21566388866877262] 
\draw  [dash pattern={on 0.84pt off 2.51pt}]  (345,108) -- (330,90) ;
%Straight Lines [id:da27250406109720127] 
\draw  [dash pattern={on 0.84pt off 2.51pt}]  (255,108) -- (240,90) ;
%Straight Lines [id:da018538261143278056] 
\draw  [dash pattern={on 0.84pt off 2.51pt}]  (240,126) -- (270,144) ;
%Straight Lines [id:da7982566199666546] 
\draw  [dash pattern={on 0.84pt off 2.51pt}]  (315,108) -- (360,144) ;

% Text Node
\draw (156,164.7) node  [font=\scriptsize] [align=left] {\begin{minipage}[lt]{10.2pt}\setlength\topsep{0pt}
$\displaystyle \gamma $
\end{minipage}};
% Text Node
\draw (234,164.7) node  [font=\scriptsize] [align=left] {\begin{minipage}[lt]{10.2pt}\setlength\topsep{0pt}
$\displaystyle \gamma ''$
\end{minipage}};
% Text Node
\draw (324,164.7) node  [font=\scriptsize] [align=left] {\begin{minipage}[lt]{10.2pt}\setlength\topsep{0pt}
$\displaystyle \gamma '$
\end{minipage}};

\end{tikzpicture}

\caption{Choosing another right turn to compute the negative multiplicity of the path from Example \ref{ex-latticepath} yields the same negative multiplicity.}\label{fig-pathex3}
\end{figure}
For the path $\gamma''$, it does not matter which turn we pick next, we always have to cut triangles and they are the same for any choice. In total, we have 3 triangles of areas $4$, $3$ and $3$ yielding the multiplicity $36$ for $\gamma''$. To continue the recursion for $\gamma'$, we pick the second right turn of $\gamma'$, as indicated in the picture. Picking the first right turn yields a symmetric subdivision. We have to cut two triangles of areas $8$ and $4$, and to get to $\gamma'$ we already cut a triangle of are $2$. Altogether, we obtain $32$ of the multiplicity  of  $\gamma'$, and $36+2\cdot 32=100$ for the multiplicity of $\gamma$, i.e., the same value we obtained by always picking the first right turn as in Example \ref{ex-latticepath}.
\end{example}

\subsection{Final step of the proof of Theorem \ref{maintheorem1}: the multiplicity of each path that contributes is polynomial}\label{subsec-multpp}

\begin{proposition}
	The multiplicity of $\gamma(\mathbf{A},\mathbf{B})$ is polynomial in $x_1,\ldots,x_n$.
\end{proposition}

\begin{proof}
	To simplify notation, fix $\gamma = \gamma(\mathbf{A}, \mathbf{B})$. The positive multiplicity of $\gamma$ is 
	\begin{equation*}
		\mult_{+}(\gamma) = \prod_{i=1}^{n} x_i.
	\end{equation*}
	This follows from Lemma \ref{lem:1-step} and the fact that the positive multiplicity is resolved columnwise. 
	The lower multiplicity is harder to compute using the original Mikhalkin lattice path algorithm, see Example \ref{ex:lowerMultiplicity}. Instead, we use Proposition \ref{prop:rightturns} and judiciously choose our right turns so that, at each step, there is exactly 1 choice in the algorithm, to cut or to complete the parallelogram.  
	
	First, we only choose right turns of the form up-right to vertically down, or right to vertically down, excluding the top-right corner. As all vertical down steps need to be transported to the left edge by Proposition \ref{prop:downsteps}, we must complete the parallelogram at this step. Eventually, we reach a $\gamma'$ such that all vertical down steps are along the left edge of $\Delta$ (of course, the $d$ down step is still along the right-edge of $\Delta$). In particular, the path from the bottom-left corner to the top-right corner of $\Delta$ is (weakly) increasing. Now, we always choose the last right turn. Inductively, we see that this will always be between the last two steps. At each phase of the algorithm, we cannot complete the parallelogram (doing so would either take the lattice path outside of $\Delta$, or along the relative interior of the bottom edge of $\Delta$, neither of which are allowed). Thus, we cut the corner at each phase. Thus, we found the requisite sequence of right turns. As each triangle has coordinates that are (linear) polynomials in the $x_i$, the areas of the triangles in the induced lower subdivision are polynomials in the $x_i$. This proves that $\mult(\gamma(\mathbf{A}, \mathbf{B}))$ is polynomial away from the walls defined by Formula \eqref{eq:walls}. 	
\end{proof}

\begin{example}
	\label{ex:lowerMultiplicity}
	Consider the rectangle $\Delta$ of width $c=3$ and height $d$, and let $n=2$. Consider the two lattices paths $\gamma_1, \gamma_2$ in Figure \ref{fig:lowerMultiplicityExample}. The first one satisfies $x_1 < d / 3$ whereas the second one satisfies $x_1 > d / 3$. In computing the lower multiplicity, there are 2 lower subdivisions for the first but three for the second. Nevertheless, as a consequence of Proposition \ref{prop:rightturns}, both lower multiplicities satisfy the same formula:
\begin{equation*}
\mult_{-}(\gamma_1) = \mult_{-}(\gamma_2) = 6d^2 x_1 - 3 d x_1^2.
\end{equation*}	
\end{example}

\begin{figure}
	\includegraphics[height=2.5cm]{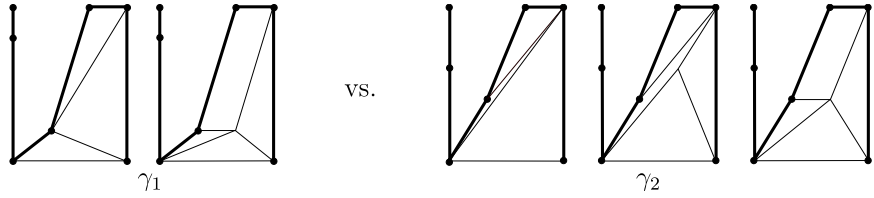}
	\caption{Lower multiplicities of two paths.}
	\label{fig:lowerMultiplicityExample}
\end{figure}

\subsection{Examples}\label{subsec-pathex}

In the following example, we consider a case for which we obtain a count which is polynomial in the entries of our boundary partitions:

\begin{example}
For some number $c$, set $\mu=((c),(c))$, i.e., we fix full contact order with the boundary sections. Let $\nu=(\nu_1,\nu_2)$ be such that no entry of $\nu_1$ appears in $\nu_2$ and vice versa.

A path $\gamma$ contributing to the count $N(\mu,\nu)$ has $n=\ell(\nu_1)+\ell(\nu_2)+1$ steps. The steps that $\gamma$ takes on the left edge of the rectangle of size $\ell(\nu_i)\cdot c$ must be of sizes that appear in $\nu_1$. It is possible for $\gamma$ to take all steps on the left edge, thus using up $\ell(\nu_1)$ steps. If $\gamma$ does not use up all these steps, but leaves $b$ steps of some sizes $\nu_{1i}$ to be filled, then there must be vertical down steps of these sizes somewhere in an interior column of the rectangle. This is true as these steps must appear while performing the recursion, and the only way to produce such steps is via completing a parallelogram. 
But then there are at least $b+2$ steps of $\gamma$ which are not on the left or right boundary edge of the rectangle: a step to the start of the vertical down steps, the vertical down steps, and a step away from them. As a consequence, there can be at most $n-\ell(\nu_1)-2=\ell(\nu_2)-1$ steps on the right edge of the rectangle. But then again, these steps must appear when performing the recursion. As we assumed that no entry of $\nu_1$ appears in $\nu_2$ and vice versa, we cannot produce them from the vertical down steps which produce the left steps of $\gamma$. But then they cannot be produced and thus such a path would have multiplicity $0$.

It follows that the only type of path which contributes to the count has $\ell(\nu_1)$ vertical down steps on the left edge, one step connecting the lower left with the upper right corner, and $\ell(\nu_2)$ vertical down steps on the right edge, see Figure \ref{fig-latticepathcc}.

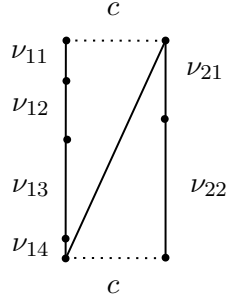
\begin{figure}
\begin{center}

\tikzset{every picture/.style={line width=0.75pt}} %set default line width to 0.75pt        

\begin{tikzpicture}[x=0.75pt,y=0.75pt,yscale=-1,xscale=1]
%uncomment if require: \path (0,784); %set diagram left start at 0, and has height of 784

%Straight Lines [id:da6546026692841378] 
\draw    (200.28,110.13) -- (200.28,219.73) ;
%Straight Lines [id:da8840151820901058] 
\draw    (200.28,110.13) -- (150,220) ;
%Straight Lines [id:da2749313547263582] 
\draw    (149.88,220.13) -- (150,110) ;
%Shape: Circle [id:dp6182703997731491] 
\draw  [color={rgb, 255:red, 0; green, 0; blue, 0 }  ,draw opacity=1 ][fill={rgb, 255:red, 0; green, 0; blue, 0 }  ,fill opacity=1 ] (148.8,110.13) .. controls (148.8,109.31) and (149.47,108.65) .. (150.28,108.66) .. controls (151.09,108.66) and (151.75,109.32) .. (151.75,110.13) .. controls (151.75,110.95) and (151.09,111.6) .. (150.27,111.6) .. controls (149.46,111.6) and (148.8,110.94) .. (148.8,110.13) -- cycle ;
%Shape: Circle [id:dp1811298005269577] 
\draw  [color={rgb, 255:red, 0; green, 0; blue, 0 }  ,draw opacity=1 ][fill={rgb, 255:red, 0; green, 0; blue, 0 }  ,fill opacity=1 ] (148.8,130.53) .. controls (148.8,129.71) and (149.47,129.05) .. (150.28,129.06) .. controls (151.09,129.06) and (151.75,129.72) .. (151.75,130.53) .. controls (151.75,131.35) and (151.09,132) .. (150.27,132) .. controls (149.46,132) and (148.8,131.34) .. (148.8,130.53) -- cycle ;
%Shape: Circle [id:dp05181008589066194] 
\draw  [color={rgb, 255:red, 0; green, 0; blue, 0 }  ,draw opacity=1 ][fill={rgb, 255:red, 0; green, 0; blue, 0 }  ,fill opacity=1 ] (149.2,160.13) .. controls (149.2,159.31) and (149.87,158.65) .. (150.68,158.66) .. controls (151.49,158.66) and (152.15,159.32) .. (152.15,160.13) .. controls (152.15,160.95) and (151.49,161.6) .. (150.67,161.6) .. controls (149.86,161.6) and (149.2,160.94) .. (149.2,160.13) -- cycle ;
%Shape: Circle [id:dp3491706182779788] 
\draw  [color={rgb, 255:red, 0; green, 0; blue, 0 }  ,draw opacity=1 ][fill={rgb, 255:red, 0; green, 0; blue, 0 }  ,fill opacity=1 ] (148.4,210.13) .. controls (148.4,209.31) and (149.07,208.65) .. (149.88,208.66) .. controls (150.69,208.66) and (151.35,209.32) .. (151.35,210.13) .. controls (151.35,210.95) and (150.69,211.6) .. (149.87,211.6) .. controls (149.06,211.6) and (148.4,210.94) .. (148.4,210.13) -- cycle ;
%Shape: Circle [id:dp8012723163701344] 
\draw  [color={rgb, 255:red, 0; green, 0; blue, 0 }  ,draw opacity=1 ][fill={rgb, 255:red, 0; green, 0; blue, 0 }  ,fill opacity=1 ] (148.4,220.13) .. controls (148.4,219.31) and (149.07,218.65) .. (149.88,218.66) .. controls (150.69,218.66) and (151.35,219.32) .. (151.35,220.13) .. controls (151.35,220.95) and (150.69,221.6) .. (149.87,221.6) .. controls (149.06,221.6) and (148.4,220.94) .. (148.4,220.13) -- cycle ;
%Shape: Circle [id:dp9611171269555462] 
\draw  [color={rgb, 255:red, 0; green, 0; blue, 0 }  ,draw opacity=1 ][fill={rgb, 255:red, 0; green, 0; blue, 0 }  ,fill opacity=1 ] (198.8,110.13) .. controls (198.8,109.31) and (199.47,108.65) .. (200.28,108.66) .. controls (201.09,108.66) and (201.75,109.32) .. (201.75,110.13) .. controls (201.75,110.95) and (201.09,111.6) .. (200.27,111.6) .. controls (199.46,111.6) and (198.8,110.94) .. (198.8,110.13) -- cycle ;
%Shape: Circle [id:dp46182383414070827] 
\draw  [color={rgb, 255:red, 0; green, 0; blue, 0 }  ,draw opacity=1 ][fill={rgb, 255:red, 0; green, 0; blue, 0 }  ,fill opacity=1 ] (198.4,149.73) .. controls (198.4,148.91) and (199.07,148.25) .. (199.88,148.26) .. controls (200.69,148.26) and (201.35,148.92) .. (201.35,149.73) .. controls (201.35,150.55) and (200.69,151.2) .. (199.87,151.2) .. controls (199.06,151.2) and (198.4,150.54) .. (198.4,149.73) -- cycle ;
%Shape: Circle [id:dp2943849956386352] 
\draw  [color={rgb, 255:red, 0; green, 0; blue, 0 }  ,draw opacity=1 ][fill={rgb, 255:red, 0; green, 0; blue, 0 }  ,fill opacity=1 ] (198.8,219.73) .. controls (198.8,218.91) and (199.47,218.25) .. (200.28,218.26) .. controls (201.09,218.26) and (201.75,218.92) .. (201.75,219.73) .. controls (201.75,220.55) and (201.09,221.2) .. (200.27,221.2) .. controls (199.46,221.2) and (198.8,220.54) .. (198.8,219.73) -- cycle ;
%Straight Lines [id:da805355293771125] 
\draw  [dash pattern={on 0.84pt off 2.51pt}]  (150.28,110.13) -- (200.55,110.26) ;
%Straight Lines [id:da7099504930041928] 
\draw  [dash pattern={on 0.84pt off 2.51pt}]  (149.72,219.87) -- (200,220) ;

% Text Node
\draw (135.58,117) node   [align=left] {\begin{minipage}[lt]{19.98pt}\setlength\topsep{0pt}
$\displaystyle \nu _{11}$
\end{minipage}};
% Text Node
\draw (135.99,143.66) node   [align=left] {\begin{minipage}[lt]{19.98pt}\setlength\topsep{0pt}
$\displaystyle \nu _{12}$
\end{minipage}};
% Text Node
\draw (136,184.5) node   [align=left] {\begin{minipage}[lt]{19.98pt}\setlength\topsep{0pt}
$\displaystyle \nu _{13}$
\end{minipage}};
% Text Node
\draw (136,214.5) node   [align=left] {\begin{minipage}[lt]{19.98pt}\setlength\topsep{0pt}
$\displaystyle \nu _{14}$
\end{minipage}};
% Text Node
\draw (224,125.5) node   [align=left] {\begin{minipage}[lt]{19.98pt}\setlength\topsep{0pt}
$\displaystyle \nu _{21}$
\end{minipage}};
% Text Node
\draw (226,184.5) node   [align=left] {\begin{minipage}[lt]{19.98pt}\setlength\topsep{0pt}
$\displaystyle \nu _{22}$
\end{minipage}};
% Text Node
\draw (184,94.5) node   [align=left] {\begin{minipage}[lt]{19.98pt}\setlength\topsep{0pt}
$\displaystyle c$
\end{minipage}};
% Text Node
\draw (184,234.5) node   [align=left] {\begin{minipage}[lt]{19.98pt}\setlength\topsep{0pt}
$\displaystyle c$
\end{minipage}};

\end{tikzpicture}

\end{center}
\caption{The only type of path contributing to $N(((c)(c)),\nu)$. In the exemplary picture, we assume $\ell(\nu_1)=4$ and $\ell(\nu_2)=2$.}\label{fig-latticepathcc}
\end{figure}

Only one subdivision fits underneath such a path, we can fill up with triangles of areas $c\cdot \nu_{ij}$. By choosing the order of the vertical down steps, we obtain $\frac{\ell(\nu_1)!}{\Aut{\nu_1}}\cdot \frac{\ell(\nu_2)!}{\Aut{\nu_2}}$ such paths.

Thus, altogether the count $N((c),(c),\nu_1,\nu_2)$ equals
\begin{equation*}
N((c),(c),\nu_1,\nu_2) = \frac{\ell(\nu_1)!}{\Aut{\nu_1}}\cdot \frac{\ell(\nu_2)!}{\Aut{\nu_2}} \cdot \prod_{i=1,2}\prod_{j=1}^{\ell(\nu_i)} \nu_{ij}\cdot c^{\ell(\nu_1)+\ell(\nu_2)}
\end{equation*}
which is polynomial in the $\nu_{ij}$ if we fix $\ell(\nu_i)$ and $c$.

The assumption that  no entry of $\nu_1$ appears in $\nu_2$ and vice versa was necessary for the argument: else, we could have contributions of reducible tropical curves. Thus, the space of all entries in $\nu$ is subdivided by walls given by equalities among entries in $\nu_1$ and entries in $\nu_2$. Away from these walls, the count is polynomial in the $\nu_{ij}$.

\end{example}

With the following example, we demonstrate that the tools developed here can in principle also be used to deduce piecewise polynomial results for counts of curves of higher genus, however, the arguments are more involved. For that reason we restrict ourselves to a case study.

\begin{example}
We compute the number of tropical curves of genus $1$ in $\mathbb{P}^1\times\mathbb{P}^1$ with contact orders $2$ with the $0$-section, $(1,1)$ with the $\infty$-section, and full contact order $\nu_1=\nu_2=(n)$ with both fibers.
A lattice path for this count must have $5$ steps. The first step must be downwards of size $n$. The second step must reach the top point of the middle column of lattice points, because otherwise we could not produce two steps of size one on the upper edge of the rectangle in the recursion. The last step must be a size $n$ step on the right edge of the rectangle and the fourth step must go from the middle column to the top right corner. The third step can have any size $i$ from $1$ to $n-1$, see Figure \ref{fig-genus1}.

The multiplicity of such a path is $n^2\cdot i^2\cdot 2(n-i)$. We have to sum from $i=1$ to $n-1$, obtaining
\begin{equation*}
\sum_{i=1}^{n-1}n^2\cdot i^2\cdot 2(n-i) = 2n^3\cdot \sum_{i=1}^{n-1}i^2-2n^2\cdot \sum_{i=1}^{n-1}i^3 = 2n^3((n-1)\cdot n\cdot (2n-1)\cdot \frac{1}{6}) - 2n^2((n-1)n\cdot \frac{1}{2})^2= \frac{1}{6}(n^6-n^4)
\end{equation*}
which is, as expected, polynomial in $n$. Here, we use the well-known Bernoulli formulas for the sums of powers.

\begin{figure}
\begin{center}

\tikzset{every picture/.style={line width=0.75pt}} %set default line width to 0.75pt        

\begin{tikzpicture}[x=0.75pt,y=0.75pt,yscale=-1,xscale=1]
%uncomment if require: \path (0,784); %set diagram left start at 0, and has height of 784

%Shape: Ellipse [id:dp928645348667898] 
\draw  [color={rgb, 255:red, 0; green, 0; blue, 0 }  ,draw opacity=1 ][fill={rgb, 255:red, 0; green, 0; blue, 0 }  ,fill opacity=1 ] (378.8,460.36) .. controls (378.81,459.23) and (379.69,458.32) .. (380.77,458.32) .. controls (381.86,458.32) and (382.74,459.24) .. (382.73,460.37) .. controls (382.73,461.49) and (381.85,462.41) .. (380.77,462.4) .. controls (379.68,462.4) and (378.8,461.49) .. (378.8,460.36) -- cycle ;
%Shape: Ellipse [id:dp8416423612474542] 
\draw  [color={rgb, 255:red, 0; green, 0; blue, 0 }  ,draw opacity=1 ][fill={rgb, 255:red, 0; green, 0; blue, 0 }  ,fill opacity=1 ] (398,460.36) .. controls (398.01,459.23) and (398.89,458.32) .. (399.97,458.32) .. controls (401.06,458.32) and (401.94,459.24) .. (401.93,460.37) .. controls (401.93,461.49) and (401.05,462.41) .. (399.97,462.4) .. controls (398.88,462.4) and (398,461.49) .. (398,460.36) -- cycle ;
%Shape: Ellipse [id:dp36058859126278364] 
\draw  [color={rgb, 255:red, 0; green, 0; blue, 0 }  ,draw opacity=1 ][fill={rgb, 255:red, 0; green, 0; blue, 0 }  ,fill opacity=1 ] (378.4,540.36) .. controls (378.41,539.23) and (379.29,538.32) .. (380.37,538.32) .. controls (381.46,538.32) and (382.34,539.24) .. (382.33,540.37) .. controls (382.33,541.49) and (381.45,542.41) .. (380.37,542.4) .. controls (379.28,542.4) and (378.4,541.49) .. (378.4,540.36) -- cycle ;
%Shape: Ellipse [id:dp7232154514391688] 
\draw  [color={rgb, 255:red, 0; green, 0; blue, 0 }  ,draw opacity=1 ][fill={rgb, 255:red, 0; green, 0; blue, 0 }  ,fill opacity=1 ] (397.6,539.96) .. controls (397.61,538.83) and (398.49,537.92) .. (399.57,537.92) .. controls (400.66,537.92) and (401.54,538.84) .. (401.53,539.97) .. controls (401.53,541.09) and (400.65,542.01) .. (399.57,542) .. controls (398.48,542) and (397.6,541.09) .. (397.6,539.96) -- cycle ;
%Shape: Ellipse [id:dp6894943340292595] 
\draw  [color={rgb, 255:red, 0; green, 0; blue, 0 }  ,draw opacity=1 ][fill={rgb, 255:red, 0; green, 0; blue, 0 }  ,fill opacity=1 ] (388.4,460.36) .. controls (388.41,459.23) and (389.29,458.32) .. (390.37,458.32) .. controls (391.46,458.32) and (392.34,459.24) .. (392.33,460.37) .. controls (392.33,461.49) and (391.45,462.41) .. (390.37,462.4) .. controls (389.28,462.4) and (388.4,461.49) .. (388.4,460.36) -- cycle ;
%Shape: Ellipse [id:dp053566548082427] 
\draw  [color={rgb, 255:red, 0; green, 0; blue, 0 }  ,draw opacity=1 ][fill={rgb, 255:red, 0; green, 0; blue, 0 }  ,fill opacity=1 ] (388.4,509.56) .. controls (388.41,508.43) and (389.29,507.52) .. (390.37,507.52) .. controls (391.46,507.52) and (392.34,508.44) .. (392.33,509.57) .. controls (392.33,510.69) and (391.45,511.61) .. (390.37,511.6) .. controls (389.28,511.6) and (388.4,510.69) .. (388.4,509.56) -- cycle ;
%Straight Lines [id:da06515675933718845] 
\draw    (380.77,460.36) -- (380.37,540.36) ;
%Straight Lines [id:da7196134918394833] 
\draw    (399.97,459.96) -- (399.57,539.96) ;
%Straight Lines [id:da3837636293353466] 
\draw    (380.37,540.36) -- (390.37,460.36) ;
%Straight Lines [id:da12378654087146124] 
\draw    (390.37,509.56) -- (399.97,460.36) ;
%Straight Lines [id:da8179290203624963] 
\draw    (390.37,460.36) -- (390.37,509.56) ;

% Text Node
\draw (366.9,494.97) node   [align=left] {\begin{minipage}[lt]{9.66pt}\setlength\topsep{0pt}
$\displaystyle n$
\end{minipage}};
% Text Node
\draw (410.9,493.77) node   [align=left] {\begin{minipage}[lt]{9.66pt}\setlength\topsep{0pt}
$\displaystyle n$
\end{minipage}};
% Text Node
\draw (397.47,473.22) node   [align=left] {\begin{minipage}[lt]{9.66pt}\setlength\topsep{0pt}
$\displaystyle i$
\end{minipage}};

\end{tikzpicture}

\end{center}
\caption{The type of lattice path needed to compute the number of tropical curves of genus $1$ in $\mathbb{P}^1\times\mathbb{P}^1$ with contact orders $2$ with the $0$-section, $(1,1)$ with the $\infty$-section, and full contact order $\nu_1=\nu_2=(n)$ with both fibers.}\label{fig-genus1}
\end{figure}
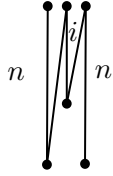

\end{example}

\section{Subfloor diagrams} \label{sec-subfloor}
In this section, we introduce a tool that enables computations by hand more easily compared to the lattice paths. It can be viewed as a combination of floor diagrams with lattice paths. Floor diagrams for plane curve counts have been studied in \cite{BBM14, BM08, FM09}. Floor diagrams in higher dimension have been studied in \cite{BM}. Floor diagrams that determine stationary tropical logarithmic Gromov--Witten invariants of Hirzebruch surfaces with descendants have been developed in  \cite{BGM, CJMR17}. Our new combination of floor diagrams with lattice paths can, besides for concrete computations, also be used to obtain structural results.

\subsection{Counting subfloor diagrams}\label{subsec-countsubfloor}

As before, fix $(\mu,\nu)$ and let $n=\ell(\mu_1)+\ell(\mu_2)+\ell(\nu_1)+\ell(\nu_2)-1$ be the number of point conditions.

Throughout this section, we let $\mu_2=(1,\ldots,1)$, i.e., we consider curves with trivial contact order along one boundary section.

\begin{definition}\label{def-subfloordiagram}
A \emph{subfloor diagram} $F$ of partial degree $\nu$ is a forest on $n$ linearly ordered vertices.
\begin{itemize}
\item Every edge $e$ of a subfloor diagram is equipped with an \emph{expansion factor} $w_e\in \mathbb{N}_{>0}$.
\item Half-edges, resp.\ ends adjacent to only one vertex, are allowed, they have to be oriented to point to the left or to the right. The multiset of expansion factors of the left ends equals $\nu_1$, the multiset of expansion factors of right ends equals $\nu_2$. 
\item Vertices are colored, black or white. The graph is bipartite. Ends have to be adjacent to a black vertex. There are $\ell(\mu_2)$ white vertices.
\end{itemize}
For a vertex $v$ of a subfloor diagram, we define its \emph{divergence} $\di(v)$ to be the sum of the expansion factors of the incoming edges minus the sum of the expansion factors of the outgoing edges.
\begin{itemize}
\item A black vertex must be of valence $2$ and divergence $0$.
\end{itemize}
\end{definition}

\begin{example}\label{ex-subfloordiagrams}
Figure \ref{fig-subfloorex1} shows three subfloor diagrams. The first is of partial degree $((1),(1))$. All expansion factors are one, which is why we do not specify them in the picture. The divergence of the first white vertex is $1$, the divergence of the second white vertex $-1$. The second is of partial degree $((1,1,1),(1,1,1))$. Again all expansion factors are one. The divergence of the first white vertex is $1$, the divergence of the second white vertex $-1$. The third is of partial degree $((2,1),(2,1))$. We specify only expansion factors which are not one. There are four edges with expansion factor $2$. The divergence of the white vertices are $2$, $1$, $-1$, $-2$.
\begin{figure}
\begin{center}
	\includegraphics[height=4cm]{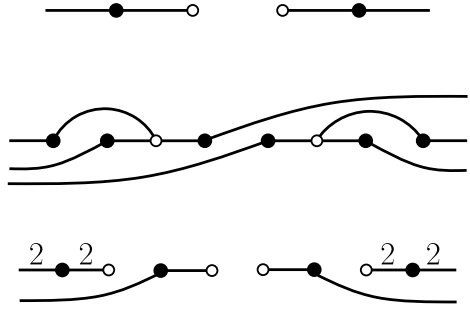}
%\input{subfloordiagramex1.pstex_t}
%\vspace{1cm}
%\input{subfloordiagramex2.pstex_t}
\caption{Three subfloor diagrams.}\label{fig-subfloorex1}
\end{center}
\end{figure}
\end{example}

\begin{construction}\label{const-mult}
Given a subfloor diagram $F$of partial degree $\nu$, the following algorithm computes its $\mu$-multiplicity:
\begin{enumerate}
\item Choose pairwise disjoint subsets $I_1,\ldots,I_r$ of white vertices such that we obtain a tree when we identify the vertices in each set $I_i$.
\item For each subset $I_i$, the divergences of the contained vertices must be decreasing with the linear order.
\item For each subset $I_i$, the divergences must sum to $0$.
\item The multiset consisting of the sizes of the $I_i$ must equal $\mu_1$.
\item For each $I_i$, we consider the Newton fan $D$ given by $\{ \binom{-\di(v)}{1}\;|\; v\in I_i \}\cup \{\binom{0}{-|I_i|}\}$. 
We let $\mult(I_i)$ to be the multiplicity of the lattice path $\delta_+$ given by the polygon dual to $D$ and the path $\delta_+$ given by $D$ (see Definition \ref{def-multpath}).
\end{enumerate}
We set 
\begin{equation*}
\mult(F)=\Big( \sum_{(I_1,\ldots,I_r)} \prod_i \mult(I_i) \Big) \cdot \prod_e w_e,
\end{equation*}
where the sum goes over all subsets $I_1,\ldots,I_r$ that satisfy the requirements above and the second product goes over all bounded edges $e$ of $F$.
\end{construction}

\begin{example}\label{ex-subfloormult}
Let $(\mu,\nu)= ((2),(1,1),(1),(1))$.
Consider the top subfloor diagram $F_1$ in Figure \ref{fig-subfloorex1}. It is of partial degree $\nu$. To produce a tree, we have to identify both white vertices, i.e., $r=1$ and $I_1$ contains both white vertices. This suits the fact that $\mu_1=(2)$.
Their divergences sum to $0$, and are decreasing (see Example \ref{ex-subfloordiagrams}). We consider the Newton fan $\{\binom{-1}{1},\binom{1}{1},\binom{0}{-2}\}$ and its dual polygon with the path $\delta_+$ (see Figure \ref{fig-pathex2}). This path has multiplicity $2$. Since all expansion factors are $1$, we obtain $\mult(F_1)=2$ in total.
\begin{figure}
\begin{center}
\includegraphics[height=1cm]{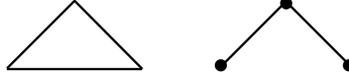}
\caption{The polygon and lattice path $\delta_+$ given by the choice of $I_1$ consisting of all white vertices for the top subfloor diagram in Figure \ref{fig-subfloorex1}.}\label{fig-pathex2}
\end{center}
\end{figure}
In the same way, we obtain, for $(\mu,\nu)=((2),(1,1),(1,1,1),(1,1,1))$,  $\mult(F_2)=2$ for the second subfloor diagram in Figure \ref{fig-subfloorex1}.

Now let $(\mu,\nu)=((4),(1,1,1,1),(2,1),(2,1))$.
Consider the subfloor diagram $F_3$ in Figure \ref{fig-subfloorex1}. It is of partial degree $\nu$. To produce a tree, again we have to identify all white vertices. This suits the fact that $\mu_1=(4)$. Their divergences are decreasing and sum up to $0$ (see Example \ref{ex-subfloordiagrams}). We consider the Newton fan $\{\binom{-2}{1}, \binom{-1}{1}, \binom{1}{1}, \binom{2}{1},\binom{0}{-4}\}$, its dual polygon and the path $\delta_+$ are depicted in Figure \ref{fig-pathex}. In Example \ref{ex-latticepath}, we computed the multiplicity of $\delta_+$ to be $100$. Altogether, the multiplicity of $F$ is $100\cdot 2\cdot 2=400$.
\end{example}

\begin{definition}
We define $N^\floor(\mu,\nu)$ to be the number of subfloor diagrams of partial degree $\nu$, counted with their $\mu$-multiplicity.
\end{definition}

\begin{example}
Let $\nu=((1),(1))$ and $\mu=((2),(1,1))$. Since $\ell(\mu_2)=2$, we have $2$ white vertices, and $n=1+1+1+2-1=4$ vertices altogether. As $\mu_1=(2)$, we know that we need to specify one set $I_1$ of white vertices of size $2$, so all white vertices have to be identified. For the subfloor diagram, they thus have to be in different connected components. As the graph has to be bipartite and ends have to be adjacent to black vertices, and since we need one left and one right end, both with expansion factor $1$, the subfloor diagram on the top of Figure \ref{fig-subfloorex1} is the only one we have to consider. In example \ref{ex-subfloormult}, we computed its $\mu$-multiplicity to be $2$. Altogether, we obtain $N^\floor(\mu,\nu)=2$ for $(\mu,\nu)=((2),(1,1),(1),(1))$.
\end{example}

\begin{example}
In Figure \ref{fig-180}, we compute $N^\floor(\mu,\nu)=180$ for $(\mu,\nu)=((2),(1,1),(2,1),(3))$. We draw three subfloor diagrams. The first picture implies the existence of a second, similar, subfloor diagram, which differs from the picture by the order of the first two points. We refrain from drawing this similar picture. As it has the same multiplicity, we just take the multiplicity of the first picture times $2$. There is only one choice of subset of white vertices to be identified, the corresponding dual lattice polygon in which we perform the lattice path algorithm is sketched below the subfloor diagram. The lattice path algorithm computes $6$---we just cut the triangle which is of normalized area $6$. Combined the weight of the bounded edges, we obtain $6\cdot 3\cdot 2=36$ for the the multiplicity of the first picture.

The computation for the second and third picture is similar. There are three more subfloor diagrams similar to those, which only differ by moving the leftmost point one or two steps to the right. We thus take the multiplicity of both pictures times $3$ each. Altogether, we obtain 
\begin{equation*}
N^\floor(\mu,\nu)=36\cdot 2+ 12\cdot 3 + 24\cdot 3 =180  \hspace{10pt} \text{ for } \hspace{10pt}  (\mu,\nu)=((2),(1,1),(2,1),(3)).
\end{equation*}

\begin{figure}
\begin{center}

\tikzset{every picture/.style={line width=0.75pt}} %set default line width to 0.75pt        

\begin{tikzpicture}[x=0.75pt,y=0.75pt,yscale=-1,xscale=1]
%uncomment if require: \path (0,784); %set diagram left start at 0, and has height of 784

%Straight Lines [id:da6546026692841378] 
\draw    (451.67,121.04) -- (478.34,121.04) ;
%Straight Lines [id:da8840151820901058] 
\draw    (523.36,121.37) -- (592,121.37) ;
%Straight Lines [id:da2749313547263582] 
\draw    (490,170) -- (500,140) ;
%Shape: Ellipse [id:dp05181008589066194] 
\draw  [color={rgb, 255:red, 0; green, 0; blue, 0 }  ,draw opacity=1 ][fill={rgb, 255:red, 255; green, 255; blue, 255 }  ,fill opacity=1 ] (521.4,121.37) .. controls (521.4,120.24) and (522.28,119.33) .. (523.37,119.33) .. controls (524.45,119.33) and (525.33,120.25) .. (525.33,121.37) .. controls (525.32,122.5) and (524.44,123.41) .. (523.36,123.41) .. controls (522.27,123.41) and (521.39,122.49) .. (521.4,121.37) -- cycle ;
%Shape: Ellipse [id:dp46182383414070827] 
\draw  [color={rgb, 255:red, 0; green, 0; blue, 0 }  ,draw opacity=1 ][fill={rgb, 255:red, 255; green, 255; blue, 255 }  ,fill opacity=1 ] (478.34,121.04) .. controls (478.34,119.91) and (479.22,119) .. (480.31,119) .. controls (481.39,119) and (482.27,119.92) .. (482.27,121.04) .. controls (482.27,122.17) and (481.39,123.08) .. (480.3,123.08) .. controls (479.21,123.08) and (478.34,122.16) .. (478.34,121.04) -- cycle ;
%Straight Lines [id:da805355293771125] 
\draw    (510,170) -- (500,140) ;
%Straight Lines [id:da461367343210896] 
\draw    (489.83,170.18) -- (510,170) ;
%Shape: Ellipse [id:dp266191127567749] 
\draw  [color={rgb, 255:red, 0; green, 0; blue, 0 }  ,draw opacity=1 ][fill={rgb, 255:red, 0; green, 0; blue, 0 }  ,fill opacity=1 ] (449.7,121.03) .. controls (449.71,119.9) and (450.59,118.99) .. (451.67,118.99) .. controls (452.76,119) and (453.64,119.91) .. (453.63,121.04) .. controls (453.63,122.17) and (452.75,123.08) .. (451.67,123.08) .. controls (450.58,123.08) and (449.7,122.16) .. (449.7,121.03) -- cycle ;
%Shape: Ellipse [id:dp25863228602514665] 
\draw  [color={rgb, 255:red, 0; green, 0; blue, 0 }  ,draw opacity=1 ][fill={rgb, 255:red, 0; green, 0; blue, 0 }  ,fill opacity=1 ] (548.73,121.71) .. controls (548.73,120.58) and (549.62,119.67) .. (550.7,119.67) .. controls (551.79,119.68) and (552.66,120.59) .. (552.66,121.72) .. controls (552.66,122.85) and (551.78,123.76) .. (550.69,123.76) .. controls (549.61,123.76) and (548.73,122.84) .. (548.73,121.71) -- cycle ;
%Curve Lines [id:da21350097058857254] 
\draw    (426.34,120.04) .. controls (439,119.7) and (455.33,135.97) .. (480.3,123.08) ;
%Shape: Ellipse [id:dp949785370140995] 
\draw  [color={rgb, 255:red, 0; green, 0; blue, 0 }  ,draw opacity=1 ][fill={rgb, 255:red, 0; green, 0; blue, 0 }  ,fill opacity=1 ] (424.37,120.03) .. controls (424.37,118.9) and (425.26,117.99) .. (426.34,117.99) .. controls (427.43,118) and (428.3,118.91) .. (428.3,120.04) .. controls (428.3,121.17) and (427.42,122.08) .. (426.33,122.08) .. controls (425.25,122.08) and (424.37,121.16) .. (424.37,120.03) -- cycle ;
%Curve Lines [id:da03028552949118768] 
\draw    (398,114.3) .. controls (425,107.3) and (422.33,107.97) .. (451.67,121.04) ;
%Curve Lines [id:da7717255322046481] 
\draw    (398.33,120.97) .. controls (411,120.63) and (416.67,120.97) .. (426.34,120.04) ;
%Straight Lines [id:da6744020187143857] 
\draw    (453.67,211.04) -- (458.67,211.04) -- (480.34,211.04) ;
%Straight Lines [id:da675798273788455] 
\draw    (525.36,211.37) -- (594,211.37) ;
%Straight Lines [id:da1777183210260681] 
\draw    (490,260) -- (500,250) ;
%Shape: Ellipse [id:dp9392984931891237] 
\draw  [color={rgb, 255:red, 0; green, 0; blue, 0 }  ,draw opacity=1 ][fill={rgb, 255:red, 255; green, 255; blue, 255 }  ,fill opacity=1 ] (523.4,211.37) .. controls (523.4,210.24) and (524.28,209.33) .. (525.37,209.33) .. controls (526.45,209.33) and (527.33,210.25) .. (527.33,211.37) .. controls (527.32,212.5) and (526.44,213.41) .. (525.36,213.41) .. controls (524.27,213.41) and (523.39,212.49) .. (523.4,211.37) -- cycle ;
%Shape: Ellipse [id:dp1928927321624967] 
\draw  [color={rgb, 255:red, 0; green, 0; blue, 0 }  ,draw opacity=1 ][fill={rgb, 255:red, 255; green, 255; blue, 255 }  ,fill opacity=1 ] (480.34,211.04) .. controls (480.34,209.91) and (481.22,209) .. (482.31,209) .. controls (483.39,209) and (484.27,209.92) .. (484.27,211.04) .. controls (484.27,212.17) and (483.39,213.08) .. (482.3,213.08) .. controls (481.21,213.08) and (480.34,212.16) .. (480.34,211.04) -- cycle ;
%Straight Lines [id:da642389544369315] 
\draw    (510,260) -- (500,250) ;
%Straight Lines [id:da9026911403086068] 
\draw    (490,260) -- (510,260) ;
%Shape: Ellipse [id:dp31921206850045536] 
\draw  [color={rgb, 255:red, 0; green, 0; blue, 0 }  ,draw opacity=1 ][fill={rgb, 255:red, 0; green, 0; blue, 0 }  ,fill opacity=1 ] (451.7,211.03) .. controls (451.71,209.9) and (452.59,208.99) .. (453.67,208.99) .. controls (454.76,209) and (455.64,209.91) .. (455.63,211.04) .. controls (455.63,212.17) and (454.75,213.08) .. (453.67,213.08) .. controls (452.58,213.08) and (451.7,212.16) .. (451.7,211.03) -- cycle ;
%Shape: Ellipse [id:dp5464808374626406] 
\draw  [color={rgb, 255:red, 0; green, 0; blue, 0 }  ,draw opacity=1 ][fill={rgb, 255:red, 0; green, 0; blue, 0 }  ,fill opacity=1 ] (550.73,211.71) .. controls (550.73,210.58) and (551.62,209.67) .. (552.7,209.67) .. controls (553.79,209.68) and (554.66,210.59) .. (554.66,211.72) .. controls (554.66,212.85) and (553.78,213.76) .. (552.69,213.76) .. controls (551.61,213.76) and (550.73,212.84) .. (550.73,211.71) -- cycle ;
%Shape: Ellipse [id:dp7543456837232857] 
\draw  [color={rgb, 255:red, 0; green, 0; blue, 0 }  ,draw opacity=1 ][fill={rgb, 255:red, 0; green, 0; blue, 0 }  ,fill opacity=1 ] (426.37,210.03) .. controls (426.37,208.9) and (427.26,207.99) .. (428.34,207.99) .. controls (429.43,208) and (430.3,208.91) .. (430.3,210.04) .. controls (430.3,211.17) and (429.42,212.08) .. (428.33,212.08) .. controls (427.25,212.08) and (426.37,211.16) .. (426.37,210.03) -- cycle ;
%Curve Lines [id:da4922169341572882] 
\draw    (400,210) .. controls (411,210.67) and (415.33,209.97) .. (428.34,210.04) ;
%Curve Lines [id:da4775674548057699] 
\draw    (400,220) .. controls (423.67,218.97) and (439.67,216.63) .. (453.67,211.04) ;
%Curve Lines [id:da8959624328824524] 
\draw    (428.34,210.04) .. controls (451.33,202.3) and (497.33,195.97) .. (523.4,211.37) ;
%Straight Lines [id:da7263296150088369] 
\draw    (453.67,311.04) -- (458.67,311.04) -- (480.34,311.04) ;
%Straight Lines [id:da12268468510841102] 
\draw    (525.36,311.37) -- (594,311.37) ;
%Straight Lines [id:da7428393181441596] 
\draw    (490,360) -- (500,340) ;
%Shape: Ellipse [id:dp394817066221969] 
\draw  [color={rgb, 255:red, 0; green, 0; blue, 0 }  ,draw opacity=1 ][fill={rgb, 255:red, 255; green, 255; blue, 255 }  ,fill opacity=1 ] (523.4,311.37) .. controls (523.4,310.24) and (524.28,309.33) .. (525.37,309.33) .. controls (526.45,309.33) and (527.33,310.25) .. (527.33,311.37) .. controls (527.32,312.5) and (526.44,313.41) .. (525.36,313.41) .. controls (524.27,313.41) and (523.39,312.49) .. (523.4,311.37) -- cycle ;
%Shape: Ellipse [id:dp4879154776667337] 
\draw  [color={rgb, 255:red, 0; green, 0; blue, 0 }  ,draw opacity=1 ][fill={rgb, 255:red, 255; green, 255; blue, 255 }  ,fill opacity=1 ] (480.34,311.04) .. controls (480.34,309.91) and (481.22,309) .. (482.31,309) .. controls (483.39,309) and (484.27,309.92) .. (484.27,311.04) .. controls (484.27,312.17) and (483.39,313.08) .. (482.3,313.08) .. controls (481.21,313.08) and (480.34,312.16) .. (480.34,311.04) -- cycle ;
%Straight Lines [id:da47510521083000634] 
\draw    (510,360) -- (500,340) ;
%Straight Lines [id:da29002210748369617] 
\draw    (490,360) -- (510,360) ;
%Shape: Ellipse [id:dp9357675374256447] 
\draw  [color={rgb, 255:red, 0; green, 0; blue, 0 }  ,draw opacity=1 ][fill={rgb, 255:red, 0; green, 0; blue, 0 }  ,fill opacity=1 ] (451.7,311.03) .. controls (451.71,309.9) and (452.59,308.99) .. (453.67,308.99) .. controls (454.76,309) and (455.64,309.91) .. (455.63,311.04) .. controls (455.63,312.17) and (454.75,313.08) .. (453.67,313.08) .. controls (452.58,313.08) and (451.7,312.16) .. (451.7,311.03) -- cycle ;
%Shape: Ellipse [id:dp8752427440946637] 
\draw  [color={rgb, 255:red, 0; green, 0; blue, 0 }  ,draw opacity=1 ][fill={rgb, 255:red, 0; green, 0; blue, 0 }  ,fill opacity=1 ] (550.73,311.71) .. controls (550.73,310.58) and (551.62,309.67) .. (552.7,309.67) .. controls (553.79,309.68) and (554.66,310.59) .. (554.66,311.72) .. controls (554.66,312.85) and (553.78,313.76) .. (552.69,313.76) .. controls (551.61,313.76) and (550.73,312.84) .. (550.73,311.71) -- cycle ;
%Shape: Ellipse [id:dp6432774176603042] 
\draw  [color={rgb, 255:red, 0; green, 0; blue, 0 }  ,draw opacity=1 ][fill={rgb, 255:red, 0; green, 0; blue, 0 }  ,fill opacity=1 ] (426.37,310.03) .. controls (426.37,308.9) and (427.26,307.99) .. (428.34,307.99) .. controls (429.43,308) and (430.3,308.91) .. (430.3,310.04) .. controls (430.3,311.17) and (429.42,312.08) .. (428.33,312.08) .. controls (427.25,312.08) and (426.37,311.16) .. (426.37,310.03) -- cycle ;
%Curve Lines [id:da2394145003297381] 
\draw    (400,310) .. controls (411,310.67) and (415.33,309.97) .. (428.34,310.04) ;
%Curve Lines [id:da6743194119553199] 
\draw    (400,320) .. controls (423.67,318.97) and (439.67,316.63) .. (453.67,311.04) ;
%Curve Lines [id:da8311125144673681] 
\draw    (428.34,310.04) .. controls (451.33,302.3) and (497.33,295.97) .. (523.4,311.37) ;

% Text Node
\draw (471.38,112.37) node   [align=left] {\begin{minipage}[lt]{14.45pt}\setlength\topsep{0pt}
$\displaystyle 2$
\end{minipage}};
% Text Node
\draw (549,110.5) node   [align=left] {\begin{minipage}[lt]{26.65pt}\setlength\topsep{0pt}
$\displaystyle 3$
\end{minipage}};
% Text Node
\draw (581,110.5) node   [align=left] {\begin{minipage}[lt]{26.65pt}\setlength\topsep{0pt}
$\displaystyle 3$
\end{minipage}};
% Text Node
\draw (425.38,103.37) node   [align=left] {\begin{minipage}[lt]{14.45pt}\setlength\topsep{0pt}
$\displaystyle 2$
\end{minipage}};
% Text Node
\draw (414.04,133.37) node   [align=left] {\begin{minipage}[lt]{14.45pt}\setlength\topsep{0pt}
$\displaystyle 1$
\end{minipage}};
% Text Node
\draw (639,159.5) node   [align=left] {\begin{minipage}[lt]{26.65pt}\setlength\topsep{0pt}
$\displaystyle 36\cdot 2$
\end{minipage}};
% Text Node
\draw (460,143) node   [align=left] {\begin{minipage}[lt]{14.45pt}\setlength\topsep{0pt}
$\displaystyle 1$
\end{minipage}};
% Text Node
\draw (473.64,193.49) node   [align=left] {\begin{minipage}[lt]{14.45pt}\setlength\topsep{0pt}
$\displaystyle 2$
\end{minipage}};
% Text Node
\draw (551,200.5) node   [align=left] {\begin{minipage}[lt]{26.65pt}\setlength\topsep{0pt}
$\displaystyle 3$
\end{minipage}};
% Text Node
\draw (589,200.5) node   [align=left] {\begin{minipage}[lt]{26.65pt}\setlength\topsep{0pt}
$\displaystyle 3$
\end{minipage}};
% Text Node
\draw (427.38,193.37) node   [align=left] {\begin{minipage}[lt]{14.45pt}\setlength\topsep{0pt}
$\displaystyle 2$
\end{minipage}};
% Text Node
\draw (430,233) node   [align=left] {\begin{minipage}[lt]{14.45pt}\setlength\topsep{0pt}
$\displaystyle 1$
\end{minipage}};
% Text Node
\draw (641,249.5) node   [align=left] {\begin{minipage}[lt]{26.65pt}\setlength\topsep{0pt}
$\displaystyle 12\cdot 3$
\end{minipage}};
% Text Node
\draw (470,223) node   [align=left] {\begin{minipage}[lt]{14.45pt}\setlength\topsep{0pt}
$\displaystyle 1$
\end{minipage}};
% Text Node
\draw (539,170.5) node   [align=left] {\begin{minipage}[lt]{26.65pt}\setlength\topsep{0pt}
$\displaystyle 6$
\end{minipage}};
% Text Node
\draw (639,329.5) node   [align=left] {\begin{minipage}[lt]{26.65pt}\setlength\topsep{0pt}
$\displaystyle 24\cdot 3$
\end{minipage}};
% Text Node
\draw (468,327) node   [align=left] {\begin{minipage}[lt]{14.45pt}\setlength\topsep{0pt}
$\displaystyle 2$
\end{minipage}};
% Text Node
\draw (551,300.5) node   [align=left] {\begin{minipage}[lt]{26.65pt}\setlength\topsep{0pt}
$\displaystyle 3$
\end{minipage}};
% Text Node
\draw (589,300.5) node   [align=left] {\begin{minipage}[lt]{26.65pt}\setlength\topsep{0pt}
$\displaystyle 3$
\end{minipage}};
% Text Node
\draw (428,333) node   [align=left] {\begin{minipage}[lt]{14.45pt}\setlength\topsep{0pt}
$\displaystyle 2$
\end{minipage}};
% Text Node
\draw (418,303) node   [align=left] {\begin{minipage}[lt]{14.45pt}\setlength\topsep{0pt}
$\displaystyle 1$
\end{minipage}};
% Text Node
\draw (470,293) node   [align=left] {\begin{minipage}[lt]{14.45pt}\setlength\topsep{0pt}
$\displaystyle 1$
\end{minipage}};
% Text Node
\draw (528,363) node   [align=left] {\begin{minipage}[lt]{14.45pt}\setlength\topsep{0pt}
$\displaystyle 4$
\end{minipage}};
% Text Node
\draw (530,267) node   [align=left] {\begin{minipage}[lt]{14.45pt}\setlength\topsep{0pt}
$\displaystyle 2$
\end{minipage}};

\end{tikzpicture}
\end{center}
\caption{The count of $N^\floor(\mu,\nu)=180$ for $(\mu,\nu)=((2),(1,1),(2,1),(3))$.}\label{fig-180}
\end{figure}
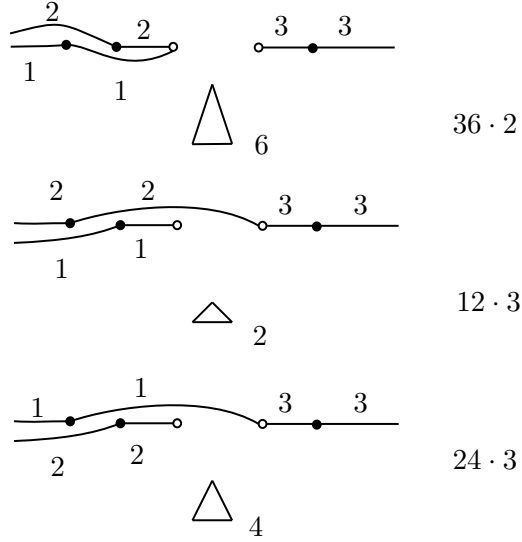
\end{example}

\subsection{From tropical curves to subfloor diagrams}\label{subsec-tropcurvessubfloor}

\begin{definition}[Horizontally stretched point conditions]
We require that the points $p_1,\ldots,p_n$ are \emph{horizontally stretched}. That is, they lie in a small strip $S= (-\epsilon, \epsilon)\times \RR$, with their distances in the $x$-direction very large.
\end{definition}

\begin{definition}
Let $(\Gamma,f)$ be a tropical stable map meeting horizontally stretched point conditions $p_i$. 
An edge of primitive direction $(1,0)$ is called an \emph{elevator}. A connected component of $\Gamma$ minus the elevators is called a \emph{floor} of $(\Gamma,f)$. 

The \emph{size} of a floor is the total weight of its vertical up- resp.\ down-ends.

Let $(\Gamma_1,f |_{\Gamma_1})$ be a floor. A \emph{subfloor} is induced by a connected component of $f(\Gamma_1) \cap S$ and consists of a metric subgraph whose unbounded edges have finite lengths and the map $f$ restricted to this subgraph. A connected component of a floor minus its subfloors is called a \emph{fork}.

\end{definition}

\begin{example}
Figure \ref{fig-exstablemap} shows a tropical stable map to $\PP^1\times\PP^1$ of degree $((2),(1,1),(1,1,1),(1,1,1))$ passing through horizontally stretched point conditions. It contains $7$ floors of which $6$ are of size zero, i.e., just marked points which are contracted. It contains $12$ elevators (on the right and left of each of the size zero floors), two subfloors $s_1$ and $s_2$, one fork $f$. The seventh floor is the union of $s_1$, $s_2$ and $f$.
\end{example}

\begin{figure}
\begin{center}
	\includegraphics[height=5cm]{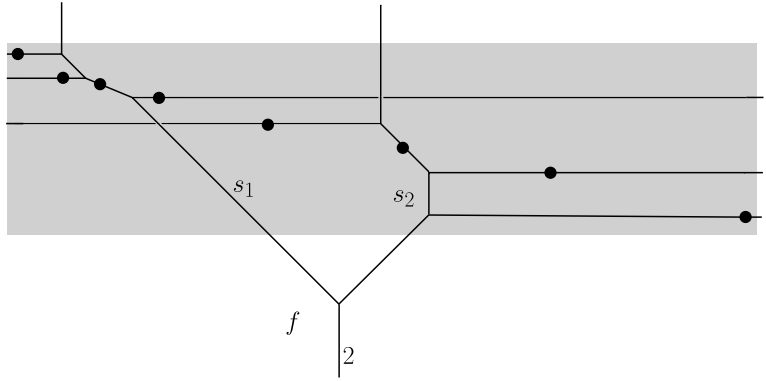}
\caption{A tropical stable map to $\PP^1\times\PP^1$ of degree $((1,1),(1,1),(2),(1,1))$ passing through horizontally stretched point conditions. The strip $S$ is marked in grey.}\label{fig-exstablemap}
\end{center}
\end{figure}

\begin{lemma}
Let $(\Gamma,f)$ be a tropical stable map to $\PP^1\times\PP^1$ of degree $(\mu,\nu)$ satisfying $n=\ell(\mu_1)+\ell(\mu_2)+\ell(\nu_1)+\ell(\nu_2)-1$ horizontally stretched point conditions. Then all elevators are mapped into the strip $S$.
\end{lemma}
\begin{proof}
This holds true since $(\Gamma,f)$ is fixed by the point conditions. In particular, the horizontal position of the elevators must be fixed by an adjacent marked point, and all point conditions lie in $S$.
\end{proof}

\begin{lemma}\label{lem-subflooronepoint}
Let $(\Gamma,f)$ be a tropical stable map to $\PP^1\times\PP^1$ of degree $(\mu,\nu)$ satisfying $n=\ell(\mu_1)+\ell(\mu_2)+\ell(\nu_1)+\ell(\nu_2)-1$ horizontally stretched point conditions. Each subfloor contains precisely one marked point.
\end{lemma}
\begin{proof}
As before, $(\Gamma,f)$ is fixed by the point conditions. Thus the vertical position of the subfloor has to be fixed. This means there has to be at least one point on the subfloor.

Given two points $p_i$ and $p_j$ on $f(\Gamma)$, consider the path in $\Gamma$ whose image under $f$ connects the two points. Then this path must involve at least an elevator or an edge which is not mapped to the strip $S$, since the points are horizontally stretched. Thus there cannot be more than one point on the subfloor.
\end{proof}

\begin{rem1}
	A floor of size zero, i.e., a marked point adjacent to two elevator edges, contains a unique subfloor, the point itself.
	A floor of size bigger zero cannot have a subfloor consisting of a single point.	
\end{rem1}

\begin{lemma}\label{lem-subfloorsize}
Let $(\Gamma,f)$ be a tropical stable map to $\PP^1\times\PP^1$ of degree $(\mu,\nu)$ satisfying $n=\ell(\mu_1)+\ell(\mu_2)+\ell(\nu_1)+\ell(\nu_2)-1$ horizontally stretched point conditions. Each subfloor of $(\Gamma,f)$ which is not just a single point has one edge leaving the top boundary of the strip $S$ and one edge leaving the bottom boundary of $S$.

Denote by $v_1$ the direction of the upper edge and by $v_2$ the direction of the lower edge, oriented  outside of $S$. Then $v_{1y}=-v_{2y}$. 
\end{lemma}

\noindent The number $v_{1y}$ from Lemma \ref{lem-subfloorsize} is called the \emph{size} of the subfloor. 

\begin{proof}
Since each subfloor is fixed by precisely one point condition by Lemma \ref{lem-subflooronepoint}, there cannot be multiple such edges, otherwise there would be a path connecting two ends of $\Gamma$ that does not meet any point conditions. The latter would lead to a $1$-dimensional movement of $(\Gamma,f)$ within its combinatorial type and still meeting the point conditions, a contradiction to the fact that $(\Gamma,f)$ is fixed by the point conditions. By the balancing condition, we must have at least one such edge.

There is a path connecting the lower and the upper edge. Since all its adjacent edges are elevators of direction $\binom{\pm w}{0}$ for some expansion factor $w$, we have $v_{1y}=-v_{2y}$.
\end{proof}

For the proof of the following lemma, our restriction $\mu_2=(1,\ldots,1)$ requiring the contact orders with one boundary section to be trivial is crucial.

\begin{lemma}\label{lem-onlyendsabovestrip}
Let $(\Gamma,f)$ be a tropical stable map to $\PP^1\times\PP^1$ of degree $(\mu,\nu)$ satisfying $n=\ell(\mu_1)+\ell(\mu_2)+\ell(\nu_1)+\ell(\nu_2)-1$ horizontally stretched point conditions. Then no vertex of $\Gamma$ is mapped above the strip $S$.
More precisely, there are only vertical up-ends above the strip.
\end{lemma}

\begin{proof}
Assume a vertex is mapped above the strip, and consider one with the biggest $y$-coordinate. If two of the adjacent edges have a direction in the open lower half space, then the $y$-coordinate of remaining adjacent edge is bigger than $1$. This edge has to be a vertical end since there is no vertex above, and thus it is of weight $1$ by our assumption $\mu_2=(1,\ldots,1)$. Thus there can be at most one edge adjacent with a direction in the open lower half space. Adjacent edges of other directions must be ends however, and there are no points outside $S$ to fix positions of edges. Thus we would have a path connecting two ends that does not meet any marked point, leading to a $1$-dimensional movement in contradiction to the fact that $(\Gamma,f)$ is fixed by the point conditions.
We conclude that no vertex is mapped anove the strip, and there are only vertical ends leaving the upper boundary of the strip.
\end{proof}

\begin{corollary} Let $(\Gamma,f)$ be a tropical stable map to $\PP^1\times\PP^1$ of degree $(\mu,\nu)$ satisfying $n=\ell(\mu_1)+\ell(\mu_2)+\ell(\nu_1)+\ell(\nu_2)-1$ horizontally stretched point conditions. Then every subfloor which is not a single point is of size $1$.
\end{corollary}
\begin{proof}
By Lemma \ref{lem-subfloorsize}, every subfloor has an upper and a lower edge leaving the strip, whose direction vectors have the same $y$-coordinate up to sign which equals the size. By Lemma  \ref{lem-onlyendsabovestrip}, the only upper edges that leave the strip are vertical ends, which by our assumption $\mu_2=(1,\ldots,1)$ are of weight $1$. Thus the size is $1$.
\end{proof}

\begin{proposition}\label{prop-tropcurvesubfloor}
Let $(\Gamma,f)$ be a tropical stable map to $\PP^1\times\PP^1$ of degree $(\mu,\nu)$ satisfying $n=\ell(\mu_1)+\ell(\mu_2)+\ell(\nu_1)+\ell(\nu_2)-1$ horizontally stretched point conditions. In the image $f(\Gamma)$, erase the forks. Shrink each subfloor to a point which is black if the subfloor is a single marked point and white otherwise. Then we obtain a subfloor diagram as in Definition \ref{def-subfloordiagram}.
\end{proposition}

\begin{proof}
Since each subfloor contains precisely one marked point by Lemma \ref{lem-subflooronepoint}, the set of vertices can be identified with the set of point conditions, which is linearly ordered by our choice of horizontally stretched conditions. As $\Gamma$ is a tree, erasing forks and shrinking subfloors to points produces a forest on this set of vertices. The edges are induced by the elevators, which come with expansion factors. The expansion factors of the left ends are given by $\nu_1$, the ones of the right by $\nu_2$ by our requirement on the degree resp.\ Newton fan of the tropical stable map. An elevator cannot connect two subfloors of size zero, because of the generic position of the points. Thus, there are no edges between two black vertices. Also, an elevator cannot connect to white vertices, i.e., two subfloors of size bigger $0$, since the horizontal position of this elevator would not be fixed then. Thus, we have a bipartite graph. Since the horizontal position of the ends has to be fixed, too, ends have to be adjacent to black vertices. Since a subfloor of size $0$ is adjacent to two elevators of the same weight, the requirement on the valency and divergence of black vertices is satisfied. As each subfloor of size bigger zero contains precisely upwards vertical end, there must be $\ell(\mu_2)$ white vertices. It follows that we obtain a subfloor diagram from $(\Gamma,f)$.
\end{proof}

\begin{theorem}\label{thm-floor=trop}
The tropical count of rational curves in $\PP^1\times\PP^1$ with contact orders $\mu,\nu$ with the toric boundary and satisfying point conditions can be computed using subfloor diagrams:
$N^\floor(\mu,\nu)=N^\trop(\mu,\nu)$.
\end{theorem}
\begin{proof}
Given a tropical stable map contributing to $N^\trop(\mu,\nu)$, we obtain a subfloor diagram from it using Proposition \ref{prop-tropcurvesubfloor}. Given a subfloor diagram $F$, it remains to show that the sum of the multicities of all tropical stable maps that produce $F$ using the construction of Proposition \ref{prop-tropcurvesubfloor} equals $\mult(F)$.
Given $F$, we can draw elevator edges adjacent to the points $p_i$ that correspond to the black points in the subfloor diagram. For each white point, there is a unique way to draw a subfloor passing through $p_i$ and adjacent to the elevators as imposed by the diagram. Thus, given a subfloor diagram $F$, we can produce a unique partial tropical stable map, namely the part whose image is contained in the strip $S$, from it such that each tropical stable map that contains the partial one yields $F$ under the construction from Proposition \ref{prop-tropcurvesubfloor}. We still need to attach suitable forks to produce a tropical stable map. It follows from Theorem \ref{thm-latticepaths} that the weighted number of forks we can attach equals the multiplicity of the lattice paths $\delta_+$ and their induced dual subdivisions as we build them in Construction \ref{const-mult}. This is true since the number of tropical stable maps does not depend on the precise location of the points (see, e.g., \cite{GM07c}), so we can assume that the edges dual to $\delta_+$ are fixed by points in Mikhalkin's position.
\end{proof}

\begin{example}
The unique partial tropical stable map we can construct from the third subfloor diagram in Figure \ref{fig-subfloorex1} is depicted in Figure \ref{fig-exstablemap2}. Its vertices contribute a factor of $4$ to the multiplicity. We can attach each of the forks depicted in Figure \ref{fig-pathex} to produce a tropical stable map of degree $((4),(1,1,1,1),(2,1),(2,1))$ from it (see Example \ref{ex-latticepath}). Thus there are altogether three tropical stable maps that produce $F$ when applying the construction of Proposition \ref{prop-tropcurvesubfloor}. The multiplicity of the first is $4\cdot 36$, of the second $4\cdot 24$ and of the third $4\cdot 40$, as the latter factors are the multiplicities of the three forks we attach. Altogether, we obtain a weighted count of $\mult(F)= 400$ tropical stable maps that produce $F$  when applying the construction of Proposition \ref{prop-tropcurvesubfloor}. The forks we can attach do not have to be equal to the forks we obtain from the lattice path algorithm, but their total multiplicity is the same.
\begin{figure}
\begin{center}
\includegraphics[width=0.9\textwidth]{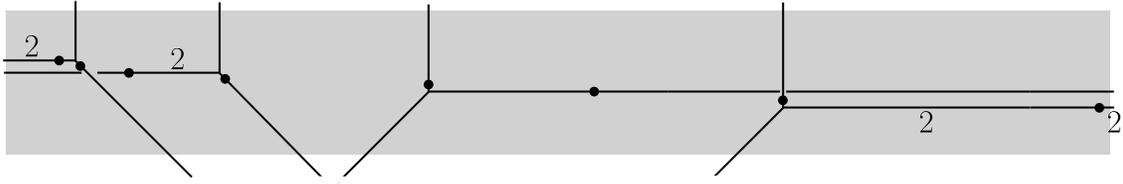}
\caption{The unique partial tropical stable map that produces the third subfloor diagram in Figure \ref{fig-subfloorex1} using the construction of Proposition \ref{prop-tropcurvesubfloor}.}\label{fig-exstablemap2}
\end{center}
\end{figure}
\end{example}

\subsection{Piecewise polynomial structure}\label{subsec-ppsubfloor}

The following theorem generalizes the result on piecewise polynomiality of the count $N(\mu,\nu)$ where $\mu=((1^a),(1^a))$ by Ardila-Brugall\'e \cite{AB17}. They showed, using floor diagrams, that this count is piecewise polynomial in the entries of the partitions $\nu$.

\begin{theorem}\label{maintheorem2}
Let $\mu_1=(2,\ldots,2,1\ldots,1)$ and $\mu_2=(1,\ldots,1)$ be partitions of a fixed number $a$. That is, we fix both tangent and trivial contact order with the $0$-section and only trivial contact order with the $\infty$-section. Let $\nu=(\nu_1,\nu_2)$ be a set of two partitions of fixed lengths of the same number whose entries we view as variables. Then the function sending $\nu$ to $N(\mu,\nu)$ is piecewise polynomial in the entries of the $\nu_i$.

\end{theorem}

\begin{proof}
We prove that $N^{\floor}(\mu,\nu)$ is piecewise polynomial. As $N^{\floor}(\mu,\nu)=N^{\trop}(\mu,\nu)$ by Theorem \ref{thm-floor=trop}, and $N^{\trop}(\mu,\nu)=N(\mu,\nu)$ by the Correspondence Theorem \ref{thm-corres}, the result follows.

Consider a subfloor diagram which has $\ell(\nu_1)$ left ends and $\ell(\nu_2)$ right ends. Assume $\mu_1$ has $t$ entries which are $2$. Each entry which is $1$ corresponds to a subset $I_i$ of size $1$ consisting of a white vertex of divergence $0$. For each of the $t$ entries which are $2$, we must have a subset $I_i$ of size $2$ consisting of two white vertices of opposite divergence. A choice of subsets $I_i$ must be such that the identifications of the white vertices in the size-two-subsets yields a tree. Thus the subfloor diagram must have $t+1$ connected components. 
The Newton polygon in which we perform a lattice path count for each of the size two subsets $I_i$ is always a triangle with a base of length $2$, as in Figure \ref{fig-180}. The lattice path recursion cuts the triangle and produces as multiplicity of the count the area of the triangle. The latter equals $2\cdot \di(v)$, where $\di(v)$ is the positive divergence of the white vertices identified via $I_i$.

Each connected component of the subfloor diagram is a tree and contains precisely one white vertex of divergence different from $0$. This is true as it must be connected to the remaining components via the identifications induced by the subsets $I_i$, and as no such identification may produce a cycle. Thus all other vertices in the component are of divergence $0$. We can compute the weights of its edges from left to right. The weights of the left ends are imposed by $\nu_1$, the weights of the right ends by $\nu_2$. The weights of all other edges are sums of entries of the $\nu_i$. Also the only nonzero divergence is a sum of entries of the $\nu_i$.
It follows that the multiplicity of the subfloor diagram is a product of sums of entries of the $\nu_i$, and thus polynomial in the $\nu_i$.

The piecewise polynomial structure appears as weights of edges have to be positive and the divergence of the left of the two identified white vertices for each size two subset $I_i$ has to be positive. Thus it depends on the inequalities that are satisfied among the various sums of entries of $\nu_i$ which subfloor diagrams appear and produce a nonzero multiplicity and which do not.
\end{proof}

\subsection{Examples}\label{subsec-subfloorex}
We end by providing examples for the piecewise polynomial counts in Theorem \ref{maintheorem2}.
\begin{example}
In Figure \ref{fig-piecewise}, we compute the polynomial $N((2),(1,1),(n_1,n_2),(n))$ using subfloor diagrams. The subfloor diagrams which appear are similar to the ones appearing in the concrete count of $N((2),(1,1),(2,1),(3))$ in Figure \ref{fig-180}. In this case, we obtain a polynomial, as $n=n_1+n_2$ always has to be bigger than the individual entries of $\nu_1$.
We obtain
\begin{displaymath}
N((2),(1,1),(n_1,n_2),(n))= 10 n^2 n_1 n_2.
\end{displaymath}

\begin{figure}
\begin{center}

\tikzset{every picture/.style={line width=0.75pt}} %set default line width to 0.75pt        

\begin{tikzpicture}[x=0.75pt,y=0.75pt,yscale=-1,xscale=1]
%uncomment if require: \path (0,784); %set diagram left start at 0, and has height of 784

%Straight Lines [id:da6546026692841378] 
\draw    (451.67,121.04) -- (478.34,121.04) ;
%Straight Lines [id:da8840151820901058] 
\draw    (523.36,121.37) -- (592,121.37) ;
%Straight Lines [id:da2749313547263582] 
\draw    (490,170) -- (500,140) ;
%Shape: Ellipse [id:dp05181008589066194] 
\draw  [color={rgb, 255:red, 0; green, 0; blue, 0 }  ,draw opacity=1 ][fill={rgb, 255:red, 255; green, 255; blue, 255 }  ,fill opacity=1 ] (521.4,121.37) .. controls (521.4,120.24) and (522.28,119.33) .. (523.37,119.33) .. controls (524.45,119.33) and (525.33,120.25) .. (525.33,121.37) .. controls (525.32,122.5) and (524.44,123.41) .. (523.36,123.41) .. controls (522.27,123.41) and (521.39,122.49) .. (521.4,121.37) -- cycle ;
%Shape: Ellipse [id:dp46182383414070827] 
\draw  [color={rgb, 255:red, 0; green, 0; blue, 0 }  ,draw opacity=1 ][fill={rgb, 255:red, 255; green, 255; blue, 255 }  ,fill opacity=1 ] (478.34,121.04) .. controls (478.34,119.91) and (479.22,119) .. (480.31,119) .. controls (481.39,119) and (482.27,119.92) .. (482.27,121.04) .. controls (482.27,122.17) and (481.39,123.08) .. (480.3,123.08) .. controls (479.21,123.08) and (478.34,122.16) .. (478.34,121.04) -- cycle ;
%Straight Lines [id:da805355293771125] 
\draw    (510,170) -- (500,140) ;
%Straight Lines [id:da461367343210896] 
\draw    (489.83,170.18) -- (510,170) ;
%Shape: Ellipse [id:dp266191127567749] 
\draw  [color={rgb, 255:red, 0; green, 0; blue, 0 }  ,draw opacity=1 ][fill={rgb, 255:red, 0; green, 0; blue, 0 }  ,fill opacity=1 ] (449.7,121.03) .. controls (449.71,119.9) and (450.59,118.99) .. (451.67,118.99) .. controls (452.76,119) and (453.64,119.91) .. (453.63,121.04) .. controls (453.63,122.17) and (452.75,123.08) .. (451.67,123.08) .. controls (450.58,123.08) and (449.7,122.16) .. (449.7,121.03) -- cycle ;
%Shape: Ellipse [id:dp25863228602514665] 
\draw  [color={rgb, 255:red, 0; green, 0; blue, 0 }  ,draw opacity=1 ][fill={rgb, 255:red, 0; green, 0; blue, 0 }  ,fill opacity=1 ] (548.73,121.71) .. controls (548.73,120.58) and (549.62,119.67) .. (550.7,119.67) .. controls (551.79,119.68) and (552.66,120.59) .. (552.66,121.72) .. controls (552.66,122.85) and (551.78,123.76) .. (550.69,123.76) .. controls (549.61,123.76) and (548.73,122.84) .. (548.73,121.71) -- cycle ;
%Curve Lines [id:da21350097058857254] 
\draw    (426.34,120.04) .. controls (439,119.7) and (455.33,135.97) .. (480.3,123.08) ;
%Shape: Ellipse [id:dp949785370140995] 
\draw  [color={rgb, 255:red, 0; green, 0; blue, 0 }  ,draw opacity=1 ][fill={rgb, 255:red, 0; green, 0; blue, 0 }  ,fill opacity=1 ] (424.37,120.03) .. controls (424.37,118.9) and (425.26,117.99) .. (426.34,117.99) .. controls (427.43,118) and (428.3,118.91) .. (428.3,120.04) .. controls (428.3,121.17) and (427.42,122.08) .. (426.33,122.08) .. controls (425.25,122.08) and (424.37,121.16) .. (424.37,120.03) -- cycle ;
%Curve Lines [id:da03028552949118768] 
\draw    (398,114.3) .. controls (425,107.3) and (422.33,107.97) .. (451.67,121.04) ;
%Curve Lines [id:da7717255322046481] 
\draw    (398.33,120.97) .. controls (411,120.63) and (416.67,120.97) .. (426.34,120.04) ;
%Straight Lines [id:da6744020187143857] 
\draw    (453.67,211.04) -- (458.67,211.04) -- (480.34,211.04) ;
%Straight Lines [id:da675798273788455] 
\draw    (525.36,211.37) -- (594,211.37) ;
%Straight Lines [id:da1777183210260681] 
\draw    (490,260) -- (500,250) ;
%Shape: Ellipse [id:dp9392984931891237] 
\draw  [color={rgb, 255:red, 0; green, 0; blue, 0 }  ,draw opacity=1 ][fill={rgb, 255:red, 255; green, 255; blue, 255 }  ,fill opacity=1 ] (523.4,211.37) .. controls (523.4,210.24) and (524.28,209.33) .. (525.37,209.33) .. controls (526.45,209.33) and (527.33,210.25) .. (527.33,211.37) .. controls (527.32,212.5) and (526.44,213.41) .. (525.36,213.41) .. controls (524.27,213.41) and (523.39,212.49) .. (523.4,211.37) -- cycle ;
%Shape: Ellipse [id:dp1928927321624967] 
\draw  [color={rgb, 255:red, 0; green, 0; blue, 0 }  ,draw opacity=1 ][fill={rgb, 255:red, 255; green, 255; blue, 255 }  ,fill opacity=1 ] (480.34,211.04) .. controls (480.34,209.91) and (481.22,209) .. (482.31,209) .. controls (483.39,209) and (484.27,209.92) .. (484.27,211.04) .. controls (484.27,212.17) and (483.39,213.08) .. (482.3,213.08) .. controls (481.21,213.08) and (480.34,212.16) .. (480.34,211.04) -- cycle ;
%Straight Lines [id:da642389544369315] 
\draw    (510,260) -- (500,250) ;
%Straight Lines [id:da9026911403086068] 
\draw    (490,260) -- (510,260) ;
%Shape: Ellipse [id:dp31921206850045536] 
\draw  [color={rgb, 255:red, 0; green, 0; blue, 0 }  ,draw opacity=1 ][fill={rgb, 255:red, 0; green, 0; blue, 0 }  ,fill opacity=1 ] (451.7,211.03) .. controls (451.71,209.9) and (452.59,208.99) .. (453.67,208.99) .. controls (454.76,209) and (455.64,209.91) .. (455.63,211.04) .. controls (455.63,212.17) and (454.75,213.08) .. (453.67,213.08) .. controls (452.58,213.08) and (451.7,212.16) .. (451.7,211.03) -- cycle ;
%Shape: Ellipse [id:dp5464808374626406] 
\draw  [color={rgb, 255:red, 0; green, 0; blue, 0 }  ,draw opacity=1 ][fill={rgb, 255:red, 0; green, 0; blue, 0 }  ,fill opacity=1 ] (550.73,211.71) .. controls (550.73,210.58) and (551.62,209.67) .. (552.7,209.67) .. controls (553.79,209.68) and (554.66,210.59) .. (554.66,211.72) .. controls (554.66,212.85) and (553.78,213.76) .. (552.69,213.76) .. controls (551.61,213.76) and (550.73,212.84) .. (550.73,211.71) -- cycle ;
%Shape: Ellipse [id:dp7543456837232857] 
\draw  [color={rgb, 255:red, 0; green, 0; blue, 0 }  ,draw opacity=1 ][fill={rgb, 255:red, 0; green, 0; blue, 0 }  ,fill opacity=1 ] (426.37,210.03) .. controls (426.37,208.9) and (427.26,207.99) .. (428.34,207.99) .. controls (429.43,208) and (430.3,208.91) .. (430.3,210.04) .. controls (430.3,211.17) and (429.42,212.08) .. (428.33,212.08) .. controls (427.25,212.08) and (426.37,211.16) .. (426.37,210.03) -- cycle ;
%Curve Lines [id:da4922169341572882] 
\draw    (400,210) .. controls (411,210.67) and (415.33,209.97) .. (428.34,210.04) ;
%Curve Lines [id:da4775674548057699] 
\draw    (400,220) .. controls (423.67,218.97) and (439.67,216.63) .. (453.67,211.04) ;
%Curve Lines [id:da8959624328824524] 
\draw    (428.34,210.04) .. controls (451.33,202.3) and (497.33,195.97) .. (523.4,211.37) ;
%Straight Lines [id:da7263296150088369] 
\draw    (453.67,311.04) -- (458.67,311.04) -- (480.34,311.04) ;
%Straight Lines [id:da12268468510841102] 
\draw    (525.36,311.37) -- (594,311.37) ;
%Straight Lines [id:da7428393181441596] 
\draw    (490,360) -- (500,340) ;
%Shape: Ellipse [id:dp394817066221969] 
\draw  [color={rgb, 255:red, 0; green, 0; blue, 0 }  ,draw opacity=1 ][fill={rgb, 255:red, 255; green, 255; blue, 255 }  ,fill opacity=1 ] (523.4,311.37) .. controls (523.4,310.24) and (524.28,309.33) .. (525.37,309.33) .. controls (526.45,309.33) and (527.33,310.25) .. (527.33,311.37) .. controls (527.32,312.5) and (526.44,313.41) .. (525.36,313.41) .. controls (524.27,313.41) and (523.39,312.49) .. (523.4,311.37) -- cycle ;
%Shape: Ellipse [id:dp4879154776667337] 
\draw  [color={rgb, 255:red, 0; green, 0; blue, 0 }  ,draw opacity=1 ][fill={rgb, 255:red, 255; green, 255; blue, 255 }  ,fill opacity=1 ] (480.34,311.04) .. controls (480.34,309.91) and (481.22,309) .. (482.31,309) .. controls (483.39,309) and (484.27,309.92) .. (484.27,311.04) .. controls (484.27,312.17) and (483.39,313.08) .. (482.3,313.08) .. controls (481.21,313.08) and (480.34,312.16) .. (480.34,311.04) -- cycle ;
%Straight Lines [id:da47510521083000634] 
\draw    (510,360) -- (500,340) ;
%Straight Lines [id:da29002210748369617] 
\draw    (490,360) -- (510,360) ;
%Shape: Ellipse [id:dp9357675374256447] 
\draw  [color={rgb, 255:red, 0; green, 0; blue, 0 }  ,draw opacity=1 ][fill={rgb, 255:red, 0; green, 0; blue, 0 }  ,fill opacity=1 ] (451.7,311.03) .. controls (451.71,309.9) and (452.59,308.99) .. (453.67,308.99) .. controls (454.76,309) and (455.64,309.91) .. (455.63,311.04) .. controls (455.63,312.17) and (454.75,313.08) .. (453.67,313.08) .. controls (452.58,313.08) and (451.7,312.16) .. (451.7,311.03) -- cycle ;
%Shape: Ellipse [id:dp8752427440946637] 
\draw  [color={rgb, 255:red, 0; green, 0; blue, 0 }  ,draw opacity=1 ][fill={rgb, 255:red, 0; green, 0; blue, 0 }  ,fill opacity=1 ] (550.73,311.71) .. controls (550.73,310.58) and (551.62,309.67) .. (552.7,309.67) .. controls (553.79,309.68) and (554.66,310.59) .. (554.66,311.72) .. controls (554.66,312.85) and (553.78,313.76) .. (552.69,313.76) .. controls (551.61,313.76) and (550.73,312.84) .. (550.73,311.71) -- cycle ;
%Shape: Ellipse [id:dp6432774176603042] 
\draw  [color={rgb, 255:red, 0; green, 0; blue, 0 }  ,draw opacity=1 ][fill={rgb, 255:red, 0; green, 0; blue, 0 }  ,fill opacity=1 ] (426.37,310.03) .. controls (426.37,308.9) and (427.26,307.99) .. (428.34,307.99) .. controls (429.43,308) and (430.3,308.91) .. (430.3,310.04) .. controls (430.3,311.17) and (429.42,312.08) .. (428.33,312.08) .. controls (427.25,312.08) and (426.37,311.16) .. (426.37,310.03) -- cycle ;
%Curve Lines [id:da2394145003297381] 
\draw    (400,310) .. controls (411,310.67) and (415.33,309.97) .. (428.34,310.04) ;
%Curve Lines [id:da6743194119553199] 
\draw    (400,320) .. controls (423.67,318.97) and (439.67,316.63) .. (453.67,311.04) ;
%Curve Lines [id:da8311125144673681] 
\draw    (428.34,310.04) .. controls (451.33,302.3) and (497.33,295.97) .. (523.4,311.37) ;

% Text Node
\draw (470,113) node   [align=left] {\begin{minipage}[lt]{14.45pt}\setlength\topsep{0pt}
$\displaystyle n_{1}{}$
\end{minipage}};
% Text Node
\draw (549,110.5) node   [align=left] {\begin{minipage}[lt]{26.65pt}\setlength\topsep{0pt}
$\displaystyle n$
\end{minipage}};
% Text Node
\draw (581,110.5) node   [align=left] {\begin{minipage}[lt]{26.65pt}\setlength\topsep{0pt}
$\displaystyle n$
\end{minipage}};
% Text Node
\draw (420,97) node   [align=left] {\begin{minipage}[lt]{14.45pt}\setlength\topsep{0pt}
$\displaystyle n_{1}{}$
\end{minipage}};
% Text Node
\draw (408.63,131.85) node   [align=left] {\begin{minipage}[lt]{14.45pt}\setlength\topsep{0pt}
$\displaystyle n_{2}{}$
\end{minipage}};
% Text Node
\draw (706,159.5) node   [align=left] {\begin{minipage}[lt]{128.25pt}\setlength\topsep{0pt}
$\displaystyle n_{1} \cdotp n_{2} \cdot n\cdotp 2( n_{1} +n_{2}) \cdot 2$
\end{minipage}};
% Text Node
\draw (460,143) node   [align=left] {\begin{minipage}[lt]{14.45pt}\setlength\topsep{0pt}
$\displaystyle n_{2}{}$
\end{minipage}};
% Text Node
\draw (551,170.5) node   [align=left] {\begin{minipage}[lt]{53.45pt}\setlength\topsep{0pt}
$\displaystyle 2( n_{1} +n_{2})$
\end{minipage}};
% Text Node
\draw (695.7,249.5) node   [align=left] {\begin{minipage}[lt]{101.05pt}\setlength\topsep{0pt}
$\displaystyle n_{1} \cdot n_{2} \cdotp n\cdotp 2n_{2} \cdot 3$
\end{minipage}};
% Text Node
\draw (530,363) node   [align=left] {\begin{minipage}[lt]{14.45pt}\setlength\topsep{0pt}
$\displaystyle 2n_{1}{}$
\end{minipage}};
% Text Node
\draw (410,197) node   [align=left] {\begin{minipage}[lt]{14.45pt}\setlength\topsep{0pt}
$\displaystyle n_{1}{}$
\end{minipage}};
% Text Node
\draw (473.64,193.49) node   [align=left] {\begin{minipage}[lt]{14.45pt}\setlength\topsep{0pt}
$\displaystyle n_{1}{}$
\end{minipage}};
% Text Node
\draw (420,333) node   [align=left] {\begin{minipage}[lt]{14.45pt}\setlength\topsep{0pt}
$\displaystyle n_{1}{}$
\end{minipage}};
% Text Node
\draw (470,327) node   [align=left] {\begin{minipage}[lt]{14.45pt}\setlength\topsep{0pt}
$\displaystyle n_{1}{}$
\end{minipage}};
% Text Node
\draw (410,233) node   [align=left] {\begin{minipage}[lt]{14.45pt}\setlength\topsep{0pt}
$\displaystyle n_{2}{}$
\end{minipage}};
% Text Node
\draw (470,223) node   [align=left] {\begin{minipage}[lt]{14.45pt}\setlength\topsep{0pt}
$\displaystyle n_{2}{}$
\end{minipage}};
% Text Node
\draw (410,297) node   [align=left] {\begin{minipage}[lt]{14.45pt}\setlength\topsep{0pt}
$\displaystyle n_{2}{}$
\end{minipage}};
% Text Node
\draw (470,287) node   [align=left] {\begin{minipage}[lt]{14.45pt}\setlength\topsep{0pt}
$\displaystyle n_{2}{}$
\end{minipage}};
% Text Node
\draw (551,199.5) node   [align=left] {\begin{minipage}[lt]{26.65pt}\setlength\topsep{0pt}
$\displaystyle n$
\end{minipage}};
% Text Node
\draw (589,200.5) node   [align=left] {\begin{minipage}[lt]{26.65pt}\setlength\topsep{0pt}
$\displaystyle n$
\end{minipage}};
% Text Node
\draw (551,299.5) node   [align=left] {\begin{minipage}[lt]{26.65pt}\setlength\topsep{0pt}
$\displaystyle n$
\end{minipage}};
% Text Node
\draw (589,299.5) node   [align=left] {\begin{minipage}[lt]{26.65pt}\setlength\topsep{0pt}
$\displaystyle n$
\end{minipage}};
% Text Node
\draw (540,263) node   [align=left] {\begin{minipage}[lt]{14.45pt}\setlength\topsep{0pt}
$\displaystyle 2n_{2}{}$
\end{minipage}};
% Text Node
\draw (690.7,319.5) node   [align=left] {\begin{minipage}[lt]{94.25pt}\setlength\topsep{0pt}
$\displaystyle n_{1} \cdot n_{2} \cdotp n\cdotp 2n_{1} \cdot 3$
\end{minipage}};

\end{tikzpicture}

\end{center}
\caption{The count of $ N((2),(1,1),(n_1,n_2),(n))= 10 n^2 n_1 n_2$ is polynomial in the entries of the partitions $\nu$.}\label{fig-piecewise}
\end{figure}
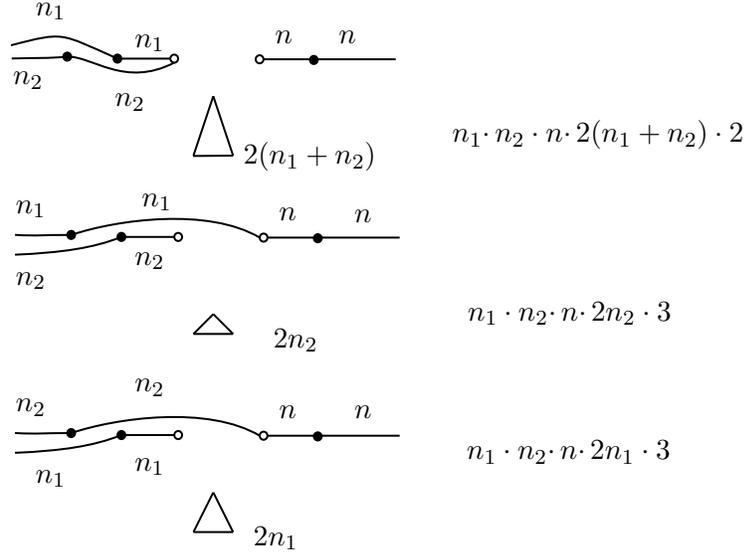

\end{example}

\begin{example}
In Figure \ref{fig-n1n2m1m2}, we demonstrate the count of $ N((2),(1,1),(n_1,n_2),(m_1,m_2))$. There exists a symmetric version for the second picture with the roles of the $n_i$ and $m_j$ exchanged (i.e., the first white vertex is $1$-valent and the second $3$-valent). For each, there is also the choice to pick $m_i$ among $i=1,2$. There are four choice of indices for the $n_k, n_l, m_i, m_j$ in the third picture. Only two of them appear together. The appearance depends on the inequalities $n_1>m_1$ or vice versa and $n_1>m_2$ or vice versa. These two inequalities imply, as $n_1+n_2=m_1+m_2$, two more inequalities of this form. We thus have four regions of polynomiality.

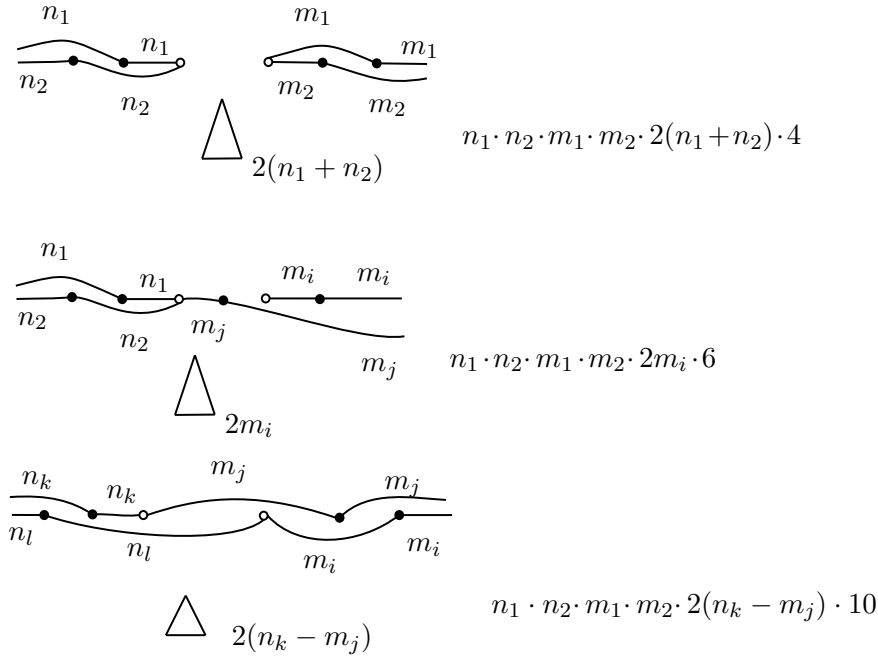
\begin{figure}
\begin{center}

\tikzset{every picture/.style={line width=0.75pt}} %set default line width to 0.75pt        

\begin{tikzpicture}[x=0.75pt,y=0.75pt,yscale=-1,xscale=1]
%uncomment if require: \path (0,784); %set diagram left start at 0, and has height of 784

%Straight Lines [id:da6546026692841378] 
\draw    (450.34,121.7) -- (477.01,121.7) ;
%Straight Lines [id:da8840151820901058] 
\draw    (525.33,121.37) -- (550.7,121.72) ;
%Straight Lines [id:da2749313547263582] 
\draw    (490,170) -- (500,140) ;
%Shape: Ellipse [id:dp05181008589066194] 
\draw  [color={rgb, 255:red, 0; green, 0; blue, 0 }  ,draw opacity=1 ][fill={rgb, 255:red, 255; green, 255; blue, 255 }  ,fill opacity=1 ] (521.4,121.37) .. controls (521.4,120.24) and (522.28,119.33) .. (523.37,119.33) .. controls (524.45,119.33) and (525.33,120.25) .. (525.33,121.37) .. controls (525.32,122.5) and (524.44,123.41) .. (523.36,123.41) .. controls (522.27,123.41) and (521.39,122.49) .. (521.4,121.37) -- cycle ;
%Shape: Ellipse [id:dp46182383414070827] 
\draw  [color={rgb, 255:red, 0; green, 0; blue, 0 }  ,draw opacity=1 ][fill={rgb, 255:red, 255; green, 255; blue, 255 }  ,fill opacity=1 ] (477.01,121.7) .. controls (477.01,120.58) and (477.89,119.66) .. (478.98,119.67) .. controls (480.06,119.67) and (480.94,120.58) .. (480.94,121.71) .. controls (480.93,122.84) and (480.05,123.75) .. (478.97,123.75) .. controls (477.88,123.75) and (477,122.83) .. (477.01,121.7) -- cycle ;
%Straight Lines [id:da805355293771125] 
\draw    (510,170) -- (500,140) ;
%Straight Lines [id:da461367343210896] 
\draw    (489.83,170.18) -- (510,170) ;
%Shape: Ellipse [id:dp266191127567749] 
\draw  [color={rgb, 255:red, 0; green, 0; blue, 0 }  ,draw opacity=1 ][fill={rgb, 255:red, 0; green, 0; blue, 0 }  ,fill opacity=1 ] (448.37,121.7) .. controls (448.37,120.57) and (449.26,119.66) .. (450.34,119.66) .. controls (451.43,119.66) and (452.3,120.58) .. (452.3,121.71) .. controls (452.3,122.83) and (451.42,123.75) .. (450.33,123.74) .. controls (449.25,123.74) and (448.37,122.83) .. (448.37,121.7) -- cycle ;
%Shape: Ellipse [id:dp25863228602514665] 
\draw  [color={rgb, 255:red, 0; green, 0; blue, 0 }  ,draw opacity=1 ][fill={rgb, 255:red, 0; green, 0; blue, 0 }  ,fill opacity=1 ] (548.73,121.71) .. controls (548.73,120.58) and (549.62,119.67) .. (550.7,119.67) .. controls (551.79,119.68) and (552.66,120.59) .. (552.66,121.72) .. controls (552.66,122.85) and (551.78,123.76) .. (550.69,123.76) .. controls (549.61,123.76) and (548.73,122.84) .. (548.73,121.71) -- cycle ;
%Curve Lines [id:da21350097058857254] 
\draw    (425,120.7) .. controls (437.67,120.37) and (454,136.63) .. (478.97,123.75) ;
%Shape: Ellipse [id:dp949785370140995] 
\draw  [color={rgb, 255:red, 0; green, 0; blue, 0 }  ,draw opacity=1 ][fill={rgb, 255:red, 0; green, 0; blue, 0 }  ,fill opacity=1 ] (423.04,120.7) .. controls (423.04,119.57) and (423.92,118.66) .. (425.01,118.66) .. controls (426.09,118.66) and (426.97,119.58) .. (426.97,120.71) .. controls (426.97,121.83) and (426.08,122.75) .. (425,122.74) .. controls (423.91,122.74) and (423.04,121.83) .. (423.04,120.7) -- cycle ;
%Curve Lines [id:da03028552949118768] 
\draw    (396.67,114.97) .. controls (423.67,107.97) and (421,108.63) .. (450.34,121.7) ;
%Curve Lines [id:da7717255322046481] 
\draw    (397,121.63) .. controls (409.67,121.3) and (415.33,121.63) .. (425,120.7) ;
%Curve Lines [id:da8959624328824524] 
\draw    (480.27,241.04) .. controls (504.33,237.47) and (567.33,263.47) .. (592,259.8) ;
%Straight Lines [id:da7263296150088369] 
\draw    (589.36,350.05) -- (594.37,350.05) -- (616.03,350.05) ;
%Straight Lines [id:da7428393181441596] 
\draw    (471.67,410.67) -- (481.67,390.67) ;
%Shape: Ellipse [id:dp394817066221969] 
\draw  [color={rgb, 255:red, 0; green, 0; blue, 0 }  ,draw opacity=1 ][fill={rgb, 255:red, 255; green, 255; blue, 255 }  ,fill opacity=1 ] (519.06,350.37) .. controls (519.06,349.24) and (519.95,348.33) .. (521.03,348.33) .. controls (522.12,348.33) and (523,349.25) .. (522.99,350.37) .. controls (522.99,351.5) and (522.11,352.41) .. (521.02,352.41) .. controls (519.94,352.41) and (519.06,351.49) .. (519.06,350.37) -- cycle ;
%Shape: Ellipse [id:dp4879154776667337] 
\draw  [color={rgb, 255:red, 0; green, 0; blue, 0 }  ,draw opacity=1 ][fill={rgb, 255:red, 255; green, 255; blue, 255 }  ,fill opacity=1 ] (458.34,350.04) .. controls (458.34,348.91) and (459.22,348) .. (460.31,348) .. controls (461.39,348) and (462.27,348.92) .. (462.27,350.04) .. controls (462.27,351.17) and (461.39,352.08) .. (460.3,352.08) .. controls (459.21,352.08) and (458.34,351.16) .. (458.34,350.04) -- cycle ;
%Straight Lines [id:da47510521083000634] 
\draw    (491.67,410.67) -- (481.67,390.67) ;
%Straight Lines [id:da29002210748369617] 
\draw    (471.67,410.67) -- (491.67,410.67) ;
%Shape: Ellipse [id:dp8752427440946637] 
\draw  [color={rgb, 255:red, 0; green, 0; blue, 0 }  ,draw opacity=1 ][fill={rgb, 255:red, 0; green, 0; blue, 0 }  ,fill opacity=1 ] (587.4,350.05) .. controls (587.4,348.92) and (588.28,348.01) .. (589.37,348.01) .. controls (590.45,348.01) and (591.33,348.93) .. (591.33,350.05) .. controls (591.33,351.18) and (590.45,352.09) .. (589.36,352.09) .. controls (588.27,352.09) and (587.4,351.17) .. (587.4,350.05) -- cycle ;
%Shape: Ellipse [id:dp6432774176603042] 
\draw  [color={rgb, 255:red, 0; green, 0; blue, 0 }  ,draw opacity=1 ][fill={rgb, 255:red, 0; green, 0; blue, 0 }  ,fill opacity=1 ] (557.33,351.37) .. controls (557.33,350.25) and (558.21,349.33) .. (559.3,349.34) .. controls (560.38,349.34) and (561.26,350.25) .. (561.26,351.38) .. controls (561.26,352.51) and (560.38,353.42) .. (559.29,353.42) .. controls (558.2,353.42) and (557.33,352.5) .. (557.33,351.37) -- cycle ;
%Curve Lines [id:da6743194119553199] 
\draw    (559.29,351.38) .. controls (573.67,336.63) and (597.67,341.63) .. (612.96,342.41) ;
%Curve Lines [id:da8311125144673681] 
\draw    (462.27,350.04) .. controls (485.27,342.31) and (513.33,336.3) .. (557.33,351.37) ;
%Shape: Ellipse [id:dp3358773492047049] 
\draw  [color={rgb, 255:red, 0; green, 0; blue, 0 }  ,draw opacity=1 ][fill={rgb, 255:red, 0; green, 0; blue, 0 }  ,fill opacity=1 ] (576.07,122.03) .. controls (576.07,120.91) and (576.95,119.99) .. (578.03,120) .. controls (579.12,120) and (580,120.91) .. (580,122.04) .. controls (579.99,123.17) and (579.11,124.08) .. (578.03,124.08) .. controls (576.94,124.08) and (576.06,123.16) .. (576.07,122.03) -- cycle ;
%Curve Lines [id:da015028253612272224] 
\draw    (523.37,119.33) .. controls (550.37,112.33) and (548.69,108.97) .. (578.03,122.04) ;
%Straight Lines [id:da17265858386156852] 
\draw    (578.03,122.04) -- (603.4,122.38) ;
%Curve Lines [id:da6154969875885251] 
\draw    (550.7,121.72) .. controls (573.33,128.02) and (581.3,133.56) .. (603.33,129.63) ;
%Straight Lines [id:da7770277195596073] 
\draw    (449.67,241.04) -- (476.34,241.04) ;
%Shape: Ellipse [id:dp13573814489333558] 
\draw  [color={rgb, 255:red, 0; green, 0; blue, 0 }  ,draw opacity=1 ][fill={rgb, 255:red, 255; green, 255; blue, 255 }  ,fill opacity=1 ] (476.34,241.04) .. controls (476.34,239.91) and (477.22,239) .. (478.31,239) .. controls (479.39,239) and (480.27,239.92) .. (480.27,241.04) .. controls (480.27,242.17) and (479.39,243.08) .. (478.3,243.08) .. controls (477.21,243.08) and (476.34,242.16) .. (476.34,241.04) -- cycle ;
%Shape: Ellipse [id:dp4271502538493044] 
\draw  [color={rgb, 255:red, 0; green, 0; blue, 0 }  ,draw opacity=1 ][fill={rgb, 255:red, 0; green, 0; blue, 0 }  ,fill opacity=1 ] (447.7,241.03) .. controls (447.71,239.9) and (448.59,238.99) .. (449.67,238.99) .. controls (450.76,239) and (451.64,239.91) .. (451.63,241.04) .. controls (451.63,242.17) and (450.75,243.08) .. (449.67,243.08) .. controls (448.58,243.08) and (447.7,242.16) .. (447.7,241.03) -- cycle ;
%Curve Lines [id:da37687798401862516] 
\draw    (424.34,240.04) .. controls (437,239.7) and (453.33,255.97) .. (478.3,243.08) ;
%Shape: Ellipse [id:dp15335335373411374] 
\draw  [color={rgb, 255:red, 0; green, 0; blue, 0 }  ,draw opacity=1 ][fill={rgb, 255:red, 0; green, 0; blue, 0 }  ,fill opacity=1 ] (422.37,240.03) .. controls (422.37,238.9) and (423.26,237.99) .. (424.34,237.99) .. controls (425.43,238) and (426.3,238.91) .. (426.3,240.04) .. controls (426.3,241.17) and (425.42,242.08) .. (424.33,242.08) .. controls (423.25,242.08) and (422.37,241.16) .. (422.37,240.03) -- cycle ;
%Curve Lines [id:da10245274728822185] 
\draw    (396,234.3) .. controls (423,227.3) and (420.33,227.97) .. (449.67,241.04) ;
%Curve Lines [id:da6131266762070717] 
\draw    (396.33,240.97) .. controls (409,240.63) and (414.67,240.97) .. (424.34,240.04) ;
%Straight Lines [id:da0746856937930378] 
\draw    (522.03,240.7) -- (590.67,240.71) ;
%Shape: Ellipse [id:dp3152626534645574] 
\draw  [color={rgb, 255:red, 0; green, 0; blue, 0 }  ,draw opacity=1 ][fill={rgb, 255:red, 255; green, 255; blue, 255 }  ,fill opacity=1 ] (520.06,240.7) .. controls (520.06,239.57) and (520.95,238.66) .. (522.03,238.66) .. controls (523.12,238.66) and (524,239.58) .. (523.99,240.71) .. controls (523.99,241.83) and (523.11,242.75) .. (522.02,242.74) .. controls (520.94,242.74) and (520.06,241.83) .. (520.06,240.7) -- cycle ;
%Shape: Ellipse [id:dp21606557710905816] 
\draw  [color={rgb, 255:red, 0; green, 0; blue, 0 }  ,draw opacity=1 ][fill={rgb, 255:red, 0; green, 0; blue, 0 }  ,fill opacity=1 ] (547.4,241.05) .. controls (547.4,239.92) and (548.28,239.01) .. (549.37,239.01) .. controls (550.45,239.01) and (551.33,239.93) .. (551.33,241.05) .. controls (551.33,242.18) and (550.45,243.09) .. (549.36,243.09) .. controls (548.27,243.09) and (547.4,242.17) .. (547.4,241.05) -- cycle ;
%Shape: Ellipse [id:dp49813794420283697] 
\draw  [color={rgb, 255:red, 0; green, 0; blue, 0 }  ,draw opacity=1 ][fill={rgb, 255:red, 0; green, 0; blue, 0 }  ,fill opacity=1 ] (498.7,241.7) .. controls (498.71,240.57) and (499.59,239.66) .. (500.67,239.66) .. controls (501.76,239.66) and (502.64,240.58) .. (502.63,241.71) .. controls (502.63,242.83) and (501.75,243.75) .. (500.67,243.74) .. controls (499.58,243.74) and (498.7,242.83) .. (498.7,241.7) -- cycle ;
%Straight Lines [id:da28073409172972175] 
\draw    (476.33,299.33) -- (486.33,269.33) ;
%Straight Lines [id:da618547765595101] 
\draw    (496.33,299.33) -- (486.33,269.33) ;
%Straight Lines [id:da3587479572257155] 
\draw    (476.16,299.51) -- (496.33,299.33) ;
%Straight Lines [id:da1503123996144814] 
\draw    (410.34,350.04) -- (394,349.97) ;
%Shape: Ellipse [id:dp928645348667898] 
\draw  [color={rgb, 255:red, 0; green, 0; blue, 0 }  ,draw opacity=1 ][fill={rgb, 255:red, 0; green, 0; blue, 0 }  ,fill opacity=1 ] (408.37,350.03) .. controls (408.37,348.9) and (409.26,347.99) .. (410.34,347.99) .. controls (411.43,348) and (412.3,348.91) .. (412.3,350.04) .. controls (412.3,351.17) and (411.42,352.08) .. (410.33,352.08) .. controls (409.25,352.08) and (408.37,351.16) .. (408.37,350.03) -- cycle ;
%Curve Lines [id:da7373951874408929] 
\draw    (393,340.97) .. controls (412.67,339.63) and (424.33,342.97) .. (434.67,349.7) ;
%Shape: Ellipse [id:dp6081210640864239] 
\draw  [color={rgb, 255:red, 0; green, 0; blue, 0 }  ,draw opacity=1 ][fill={rgb, 255:red, 0; green, 0; blue, 0 }  ,fill opacity=1 ] (432.7,349.7) .. controls (432.71,348.57) and (433.59,347.66) .. (434.67,347.66) .. controls (435.76,347.66) and (436.64,348.58) .. (436.63,349.71) .. controls (436.63,350.83) and (435.75,351.75) .. (434.67,351.74) .. controls (433.58,351.74) and (432.7,350.83) .. (432.7,349.7) -- cycle ;
%Curve Lines [id:da5335734918274457] 
\draw    (410.34,350.04) .. controls (443,361.3) and (505.68,365.61) .. (521.02,352.41) ;
%Curve Lines [id:da7690856516571655] 
\draw    (434.67,349.7) .. controls (447.34,349.37) and (448.67,350.97) .. (458.34,350.04) ;
%Curve Lines [id:da9759015284690616] 
\draw    (522.99,350.37) .. controls (538.67,367.63) and (569.33,365.63) .. (589.36,350.05) ;

% Text Node
\draw (468.67,113.67) node   [align=left] {\begin{minipage}[lt]{14.45pt}\setlength\topsep{0pt}
$\displaystyle n_{1}{}$
\end{minipage}};
% Text Node
\draw (554.14,98.39) node   [align=left] {\begin{minipage}[lt]{26.65pt}\setlength\topsep{0pt}
$\displaystyle m_{1}{}$
\end{minipage}};
% Text Node
\draw (545.67,135.5) node   [align=left] {\begin{minipage}[lt]{26.65pt}\setlength\topsep{0pt}
$\displaystyle m_{2}{}$
\end{minipage}};
% Text Node
\draw (418.67,97.67) node   [align=left] {\begin{minipage}[lt]{14.45pt}\setlength\topsep{0pt}
$\displaystyle n_{1}{}$
\end{minipage}};
% Text Node
\draw (407.29,132.52) node   [align=left] {\begin{minipage}[lt]{14.45pt}\setlength\topsep{0pt}
$\displaystyle n_{2}{}$
\end{minipage}};
% Text Node
\draw (706.67,159.17) node   [align=left] {\begin{minipage}[lt]{128.25pt}\setlength\topsep{0pt}
$\displaystyle n_{1} \cdotp n_{2} \cdot m_{1} \cdotp m_{2} \cdotp 2( n_{1} +n_{2}) \cdot 4$
\end{minipage}};
% Text Node
\draw (458.67,143.67) node   [align=left] {\begin{minipage}[lt]{14.45pt}\setlength\topsep{0pt}
$\displaystyle n_{2}{}$
\end{minipage}};
% Text Node
\draw (550.67,175.5) node   [align=left] {\begin{minipage}[lt]{53.45pt}\setlength\topsep{0pt}
$\displaystyle 2( n_{1} +n_{2})$
\end{minipage}};
% Text Node
\draw (682,271.5) node   [align=left] {\begin{minipage}[lt]{101.05pt}\setlength\topsep{0pt}
$\displaystyle n_{1} \cdot n_{2} \cdotp m_{1} \cdot m_{2} \cdotp 2m_{i} \cdot 6$
\end{minipage}};
% Text Node
\draw (543.6,412.46) node   [align=left] {\begin{minipage}[lt]{57.89pt}\setlength\topsep{0pt}
$\displaystyle 2( n_{k} -m_{j})$
\end{minipage}};
% Text Node
\draw (753.87,395.93) node   [align=left] {\begin{minipage}[lt]{176.98pt}\setlength\topsep{0pt}
$\displaystyle n_{1} \cdot n_{2} \cdotp m_{1} \cdotp m_{2} \cdotp 2( n_{k} -m_{j}) \cdot 10$
\end{minipage}};
% Text Node
\draw (608,115.17) node   [align=left] {\begin{minipage}[lt]{26.65pt}\setlength\topsep{0pt}
$\displaystyle m_{1}{}$
\end{minipage}};
% Text Node
\draw (591.67,144.5) node   [align=left] {\begin{minipage}[lt]{26.65pt}\setlength\topsep{0pt}
$\displaystyle m_{2}{}$
\end{minipage}};
% Text Node
\draw (468,233) node   [align=left] {\begin{minipage}[lt]{14.45pt}\setlength\topsep{0pt}
$\displaystyle n_{1}{}$
\end{minipage}};
% Text Node
\draw (418,217) node   [align=left] {\begin{minipage}[lt]{14.45pt}\setlength\topsep{0pt}
$\displaystyle n_{1}{}$
\end{minipage}};
% Text Node
\draw (406.63,251.85) node   [align=left] {\begin{minipage}[lt]{14.45pt}\setlength\topsep{0pt}
$\displaystyle n_{2}{}$
\end{minipage}};
% Text Node
\draw (458,263) node   [align=left] {\begin{minipage}[lt]{14.45pt}\setlength\topsep{0pt}
$\displaystyle n_{2}{}$
\end{minipage}};
% Text Node
\draw (547.67,228.83) node   [align=left] {\begin{minipage}[lt]{26.65pt}\setlength\topsep{0pt}
$\displaystyle m_{i}{}$
\end{minipage}};
% Text Node
\draw (585.67,229.83) node   [align=left] {\begin{minipage}[lt]{26.65pt}\setlength\topsep{0pt}
$\displaystyle m_{i}{}$
\end{minipage}};
% Text Node
\draw (502,257.17) node   [align=left] {\begin{minipage}[lt]{26.65pt}\setlength\topsep{0pt}
$\displaystyle m_{j}{}$
\end{minipage}};
% Text Node
\draw (588,276.5) node   [align=left] {\begin{minipage}[lt]{26.65pt}\setlength\topsep{0pt}
$\displaystyle m_{j}{}$
\end{minipage}};
% Text Node
\draw (537,304.83) node   [align=left] {\begin{minipage}[lt]{53.45pt}\setlength\topsep{0pt}
$\displaystyle 2m_{i}{}$
\end{minipage}};
% Text Node
\draw (408.67,332.67) node   [align=left] {\begin{minipage}[lt]{14.45pt}\setlength\topsep{0pt}
$\displaystyle n_{k}{}$
\end{minipage}};
% Text Node
\draw (401.96,360.85) node   [align=left] {\begin{minipage}[lt]{14.45pt}\setlength\topsep{0pt}
$\displaystyle n_{l}{}$
\end{minipage}};
% Text Node
\draw (451.33,339.67) node   [align=left] {\begin{minipage}[lt]{14.45pt}\setlength\topsep{0pt}
$\displaystyle n_{k}{}$
\end{minipage}};
% Text Node
\draw (512,328.34) node   [align=left] {\begin{minipage}[lt]{26.65pt}\setlength\topsep{0pt}
$\displaystyle m_{j}{}$
\end{minipage}};
% Text Node
\draw (599.67,335.83) node   [align=left] {\begin{minipage}[lt]{26.65pt}\setlength\topsep{0pt}
$\displaystyle m_{j}{}$
\end{minipage}};
% Text Node
\draw (461.63,368.85) node   [align=left] {\begin{minipage}[lt]{14.45pt}\setlength\topsep{0pt}
$\displaystyle n_{l}{}$
\end{minipage}};
% Text Node
\draw (558.67,374.83) node   [align=left] {\begin{minipage}[lt]{26.65pt}\setlength\topsep{0pt}
$\displaystyle m_{i}{}$
\end{minipage}};
% Text Node
\draw (611,366.17) node   [align=left] {\begin{minipage}[lt]{26.65pt}\setlength\topsep{0pt}
$\displaystyle m_{i}{}$
\end{minipage}};

\end{tikzpicture}

\end{center}

\caption{The count of $N((2),(1,1),(n_1,n_2),(m_1,m_2))$.}\label{fig-n1n2m1m2}

\end{figure}

\end{example}

With the last example, we demonstrate the piecewise polynomial behaviour also occurs for other choices of $\mu_1$.

\begin{example}
Let $\mu_1=(c)$ provide a point of full contact order with the zero section, and $\mu_2=(1,\ldots,1)$ only trivial contact order with the infinity section. Let $\nu_1=\nu_2=(x)$ provide full contact order with the zero and infinity fibers. We view $c$ as fixed and $x$ as variable, and aim to show that the count of $N(\mu,\nu)$ is polynomial in $x$ for these choices of $\mu$ and $\nu$.

We use the tool of counting subfloor diagrams for this purpose. Any fork one can attach to a subfloor diagram must have one end of weight $c$, and so it must connect all the subfloors. Thus we must have $c$ subfloors containing one white vertex each. The first such vertex has to be adjacent to an incoming edge of weight $\nu$ and to a black vertex, and the last to an outgoing edge of weight $\nu$ and a black vertex. The remaining white vertices have no adjacent vertices, see Figure \ref{fig-nunuc}. The sequence of divergences of the white vertices of this unique subfloor diagram which contributes to the count is thus $(x,0,\ldots,0,-x)$. The multiplicity of the subfloor diagram is the product of the weights of its bounded edges, which equals $x^2$, times the sum of multiplicities of the possible forks that can be attached. We compute the latter using the lattice path algorithm for the polygon depicted in Figure \ref{fig-nunuc}. In the recursion of the lattice path algorithm, we can never complete parallelograms, as we cannot make intermediate steps on the bottom edge which is required to be dual to one end of weight $c$. Thus we always cut triangles, and accordingly obtain only one suitable subdivision which is of multiplicity $x^{c-2}\cdot (x\cdot c)$. Combined with the weights of the edges, we obtain the multiplicity $c\cdot x^{c+1} $ for our subfloor diagram, and as this was the unique subfloor diagram to be counted, we have
\begin{equation*}
N((c),(1,\ldots,1), (x),(x))= c\cdot x^{c+1}
\end{equation*}
which is polynomial in $x$ for fixed $c$, as expected.

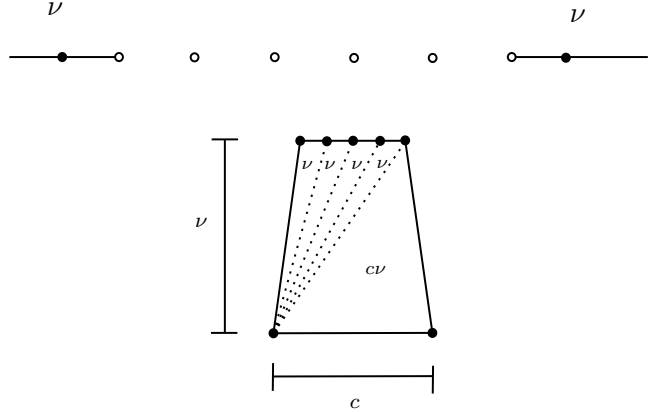
\begin{figure}
\begin{center}

\tikzset{every picture/.style={line width=0.75pt}} %set default line width to 0.75pt        

\begin{tikzpicture}[x=0.75pt,y=0.75pt,yscale=-1,xscale=1]
%uncomment if require: \path (0,784); %set diagram left start at 0, and has height of 784

%Straight Lines [id:da6546026692841378] 
\draw    (270,121.37) -- (323.34,121.37) ;
%Straight Lines [id:da8840151820901058] 
\draw    (523.36,121.37) -- (592,121.37) ;
%Straight Lines [id:da2749313547263582] 
\draw    (403.18,260.78) -- (416.68,163.63) ;
%Shape: Ellipse [id:dp6182703997731491] 
\draw  [color={rgb, 255:red, 0; green, 0; blue, 0 }  ,draw opacity=1 ][fill={rgb, 255:red, 255; green, 255; blue, 255 }  ,fill opacity=1 ] (441.89,121.82) .. controls (441.9,120.69) and (442.78,119.78) .. (443.86,119.78) .. controls (444.95,119.78) and (445.83,120.7) .. (445.82,121.83) .. controls (445.82,122.96) and (444.94,123.87) .. (443.86,123.87) .. controls (442.77,123.86) and (441.89,122.95) .. (441.89,121.82) -- cycle ;
%Shape: Ellipse [id:dp1811298005269577] 
\draw  [color={rgb, 255:red, 0; green, 0; blue, 0 }  ,draw opacity=1 ][fill={rgb, 255:red, 255; green, 255; blue, 255 }  ,fill opacity=1 ] (481.52,121.82) .. controls (481.52,120.69) and (482.4,119.78) .. (483.49,119.78) .. controls (484.57,119.78) and (485.45,120.7) .. (485.45,121.83) .. controls (485.45,122.96) and (484.56,123.87) .. (483.48,123.87) .. controls (482.39,123.86) and (481.52,122.95) .. (481.52,121.82) -- cycle ;
%Shape: Ellipse [id:dp05181008589066194] 
\draw  [color={rgb, 255:red, 0; green, 0; blue, 0 }  ,draw opacity=1 ][fill={rgb, 255:red, 255; green, 255; blue, 255 }  ,fill opacity=1 ] (521.4,121.37) .. controls (521.4,120.24) and (522.28,119.33) .. (523.37,119.33) .. controls (524.45,119.33) and (525.33,120.25) .. (525.33,121.37) .. controls (525.32,122.5) and (524.44,123.41) .. (523.36,123.41) .. controls (522.27,123.41) and (521.39,122.49) .. (521.4,121.37) -- cycle ;
%Shape: Ellipse [id:dp3491706182779788] 
\draw  [color={rgb, 255:red, 0; green, 0; blue, 0 }  ,draw opacity=1 ][fill={rgb, 255:red, 0; green, 0; blue, 0 }  ,fill opacity=1 ] (414.72,163.62) .. controls (414.72,162.49) and (415.6,161.58) .. (416.69,161.58) .. controls (417.77,161.59) and (418.65,162.5) .. (418.65,163.63) .. controls (418.65,164.76) and (417.76,165.67) .. (416.68,165.67) .. controls (415.59,165.66) and (414.72,164.75) .. (414.72,163.62) -- cycle ;
%Shape: Ellipse [id:dp8012723163701344] 
\draw  [color={rgb, 255:red, 0; green, 0; blue, 0 }  ,draw opacity=1 ][fill={rgb, 255:red, 0; green, 0; blue, 0 }  ,fill opacity=1 ] (401.22,260.78) .. controls (401.22,259.65) and (402.1,258.74) .. (403.19,258.74) .. controls (404.27,258.74) and (405.15,259.66) .. (405.15,260.79) .. controls (405.15,261.92) and (404.26,262.83) .. (403.18,262.83) .. controls (402.09,262.82) and (401.22,261.91) .. (401.22,260.78) -- cycle ;
%Shape: Ellipse [id:dp9611171269555462] 
\draw  [color={rgb, 255:red, 0; green, 0; blue, 0 }  ,draw opacity=1 ][fill={rgb, 255:red, 255; green, 255; blue, 255 }  ,fill opacity=1 ] (401.89,121.42) .. controls (401.89,120.3) and (402.77,119.38) .. (403.86,119.39) .. controls (404.94,119.39) and (405.82,120.3) .. (405.82,121.43) .. controls (405.82,122.56) and (404.94,123.47) .. (403.85,123.47) .. controls (402.77,123.47) and (401.89,122.55) .. (401.89,121.42) -- cycle ;
%Shape: Ellipse [id:dp46182383414070827] 
\draw  [color={rgb, 255:red, 0; green, 0; blue, 0 }  ,draw opacity=1 ][fill={rgb, 255:red, 255; green, 255; blue, 255 }  ,fill opacity=1 ] (323.34,121.37) .. controls (323.34,120.24) and (324.22,119.33) .. (325.31,119.33) .. controls (326.39,119.33) and (327.27,120.25) .. (327.27,121.38) .. controls (327.27,122.51) and (326.39,123.42) .. (325.3,123.42) .. controls (324.21,123.41) and (323.34,122.5) .. (323.34,121.37) -- cycle ;
%Shape: Ellipse [id:dp2943849956386352] 
\draw  [color={rgb, 255:red, 0; green, 0; blue, 0 }  ,draw opacity=1 ][fill={rgb, 255:red, 255; green, 255; blue, 255 }  ,fill opacity=1 ] (361.44,121.51) .. controls (361.44,120.38) and (362.33,119.47) .. (363.41,119.47) .. controls (364.5,119.47) and (365.38,120.39) .. (365.37,121.51) .. controls (365.37,122.64) and (364.49,123.55) .. (363.4,123.55) .. controls (362.32,123.55) and (361.44,122.63) .. (361.44,121.51) -- cycle ;
%Straight Lines [id:da805355293771125] 
\draw    (469.65,163.45) -- (483.36,260.6) ;
%Straight Lines [id:da7099504930041928] 
\draw    (416.68,163.63) -- (470.02,163.63) ;
%Shape: Ellipse [id:dp4462671729623715] 
\draw  [color={rgb, 255:red, 0; green, 0; blue, 0 }  ,draw opacity=1 ][fill={rgb, 255:red, 0; green, 0; blue, 0 }  ,fill opacity=1 ] (481.39,260.6) .. controls (481.39,259.47) and (482.28,258.56) .. (483.36,258.56) .. controls (484.45,258.57) and (485.32,259.48) .. (485.32,260.61) .. controls (485.32,261.74) and (484.44,262.65) .. (483.35,262.65) .. controls (482.27,262.64) and (481.39,261.73) .. (481.39,260.6) -- cycle ;
%Shape: Ellipse [id:dp32343079326248547] 
\draw  [color={rgb, 255:red, 0; green, 0; blue, 0 }  ,draw opacity=1 ][fill={rgb, 255:red, 0; green, 0; blue, 0 }  ,fill opacity=1 ] (467.69,163.44) .. controls (467.69,162.31) and (468.57,161.4) .. (469.66,161.4) .. controls (470.74,161.41) and (471.62,162.32) .. (471.62,163.45) .. controls (471.62,164.58) and (470.73,165.49) .. (469.65,165.49) .. controls (468.56,165.49) and (467.69,164.57) .. (467.69,163.44) -- cycle ;
%Shape: Ellipse [id:dp3753898972832801] 
\draw  [color={rgb, 255:red, 0; green, 0; blue, 0 }  ,draw opacity=1 ][fill={rgb, 255:red, 0; green, 0; blue, 0 }  ,fill opacity=1 ] (455.04,163.68) .. controls (455.04,162.55) and (455.92,161.64) .. (457.01,161.64) .. controls (458.09,161.64) and (458.97,162.56) .. (458.97,163.69) .. controls (458.97,164.82) and (458.08,165.73) .. (457,165.73) .. controls (455.91,165.72) and (455.04,164.81) .. (455.04,163.68) -- cycle ;
%Shape: Ellipse [id:dp14609868732155973] 
\draw  [color={rgb, 255:red, 0; green, 0; blue, 0 }  ,draw opacity=1 ][fill={rgb, 255:red, 0; green, 0; blue, 0 }  ,fill opacity=1 ] (441.39,163.62) .. controls (441.39,162.49) and (442.27,161.58) .. (443.36,161.58) .. controls (444.44,161.59) and (445.32,162.5) .. (445.32,163.63) .. controls (445.32,164.76) and (444.43,165.67) .. (443.35,165.67) .. controls (442.26,165.66) and (441.38,164.75) .. (441.39,163.62) -- cycle ;
%Shape: Ellipse [id:dp670435603986981] 
\draw  [color={rgb, 255:red, 0; green, 0; blue, 0 }  ,draw opacity=1 ][fill={rgb, 255:red, 0; green, 0; blue, 0 }  ,fill opacity=1 ] (428.18,163.88) .. controls (428.18,162.75) and (429.06,161.84) .. (430.15,161.84) .. controls (431.23,161.84) and (432.11,162.76) .. (432.11,163.89) .. controls (432.11,165.01) and (431.22,165.93) .. (430.14,165.92) .. controls (429.05,165.92) and (428.18,165) .. (428.18,163.88) -- cycle ;
%Straight Lines [id:da3181080857981218] 
\draw  [dash pattern={on 0.84pt off 2.51pt}]  (403.18,260.78) -- (430.14,163.88) ;
%Straight Lines [id:da9996451479866478] 
\draw  [dash pattern={on 0.84pt off 2.51pt}]  (403.18,260.78) -- (443.35,163.63) ;
%Straight Lines [id:da13884149310345628] 
\draw  [dash pattern={on 0.84pt off 2.51pt}]  (403.18,260.78) -- (457,163.68) ;
%Straight Lines [id:da05436147085274534] 
\draw  [dash pattern={on 0.84pt off 2.51pt}]  (403.18,260.78) -- (470.02,163.63) ;
%Straight Lines [id:da461367343210896] 
\draw    (403.18,260.78) -- (483.36,260.6) ;
%Straight Lines [id:da7799177481598791] 
\draw    (403.18,282.6) -- (483.36,282.42) ;
%Straight Lines [id:da1791908563281538] 
\draw    (402.85,289.88) -- (403.01,275.84) ;
%Straight Lines [id:da6522718675576377] 
\draw    (483.53,291.21) -- (483.69,277.18) ;
%Straight Lines [id:da34769611292972036] 
\draw    (378.68,260.26) -- (378.68,163.28) ;
%Straight Lines [id:da9065647200576985] 
\draw    (385.35,162.93) -- (372.01,162.93) ;
%Straight Lines [id:da7521311314968357] 
\draw    (385.01,260.6) -- (371.68,260.6) ;
%Shape: Ellipse [id:dp266191127567749] 
\draw  [color={rgb, 255:red, 0; green, 0; blue, 0 }  ,draw opacity=1 ][fill={rgb, 255:red, 0; green, 0; blue, 0 }  ,fill opacity=1 ] (294.7,121.37) .. controls (294.71,120.24) and (295.59,119.33) .. (296.67,119.33) .. controls (297.76,119.33) and (298.64,120.25) .. (298.63,121.37) .. controls (298.63,122.5) and (297.75,123.41) .. (296.67,123.41) .. controls (295.58,123.41) and (294.7,122.49) .. (294.7,121.37) -- cycle ;
%Shape: Ellipse [id:dp25863228602514665] 
\draw  [color={rgb, 255:red, 0; green, 0; blue, 0 }  ,draw opacity=1 ][fill={rgb, 255:red, 0; green, 0; blue, 0 }  ,fill opacity=1 ] (548.73,121.71) .. controls (548.73,120.58) and (549.62,119.67) .. (550.7,119.67) .. controls (551.79,119.68) and (552.66,120.59) .. (552.66,121.72) .. controls (552.66,122.85) and (551.78,123.76) .. (550.69,123.76) .. controls (549.61,123.76) and (548.73,122.84) .. (548.73,121.71) -- cycle ;

% Text Node
\draw (298.71,97.37) node   [align=left] {\begin{minipage}[lt]{14.45pt}\setlength\topsep{0pt}
$\displaystyle \nu $
\end{minipage}};
% Text Node
\draw (570.44,100.59) node   [align=left] {\begin{minipage}[lt]{26.65pt}\setlength\topsep{0pt}
$\displaystyle \nu $
\end{minipage}};
% Text Node
\draw (424.31,175.93) node  [font=\tiny] [align=left] {\begin{minipage}[lt]{10.37pt}\setlength\topsep{0pt}
$\displaystyle \nu $
\end{minipage}};
% Text Node
\draw (435.73,175.93) node  [font=\tiny] [align=left] {\begin{minipage}[lt]{10.37pt}\setlength\topsep{0pt}
$\displaystyle \nu $
\end{minipage}};
% Text Node
\draw (449.38,175.99) node  [font=\tiny] [align=left] {\begin{minipage}[lt]{10.37pt}\setlength\topsep{0pt}
$\displaystyle \nu $
\end{minipage}};
% Text Node
\draw (462.03,175.75) node  [font=\tiny] [align=left] {\begin{minipage}[lt]{10.37pt}\setlength\topsep{0pt}
$\displaystyle \nu $
\end{minipage}};
% Text Node
\draw (456.28,229.24) node  [font=\tiny] [align=left] {\begin{minipage}[lt]{10.37pt}\setlength\topsep{0pt}
$\displaystyle c\nu $
\end{minipage}};
% Text Node
\draw (448.69,296.28) node  [font=\scriptsize] [align=left] {\begin{minipage}[lt]{10.37pt}\setlength\topsep{0pt}
$\displaystyle c$
\end{minipage}};
% Text Node
\draw (370.68,204.84) node  [font=\scriptsize] [align=left] {\begin{minipage}[lt]{10.37pt}\setlength\topsep{0pt}
$\displaystyle \nu $
\end{minipage}};

\end{tikzpicture}

\end{center}

\caption{The only subfloor diagram contributing to the count of $N((c),(1,\ldots,1), (x),(x))$. Below, the lattice path that determines the possible fork to be attached to the subfloor diagram, together with the unique subdivision obtained via the recursion. In each triangle appearing in the subdivision, we mark its area to keep track of the multiplicity of the path.}\label{fig-nunuc}

\end{figure}
\end{example}

\bibliographystyle{siam} 
\bibliography{DescendantFloorCalculus}

\end{document}